\documentclass[12pt]{amsart}
\usepackage[margin= 0.75 in, top = 1 in, bottom = 1 in]{geometry}
\usepackage{appendix}
\usepackage[msc-links, initials]{amsrefs}
\makeatletter
\def\@tocline#1#2#3#4#5#6#7{\relax
  \ifnum #1>\c@tocdepth % then omit
  \else
    \par \addpenalty\@secpenalty\addvspace{#2}%
    \begingroup \hyphenpenalty\@M
    \@ifempty{#4}{%
      \@tempdima\csname r@tocindent\number#1\endcsname\relax
    }{%
      \@tempdima#4\relax
    }%
    \parindent\z@ \leftskip#3\relax \advance\leftskip\@tempdima\relax
    \rightskip\@pnumwidth plus4em \parfillskip-\@pnumwidth
    #5\leavevmode\hskip-\@tempdima
      \ifcase #1
       \or\or \hskip 1em \or \hskip 2em \else \hskip 3em \fi%
      #6\nobreak\relax
    \hfill\hbox to\@pnumwidth{\@tocpagenum{#7}}\par% <---- \dotfill -> \hfill
    \nobreak
    \endgroup
  \fi}
\makeatother
\usepackage{xcolor, framed}
\definecolor{shadecolor}{rgb}{1,0.9,0.9}

\usepackage{tabularx, graphicx, wrapfig, tikz, subcaption,pgfplots}
\usepackage{hyperref}
\hypersetup{
    colorlinks = true,linkcolor=magenta,citecolor=blue,
    linkbordercolor = {white}}

\definecolor{shadecolor}{cmyk}{0.05, 0.22,0,0}

\newenvironment{pink}{\begin{shaded*}}{ \end{shaded*} }

\usepackage{mathpazo,amsmath, amsfonts, amsthm, amssymb, commath}
\pgfplotsset{compat=1.18}

\usepackage[shortlabels]{enumitem}
\usepackage{comment}

\newcommand{\R}{{\mathbb R}}
\newcommand{\N}{{\mathbb N}}
\newcommand{\C}{{\mathbb C}}

\newcommand{\Z}{{\mathbb Z}}
\newcommand{\I}{{\mathbb I}}

\newcommand{\Oo}{\mathcal{O}}

\newcommand{\mb}[1]{\mathbf{#1}}
\newcommand{\mc}[1]{\mathcal{#1}}
\newcommand{\mf}[1]{\mathfrak{#1}}

\newcommand{\rhy}   {\textnormal{RHP}-${\mathbf Y_n}$}
\newcommand{\rhs}   {\textnormal{RHP}-${\mathbf S_n}$}
\newcommand{\rht}   {\textnormal{RHP}-${\mathbf T_n}$}
\newcommand{\rhn}   {\textnormal{RHP}-${\mathbf N_n}$}

\newcommand{\rhr}   {\textnormal{RHP}-${\mathbf R_n}$}
\newcommand{\rhp}   {\textnormal{RHP}-${\mathbf P_{e, n}}$}
\newcommand{\p}[1]{\textnormal{P}_{\uppercase\expandafter{\romannumeral#1\relax}}}
\newcommand{\psym}[1]{\textnormal{P}^{\textnormal{sym}}_{\uppercase\expandafter{\romannumeral#1\relax}}}

\newcommand{\ii}{{\mathrm i}}
\newcommand{\dd}{{\mathrm d}}
\newcommand{\ee}{{\mathrm e}}
\newcommand{\Ga}{\Gamma}

\newcommand{\RS}{\mathfrak S}
\newcommand{\RSab}{\RS_{\ualpha, \ubeta}}
\newcommand{\z}{{\boldsymbol z}}
\newcommand{\x}{{\boldsymbol x}}
\renewcommand{\a}{{\boldsymbol a}}
\renewcommand{\b}{{\boldsymbol b}}
\newcommand{\ualpha}{{\boldsymbol \alpha}}
\newcommand{\ubeta}{{\boldsymbol \beta}}
\newcommand{\hd}{\mathcal{H}}
\newcommand{\A}{{\mathcal A}}
\newcommand{\B}{{\mathrm B}}
\newcommand{\jac}{{\mathrm{Jac}}}
\def\XXint#1#2#3{{\setbox0=\hbox{$#1{#2#3}{\int}$}
    \vcenter{\hbox{$#2#3$}}\kern-.5\wd0}}

\renewcommand{\deg}{\text{deg}}
\renewcommand{\arg}{\text{arg}}

\renewcommand{\epsilon}{\varepsilon}
\renewcommand{\subset}{\subseteq}

\newcommand{\qandq}{\quad \text{and} \quad}
\newcommand{\qasq}{\quad \text{as} \quad }
\newcommand{\qforq}{\quad \text{for} \quad }
\newcommand{\qorq}{\quad \text{or} \quad }

\newtheorem{lemma}{Lemma}[section]
\newtheorem{theorem}[lemma]{Theorem}

\newtheorem{proposition}[lemma]{Proposition}
\newtheorem{corollary}[lemma]{Corollary}
\theoremstyle{definition}
\newtheorem{definition}[lemma]{Definition}
\newtheorem{example}{Example}
\AtBeginEnvironment{example}{%
  \pushQED{\qed}%
}
\AtEndEnvironment{example}{\popQED\endexample}
\newtheorem{remark}[lemma]{Remark}
\AtBeginEnvironment{remark}{%
  \pushQED{\qed}%
}
\AtEndEnvironment{remark}{\popQED\endexample}
\newtheorem{notation}{Notation}
\numberwithin{equation}{section}

\title[Asymptotics of polynomials orthogonal with respect to a generalized Freud weight]{\bf Asymptotics of polynomials orthogonal with respect to a generalized Freud weight with application to special function solutions of Painlev\'e-IV }
\author{Ahmad Barhoumi}
\email{ahmadba@kth.se}
\address{Department of Mathematics, Royal Institute of Technology (KTH), Lindstedtsv\"agen 25, 10044, Stockholm, Sweden}
\thanks{The author was supported by the European Research Council (ERC) Grant Agreement No. 101002013}
\date{\today}

\begin{document}
\setcounter{tocdepth}{1}

\begin{abstract}
    We obtain asymptotics of polynomials satisfying the orthogonality relations
    \[
    \int_{\R} z^k P_n(z; t , N) \ee^{-N \left(\frac{1}{4}z^4 + \frac{t}{2}z^2 \right)} \dd z = 0 \qforq k = 0, 1, ..., n-1,
    \]
    where the complex parameter $t$ is in the so-called two-cut region. As an application, we deduce asymptotic formulas for certain families of solutions of Painlev\'e-IV which are indexed by a non-negative integer and can be written in terms of parabolic cylinder functions. The proofs are based on the characterization of orthogonal polynomials in terms of a Riemann–Hilbert problem and the Deift-Zhou non-linear steepest descent method.
\end{abstract}
\maketitle 
\tableofcontents
\section{Introduction}
\label{sec:intro}
Let $\{P_n(z;t, N)\}_{n = 0}^\infty$ be a family of monic polynomials satisfying the orthogonality relations
\begin{equation}
\label{eq:ortho}
    \int_{\R} z^k P_n(z; t , N) \ee^{-NV(z; t)} \dd z = 0 \qforq k = 0, 1, ..., n-1,
\end{equation}
where $N \geq 0$, $t \in \C$, and
\begin{equation}
    V(z;t) := \dfrac{1}{4}z^4 + \dfrac{t}{2}z^2.
    \label{potential}
\end{equation} 
Polynomials $P_n(z; t, N)$ satisfy a three-term recurrence relation which, since \eqref{potential} is even, takes the form 
\begin{equation}
\label{eq:three-term}
    zP_n(z; t, N) = P_{n+1}(z; t, N) + \gamma_n^2(t,N) P_{n-1}(z;t, N).
\end{equation}
In this work, we obtain large-degree asymptotics of $P_n(z; t, N)$ and $\gamma_n(t, N)$ in the double-scaling regime where $n, N \to \infty$ such that $|n - N|$ is bounded and $t$ is in the so-called two-cut region, which we now explain. It is well understood that the large $n$ behavior of $P_n(z;t,N), \gamma_n^2(t, N)$ depends on the choice of the complex parameter $t$, and particularly on the geometry of the attracting set for the zeros of $P_n(z;t, N)$. More precisely, it follows from \cite{MR3306308}*{Theorem 2.3} and \cite{MR922628}*{Theorem 3} that the normalized counting measures of zeros of $P_n(z; t, N)$ weakly converge to a measure $\mu_t$ supported on a finite union of at most three analytic arcs. The regions labeled $\Oo_{1}, \Oo_{2},$ and $\Oo_{3}$ in Figure \ref{fig:phase} correspond to parameters $t$ where the measure $\mu_t$ is supported on one, two, or three analytic arcs, respectively. Remarkably, Figure \ref{fig:phase} appeared in the physical literature in \cite{MR1083917}*{Fig. 5}. It was later obtained computationally (in different coordinates) in \cite{BT}, and a rigorous description followed in \cite{MR3589917}. In both works, the authors considered a more general contour of integration than in \eqref{eq:ortho}. An explicit description of the boundaries of $\Oo_j$ tailored to our setting was given in \cite{BGM}. A precise definition of $\Oo_j$ is given in Section \ref{sec:asymptotic-analysis-prelim}, where we also state a specialization of \cite{MR3306308}*{Theorem 2.3} for the reader's convenience. 
% Our contribution is to obtain asymptotic formulas for $P_n(z;t, N), \gamma_n^2(t,N)$, and related quantities when $t \in \Oo_2$ and $n, N \to \infty$ in the regime described above, see Theorem \ref{thm:asymptotics-poly} and Theorem \ref{thm:asymptotics-gamma-h}.
\begin{figure}[t]
    \centering
    \includegraphics[width = 0.5 \textwidth]{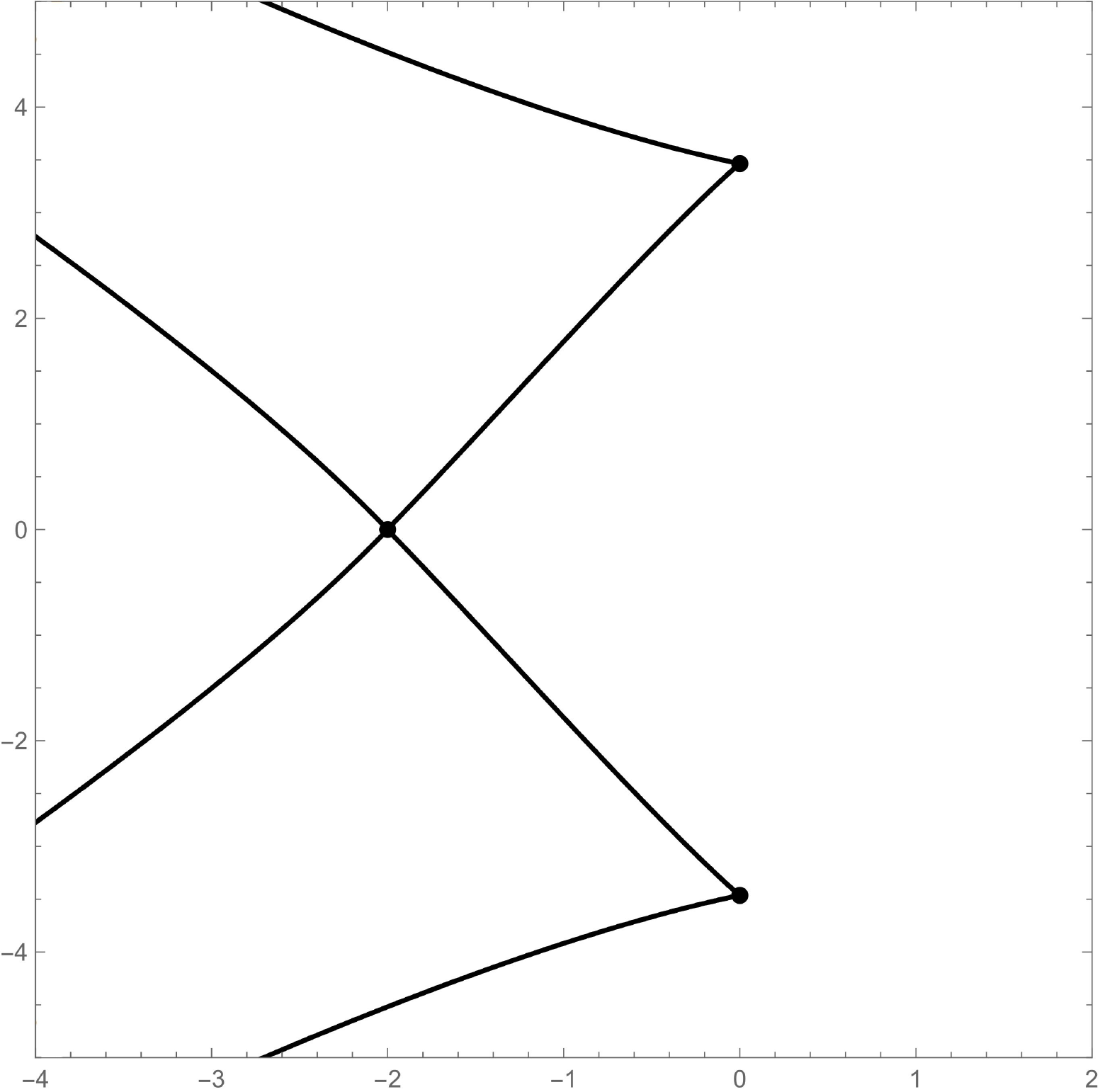}
    \put(-50,125){$\Oo_1$}
    \put(-225,125){$\Oo_2$}
    \put(-110,206){$\Oo_3$}
    \put(-85,218){$\ii \sqrt{12}$}
    \put(-85,35){$-\ii \sqrt{12}$}
    \put(-110,47){$\Oo_3$}
    \caption{The boundary of the open sets $\Oo_1, \Oo_2, \Oo_3$.} 
    \label{fig:phase}
\end{figure}

The phase diagram in Figure \ref{fig:phase} contains three points that are of particular interest in the mathematical physics literature because of the presence of ``critical phenomena." At the level of polynomials $P_n(z; t, N)$, this manifests itself in non-generic asymptotic formulas involving Painlev\'e transcendents. Asymptotic formulas in the double-scaling regime $t \to -2$ were obtained in \cite{MR1949138} (see also \cite{MR2559966}) and involve particular solutions of Painlev\'e-II. An analogous result, but involving solutions to Painlev\'e-I, when $t \to  \pm \ii \sqrt{12}$ was obtained in \cite{DK}. Similar analysis was done in \cite{BT} for more general choices of contours. As for the rest of the phase diagram, asymptotics of $\gamma_n^2(t, N)$ were obtained for $t > -2$ in \cite{BI}, for $t<-2$ in \cite{MR1715324}, and for $t \in \Oo_1$ in \cite{BGM}. Leading term asymptotics for the recurrence coefficients were stated for $t \in \Oo_2 \cup \Oo_3$ in \cite{MR3589917} (for a different set of parameters and more general contours) under the assumption that the $g$-function (see Section \ref{sec:g-fun}) satisfies certain inequalities and that certain theta functions are non-vanishing (cf. \cite{MR3589917}*{Theorem 2.19}). The asymptotic analysis for a related family of orthogonal polynomials, which we denote $L_n(x; t, \lambda)$, and which satisfy 
    \begin{equation}
    \label{eq:laguerre-def}
    \int_0^\infty x^k L_n(x; t, \lambda) x^\lambda \ee^{-\left( x^2 + t x\right)} \dd x = 0 \qforq k = 0, 1, ..., n - 1,
    \end{equation}
where $\lambda > -1$, was carried out in \cite{MR3370612} in the regime $t = N^{1/2}c$, $c < -2$, and $-\frac12 \leq \lambda \leq \frac12$. This was followed by \cite{MR3795283}, where the asymptotics of the orthogonal polynomials in the same regime were given for all $c \in \R$.

\subsection{Connection to Painlev\'e-IV} Our interest in these polynomials stems from their intimate connections with \emph{special function solutions} of Painlev\'e-IV. While generic solutions of Painlev\'e equations are highly transcendental, all but Painlev\'e-I possess special solutions written in terms of elementary and/or classical special functions. We will focus our attention on the fourth Painlev\'e equation, given by
\begin{equation}
    \dod[2]{u}{x} = \dfrac{1}{2u} \left( \dod{u}{x} \right)^2 + \dfrac32 u^3 + 4xu^2 + 2(x^2 + 1 - 2\Theta_\infty) u - \dfrac{ 8 \Theta_0^2}{u}, \quad\Theta_0, \Theta_\infty \in \C. \tag{$\p{4}$}
    \label{eq:p4-ab}
\end{equation}  
 It is well-known that, for particular choices of parameters, Painlev\'e-IV possesses both rational solutions and solutions written in terms of parabolic cylinder functions, see e.g. \cites{BCH,GLS, MR0896355,MR0217352,O, MR3729446}. The latter occur whenever any of following conditions hold:
\begin{equation*}
    \Theta_0 \pm \Theta_\infty = 0 \qorq \Theta_0 \pm \Theta_\infty = 1.
\label{eq:p4-parameter-conds}
\end{equation*}
Starting with one such solution as a ``seed" function and iterating B\"acklund transformations produces families of special function solutions of Painlev\'e-IV. As an application of our results, we show that a particular family of special solutions of Painlev\'e-IV have asymptotically pole-free regions in the complex plane. This is done by realizing poles of solutions of Painlev\'e-IV as zeros of the Hankel determinants associated with the polynomials $P_n(z;t, N)$:
\begin{equation}
H_{n - 1} (t, N) := \det \left[ \mu_{i+j}(t, N) \right]_{i, j = 0}^{n - 1}, \quad  \mu_k(t, N) := \int_\R s^{k} \ee^{-NV(s;t)} \dd s.
\label{eq:hankel-det-def}
\end{equation}
 
The strategy employed in this work for proving the existence of asymptotically pole-free regions of special function solutions of Painlev\'e equations is not new. Some of the early implementations of this strategy were to study rational solutions of Painlev\'e-II in \cite{MR3431594} and the Painlev\'e-II hierarchy in \cite{MR3562403}. Later, this was applied to Airy solutions of Painlev\'e-II in \cites{BBDY, BBDY2}. A different approach which has been extremely successful in studying the asymptotics and pole distribution of special function solutions of Painlev\'e equations is the recasting of the equations as conditions of isomonodromic deformation for a related $2 \times 2$ system of differential equations. Indeed, this method was applied to study rational solutions of Painlev\'e-II in \cites{MR3265723, MR3350600}, rational solutions of Painlev\'e-III in \cites{MR3879971,MR4046831}, and rational solutions of Painlev\'e-IV in \cites{BM, MR4153120}. Many of the studies above were motivated by the intricate behavior of the poles of special solutions of Painlev\'e equations which had been observed for, say, Painlev\'e-II and Painlev\'e-IV in e.g. \cites{MR2291140, MR1975781, MR2538285, MR2804960, MR3529955}.

The relationship between $P_n(z; t, N)$, $\gamma_n^2(t, N)$, and solutions to Painlev\'e equations is well-documented in the literature. In \cite{Shohat}, Shohat derived differential equations for polynomials orthogonal with respect to semi-classical weights which include the weight in \eqref{eq:ortho}. Furthermore, he showed that $\gamma_n^2(t, N)$ satisfy discrete version of Painlev\'e-I:
\begin{equation}
\label{string-eq}
\gamma_{n}^2(t, N) \left(t + \gamma_{n-1}^2(t, N)  + \gamma_n^2(t, N)  + \gamma_{n+1}^2(t, N)  \right) = \dfrac{n}{N}.
\end{equation}
Equation \eqref{string-eq} reappeared in \cite{FIK}, where it was observed that the recurrence coefficients can be written in terms of parabolic cylinder functions. Magnus revisited the semi-classical weights and re-derived \eqref{string-eq} in \cites{Magnus, Magnus2}. Polynomials $L_n(x; t, \lambda)$ satisfying \eqref{eq:laguerre-def} were considered in \cites{BV,FVZ,CJ}, where it was shown that the corresponding recurrence coefficients are given in terms of Wronskians of parabolic cylinder functions and are thus related to special function solutions of Painlev\'e-IV. One can immediately see that a change-of-variables symmetrizes the weight of orthogonality in \eqref{eq:laguerre-def} and maps $L_n(z; t, \lambda)$ to a family of orthogonal polynomials $S_n(x; t, \lambda )$ satisfying 
    \begin{equation}
    \int_\R x^k S_n(x; t, \lambda ) |x|^{2\lambda + 1} \ee^{-(x^4 + tx^2)} \dd x = 0 \qforq k = 0, 1, ..., n-1, \quad \lambda > -1,
        \label{eq:gen-freud}
    \end{equation}
and contains $P_n(z; t, N)$ as a special case. Indeed, a re-scaling of variables in \eqref{eq:gen-freud} gives 
\begin{equation}
    P_n(x; t, N) = \left(2^{-\frac12 } N^{\frac14 }\right)^{-\deg S_n} S_n \left( 2^{-\frac12}N^{\frac14} x; -N^\frac12 t, -\frac12 \right).
    \label{eq:freud-connection}
\end{equation}
The connection between $S_n(x; t, \lambda)$ and Painlev\'e-IV was worked out and various properties of the polynomials $S_n(x; t, \lambda)$ were proved in \cites{MR3494152, MR3733254}. We note that the weight of orthogonality in \eqref{eq:ortho} can also be interpreted as a Freud weight (when $t\in \R_+$) for which there is a massive literature; see \cites{MR0419895, MR0862231}, \cite{DLMF}*{Section 18.32}, and references therein.

\subsection{Quartic matrix model} Another important application of the polynomials $P_n(x; t, N)$ is to the quartic one-matrix model, which is defined by the distribution 
    \begin{equation}
    \frac{1}{{\mc Z}_{n}(u, N)} \ee^{-N\mathrm{Tr}\left( \frac12 \mb M^2 + \frac{u}{4} \mb M^4\right)} \dd \mb M, 
    \label{eq:rmt-dist}
    \end{equation}
where $\mb M$ is an $n \times n$ Hermitian matrix, $u>0$, and $\mc Z_n(u, N)$ is the associated partition function. The quartic model was first investigated in \cites{BIPZ, BIZ}, where it was observed that when $n = N$, the associated free energy admits an asymptotic expansion in inverse powers of $N^{-2}$ and that the coefficient of $N^{-2g}$, as functions of $u$, serve as generating functions for the number of four-valent graphs on a Riemann surface of genus $g$. This expansion is known as the \emph{topological expansion}, and its existence and interpretation of its coefficients was proven in \cite{MR1953782} for one-cut, even polynomial potentials satisfying a ``small-coefficients" assumption, see \cite{MR1953782}*{Equation (1.3) and Theorem 1.1}. One-cut, even polynomial potentials were revisited in \cite{BI} and the topological expansion was proven using different methods for potentials which remain in the one-cut regime after certain deformations. In fact, an asymptotics expansion of the free energy in the boundary case $t = -2$ was also obtained in \cite{BI}. The asymptotics of the partition function in the two-cut region were obtained in \cite{MR3370612}.

\begin{remark}
An analogous model, known as the cubic model, was also investigated in \cite{BIPZ} and is connected with the enumeration of 3-valent graphs on Riemann surfaces. This corresponds to replacing the exponent in \eqref{eq:rmt-dist} by $\mathcal{V}(\mb M) =  \frac12 \mb M^2 + \frac{u}{3} \mb M^3$ , and results in the replacement of the potential $V(z)$ by the polynomial $\tilde{V}(z) = z^3/3 + tz^2/2$. The reader might immediately object since the integrals in \eqref{eq:ortho} and the total mass of \eqref{eq:rmt-dist} are not convergent. To make sense of the model, one needs to deform the contour of integration to the complex plane, at which point the model loses its interpretation as an ensemble of random Hermitian matrices. However, the enumeration property still holds as was shown in \cite{BD}. Later, the connection between the partition function and solutions of Painlev\'e-II was established in \cite{BBDY} and further expanded on in \cite{BBDY2}. In many ways, this manuscript can be viewed as the analog of \cite{BBDY} for the quartic model. 
\end{remark}

In this contribution, we content ourselves with the asymptotic analysis of $P_n(z; t, N)$, quantities related to these polynomials, and the corresponding special function solutions of Painlev\'e-IV. That being said, we do point out that the partition function and (the derivative of) the free energy can be written in terms of $\tau-$ and $\sigma$-functions of Painlev\'e-IV, see Remark \ref{remark:partition-tau}. 

\subsection{Overview} 

We recall the definition of the equilibrium measure and $S$-curves, state relevant results on these from the literature, and introduce functions which are important to our analysis in Section \ref{sec:asymptotic-analysis-prelim}. In Section \ref{sec:main2}, we state our main results: leading term asymptotics for the orthogonal polynomials, recurrence coefficients, and normalizing constants when $t \in \Oo_2$, see Theorem \ref{thm:asymptotics-poly} and Theorem \ref{thm:asymptotics-gamma-h}. In Section 4, we define and, as a consequence of our main results, deduce leading term asymptotics for certain parabolic cylinder solutions of Painlev\'e-IV (cf. Corollary \ref{cor:asymp-p4}). We collect various useful facts about the leading term of the asymptotic expansion of $P_n(z;t, N)$ in Section \ref{sec:g-szego}. Section \ref{sec:rh} is devoted to the implementation of the non-linear steepest descent analysis to prove the main results. In Appendix \ref{appendix-a}, we record an expression for the Hankel determinants $H_n(t, N)$ as a product of $\tau$-functions for Painlev\'e-IV and use it to prove a statement about the degrees of $P_n(t, N)$. 

% In Section \ref{sec:p4} we recall the symmetric and Hamiltonian reformulations of Painlev\'e-IV. In Section \ref{sec:p4-op} we recall the connection between the polynomials $P_n(z;t, N)$ and recurrence coefficients $\gamma_n^2(t, N)$ with solutions of Painlev\'e-IV. In Section \ref{sec:results-1} we state our result on the asymptotic behavior of special function solutions of Painlev\'e-IV. We recall the definition of the equilibrium measure and $S$-curves, state relevant results on these from the literature, and introduce various functions important to our analysis in Section \ref{sec:asymptotic-analysis-prelim}. The proofs of Theorem \ref{thm:asymptotics-poly} and Theorem \ref{thm:asymptotics-gamma-h} are carried out in Section \ref{sec:rh}. We conclude with a short discussion in Section \ref{sec:conclusion}. Appendix \ref{appendix-a} contains a proof of Lemma \ref{lemma:hankel-factorization} and the mapping of notation necessary to conclude Theorem \ref{thm:main2} from \cite{MR3494152}.

%%%%%%%%%%%%%%%%%%%%%%%%%%
%%%%%%%%%%%%%%%%%%%%%%%%%%
\section{Preliminaries to the asymptotic analysis of \texorpdfstring{$P_n(z;t, N), \gamma^2_n(t, N)$}{the orthogonal polynomials}}
\label{sec:asymptotic-analysis-prelim}

\subsection{More on polynomials $P_n(z;t, N)$}

Computing the polynomial $P_n(z;t, N)$ amounts to solving an under-determined, homogeneous system of equations and so it always exists, but may have degree smaller than $n$. Henceforth, for any $N \geq 0, t \in \C,$ and $n \in \Z_{\geq 0}$, we denote by $P_n(z;t, N)$ the \emph{unique monic polynomial of degree $n$ satisfying \eqref{eq:ortho}, if it exists}. It holds that $P_n(z;t, N)$ exists if and only if $H_{n-1}(t, N) \neq 0$, with $H_{n-1}(t, N)$ as in \eqref{eq:hankel-det-def}. 
\begin{remark}
    A consequence of our analysis will be that when $t \in \Oo_2$ and for all $n$ large enough, it holds that $ P_n(z; t, N)$ exists. The same is known to be true for $\Oo_1$, see e.g. \cite{BGM}. When $t \in \Oo_3$, degeneration can and will occur since $H_n(t, N)$ has zeros there, see Figure \ref{fig:polest} for example. We remark that the way these degeneration occur is highly structured, see Proposition \ref{prop:deg}.
\end{remark}

Given the possibility of degeneration, we first make sense of \eqref{eq:three-term}. Observe that by \cite{DLMF}*{Equation 12.5.1}, the evenness of $V(z; t)$, and the change-of-variables $s^2 = 2^{1/2} N^{-1/2} x$, we have 
\begin{equation}
        \mu_0(t, N) =  {2^{\frac14}{\pi}^{\frac12}}{N^{-\frac14}} \ee^{\frac{N}{8}t^2} D_{-\frac12} \left( {N^{\frac12}}{2^{-\frac12}}t \right),
    \label{eq:moment-0}
\end{equation}
where $D_\nu(\diamond)$ is the parabolic cylinder functions (cf. \cite{DLMF}*{Section 12}). Furthermore, for all $k \in \N$ the following identity holds:
\begin{equation}
    \dod[k]{\mu_0}{t}(t, N) = (-1)^k {N^k}{2^{-k}}\int_\R s^{2k} \ee^{-NV(s;t)} \dd s = (-1)^k {N^k}{2^{-k}} \mu_{2k}(t, N).
    \label{eq:moment-derivative}
\end{equation}
Combining \eqref{eq:moment-0}, \eqref{eq:moment-derivative}, we arrive at 
\begin{equation}
    \mu_{2k}(t, N) = 
       (-1)^{k}{2^{k + \frac14}}{N^{-(k +\frac14)}}\pi^{1/2} \dod[{k}]{}{t} \left(  \ee^{\frac{N}{8}t^2} D_{-\frac12}\left({N^{\frac12}}{2^{-\frac12}} t\right) \right).
       \label{eq:even-moments}
\end{equation}
Since odd moments vanish by evenness of $V(z; t)$, it follows that for every $n \in \N$, $H_{n - 1}(t, N)$ is an entire function of $t$. From the classical formulas for the normalizing constants
\begin{equation}
\label{eq:normalizing}
    h_n(t, N) := \int_\R P_n^2(z; t, N) \ee^{-NV(z; t)} \dd z =  \dfrac{H_n(t, N)}{H_{n-1}(t, N)},
\end{equation}
and the recurrence coefficients
\begin{equation}
\label{eq:gamma}
    \gamma_n^2(t, N) = \dfrac{ h_n(t, N)}{h_{n-1}(t, N)},
\end{equation}
it follows that $\gamma^2_{n}(t, N), h_n(t, N)$ are meromorphic functions of $t$. The same is true of $P_n(z; t, N)$ since for any $t$ such that $H_{n-1}(t, N) \neq 0$ we have
\begin{equation*}
    P_n(z; t, N) = \dfrac{1}{H_{n - 1}(t, N)}\det \begin{bmatrix}
\mu_{0}(t, N) & \mu_{1}(t, N) &  \cdots & \mu_{n}(t,N) \\
\mu_{1}(t, N) & \mu_{2}(t, N) & \cdots & \mu_{n+1}(t, N) \\
\vdots & \vdots &  \ddots & \vdots \\
\mu_{n-1}(t, N) & \mu_{n}(t, N) &  \cdots & \mu_{2n-1}(t, N) \\
1 & z &  \cdots & z^n
\end{bmatrix}.
\end{equation*}
We can now interpret \eqref{eq:three-term}: the roots of $\prod_{j = 0}^{n+1} H_j(t, N)$ form a discrete, countable subset of $\C$, outside of which \eqref{eq:three-term} holds and is then meromorphically continued to all $t \in \C$. Next, we recall notions from potential theory which will be necessary for the subsequent asymptotic analysis.

\subsection{Potential theory and S-curves}
It is well understood that the zeros of polynomials satisfying \eqref{eq:ortho} asymptotically distribute as a certain \emph{weighted equilibrium measure} on an \emph{S-contour} corresponding to the weight function \eqref{potential}. Let us start with some definitions.
\begin{definition} 
\label{def:eq}
Let $V$ be an entire function. The logarithmic energy in the external field $\Re V$ of a measure $\nu$ in the complex plane
is defined as
\[
E_V(\nu)=\iint \log \frac{1}{|s-t|}\mathrm d\nu(s) \mathrm d\nu(t)+\int \Re V(s)\mathrm d\nu(s).
\]
The equilibrium energy of a contour $\Ga$ in the external field $\Re V$ is equal to
\begin{equation}
\label{em1}
\mathcal E_V(\Ga)=\inf_{\nu\in \mathcal M(\Ga)} E_V(\nu),
\end{equation}
where $\mathcal M(\Ga)$ denotes the space of Borel probability measures on $\Ga$.
\end{definition}

When $\Re V(s)-\log|s|\to+\infty$ as $\Ga\ni s\to\infty$, there exists a unique minimizing measure for \eqref{em1}, which is called the {\it weighted equilibrium measure} of $\Ga$ in the external field $\Re V$, say $\mu_\Ga$, see \cite{MR1485778}*{Theorem I.1.3}. The support of $\mu_\Ga$, say $J_\Ga$, is a compact subset of $\Ga$. The equilibrium measure is characterized by the Euler--Lagrange variational conditions:
\begin{equation}
\label{em2}
U^{\mu_\Gamma}(z)+\dfrac{1}{2} \Re V(z)\;
\left\{
\begin{aligned}
&= \ell_\Ga,\qquad z\in J_\Ga,\\
&\ge \ell_\Ga,\qquad z\in \Ga\setminus J_\Ga,
\end{aligned}
\right.
\end{equation}
where $\ell_\Ga$ is a constant, the Lagrange multiplier, and \( U^\mu(z):=-\int\log|z-s|\dd\mu(s) \) is the logarithmic potential of $\mu$, see \cite{MR1485778}*{Theorem~I.3.3}. 

Due to the analyticity of the integrand in \eqref{eq:ortho}, the contour of integration in the same equation can be deformed without changing the orthogonality condition. Henceforth, we suppose that the polynomials $\{P_n(z;t, N)\}_{n = 0}^\infty$ are defined with $\Ga \in\mathcal T$ where \( \mathcal T \) is the following class of contours. 
\begin{definition} 
Let $\mathcal T$ be the set of all piece-wise smooth contours that extend to infinity in both directions and admit a parametrization $z(s)$, $s\in\R$, for which there exists $\epsilon\in(0,\pi/8)$ and $s_0>0$ such that
\begin{equation*}
\label{cm2b}
\left\{
\begin{array}{ll}
|\arg(z(s))|\leq \epsilon, & s\geq s_0, \medskip \\
|\arg(z(s))-\pi|\leq \epsilon, & s\leq -s_0,
\end{array}
\right., \quad \arg(z(s))\in[0,2\pi).
\end{equation*}
\end{definition}

\sloppy Following the works of Stahl \cites{stahl-domains, stahl-structure,stahl-complexortho} and Gonchar and Rakhmanov \cite{MR922628}, it is now well-understood that among the contours in $\mc T$, the zeros of the orthogonal polynomials will accumulate on the contour whose equilibrium measure has the \emph{S-property} in the external field $\Re V$.

\begin{definition} 
The support $J_\Ga$ has the S-property in the external field $\Re V$, if it consists of a finite number of open analytic arcs and their endpoints, and  on each arc it holds that
\begin{equation}
\label{em4}
\frac{\partial }{\partial n_+}\,\big(2U^{\mu_\Gamma}+\Re V\big)=
\frac{\partial }{\partial n_-}\,\big(2U^{\mu_\Gamma}+\Re V\big),
\end{equation}
where $\frac{\partial }{\partial n_+}$ and  $\frac{\partial }{\partial n_-}$ are the normal derivatives from the $(+)$- and $(-)$-side of $\Ga$. We shall say that a curve $\Ga\in\mathcal T$ is an S-curve in the field $\Re V$ if $J_\Ga$ has the S-property in this field.
\end{definition}

It is known that $J_\Ga$ is comprised of \emph{critical trajectories} of a certain quadratic differential. Recall that if $Q$ is a meromorphic function, a \emph{trajectory} (resp. \emph{orthogonal trajectory}) of a quadratic differential $-Q(z)\mathrm dz^2$ is a maximal\footnote{A smooth arc satisfying \eqref{eq:qd-cond} is maximal if it is not properly contained in any other smooth arc satisfying \eqref{eq:qd-cond}.} smooth arc on which
\begin{equation}
-Q(z(s))\big(z^\prime(s)\big)^2>0 \quad \big(\text{resp.} \quad -Q(z(s))\big(z^\prime(s)\big)^2<0\big)
    \label{eq:qd-cond}
\end{equation}
for any local uniformizing parameter. A trajectory is called \emph{critical} if it is incident with a \emph{finite critical point} (a zero or a simple pole of $-Q(z)\mathrm dz^2$) and it is called \emph{short} if it is incident only with finite critical points. We designate the expression \emph{critical (orthogonal) graph of $-Q(z)\mathrm dz^2$} for the totality of the critical (orthogonal) trajectories $-Q(z)\mathrm dz^2$.

In our setting, Kuijlaars and Silva \cite{MR3306308}*{Theorems~2.3 and~2.4} already identified some important features of the function $Q$, which we summarize below for $V(z;t)$ as in \eqref{potential}.

\begin{theorem}[\cite{MR3306308}]
\label{KS-thm}
Let $V(z;t)$ be given by \eqref{potential}. Then,
\label{fundamental} 
\begin{enumerate}[label=(\alph*)]
  \item there exists a contour $\Ga_t\in\mathcal T$ such that
\begin{equation}
\label{em3}
\mathcal E_V(\Ga_t)=\sup_{\Ga\in\mathcal T} \mathcal E_V(\Ga).
\end{equation}
  \item  The support $J_t$ of the equilibrium measure $\mu_t:=\mu_{\Ga_t}$ has the S-property in the external field $\Re V(z;t)$ and the measure \( \mu_t \) is uniquely determined by \eqref{em4}. Therefore, \( \mu_t \) and its support \( J_t:= J_{\Ga_t} \) are the same for every $\Ga_t$ satisfying \eqref{em3}.
  \item The function
\begin{equation}
\label{em5}
Q(z;t)=\left(\frac{V'(z;t)}{2}- \int \frac{\mathrm d\mu_t(x)}{z-x}\right)^2,\quad z\in \C\setminus J_t,
\end{equation}
is a polynomial of degree 6.
\item The support $J_t$ consists of some short critical trajectories of the quadratic differential $-Q(z;t)\mathrm dz^2$ and the equation
\begin{equation}
\label{em6}
\mathrm d\mu_t(z)=\frac{1}{\pi \ii}\,Q_+^{1/2}(z;t)\mathrm dz, \quad z\in J_t,
\end{equation}
holds on each critical trajectory, where $Q^{1/2}(z;t)$ is analytic in $\C \setminus J_t$ and satisfying $Q^{1/2}(z;t)=\frac12z^3+\mathcal{O}(z)$ as $z\to\infty$ (in what follows, \( Q^{1/2}(z;t) \) will always stand for this branch).
\end{enumerate}  
\end{theorem}

For a description of the structure of the critical graphs of a general quadratic differential, see e.g. \cites{MR0096806,MR743423}. Since $\deg\, Q=6$, $J_t$ consists of one, two, or three arcs, corresponding to the cases where $Q(z;t)$ has two, four, or six simple zeros, respectively. This depends on the choice of $t$, and the phase diagram is shown in Figure \ref{fig:phase}.

\begin{remark}
    Away from \( J_t \), one has freedom in choosing \( \Ga_t \). Indeed, let $e \in J_t$ be arbitrary and set
    \begin{equation}
    \label{em0}
    \mathcal U(z;t) := \Re \left(2\int_e^z Q^{1/2}(z;t)\dd z \right) = \ell_{\Ga_t} - \Re \ (V(z;t)) - 2U^{\mu_t}(z),
    \end{equation}
    where the second equality follows from \eqref{em5} (since the constant \( \ell_{\Gamma_t} \) in \eqref{em2} is the same for both connected components of \( J_t \) and the integrand is purely imaginary on \( J_t \), the choice of \( e \) is indeed not important). Clearly, \( \mathcal U(z;t) \) is a subharmonic function (harmonic away from \( J_t \)) which is equal to zero on \( J_t \) by \eqref{em2}. The trajectories of \( -Q(z;t)\dd z^2 \) emanating out of the endpoints of \( J_t \) belong to the set \( \{z:\mathcal U(z;t)=0 \} \) and it follows from the variational condition \eqref{em2} that \( \Ga_t\setminus J_t\subset \{z:\mathcal U(z;t)<0\} \). However, within the region \( \{z:\mathcal U(z;t)<0 \} \) the set \( \Ga_t\setminus J_t \) can be varied freely.
    \label{remark:contour-freedom}
\end{remark}
The critical graphs of $-Q(z; t) \dd z^2$ were studied in great detail in \cite{BGM} in all the above cases. The next subsection defines the regions $\Oo_j$ and summarizes results relevant to $t\in \Oo_2$.
\subsection{\texorpdfstring{Critical graph of $-Q(z)\dd z^2$}{Critical graphs}}

To state the next theorem, we will need the following notation 
\begin{notation}
$\Gamma (z_1, z_2)$ (resp. $\Ga[z_1, z_2]$) stands for the trajectory or orthogonal trajectory (resp. the closure of) of the differential \( -Q(z;t)\dd z^2 \) connecting $z_1$ and $z_2$, oriented from $z_1$ to $z_2$, and $\Ga(z,e^{\mathrm i\theta}\infty)$ (resp. $\Ga(e^{\mathrm i\theta}\infty,z)$) stands for the orthogonal trajectory ending at $z$, approaching infinity at the angle $\theta$, and oriented away from $z$ (resp. oriented towards $z$).\footnote{This notation is unambiguous as the corresponding trajectories are unique for polynomial differentials as follows from Teichm\"uller's lemma.}
\label{note:notation}
\end{notation}
\begin{definition}
     The open sets $\Oo_j, j = 1, 2, 3$ are defined by the property that 
    \[
    t\in \Oo_j \implies Q(z;t) \text{ has $2j$ simple zeros}.
    \]
    These sets are shown in Figure \ref{fig:phase}, where $\Oo_j$ is labelled the ``$j$-cut region."
\end{definition}

Let $I_t$ be an analytic arc connecting $-a_2(t), a_2(t)$, oriented towards $a_2(t)$, lies entirely in the region $\{\mathcal U(z; t) <0\}$, and for a fixed neighborhood $U_{\pm a_2}$ of $\pm a_2(t)$, $I_t \cap U_{\pm a_2}$ coincides with orthogonal trajectories emanating from $\pm a_2$.The following theorem summarizes the relevant results from \cite{BGM}*{Section 2.4.2}.

\begin{theorem}[\cite{BGM}] \label{thm:2-cut}
Let $\mu_t$ and $Q(z; t)$ be as in Theorem \ref{KS-thm}.  If $t \in \overline{\mathcal{O}_2}$, then 
\begin{equation}
\label{Q-two-cut}
    Q(z;t) = \dfrac{1}{4} z^2 (z^2 - a_2^2(t))(z^2 - b_2^2(t)),
\end{equation}
where 
\begin{equation}
\label{a-b-2-cut}
    a_2(t) = \sqrt{-2 -t}, \quad b_2(t) = \sqrt{2 - t},
\end{equation}
and $\sqrt{\diamond}$ is the principle branch. Furthermore, we have that the support of $\mu_t$ is given by 
\begin{equation}
    J_{t} = \Ga[-b_2, -a_2] \cup \Ga[a_2, b_2] := J_{t, 1} \cup J_{t, 2},
    \label{eq:jt-arcs}
\end{equation}
    where $J_{t, 1} = \Ga[-b_2, -a_2]$ and $J_{t, 2} = \Ga[a_2, b_2]$. For $t \in \mathcal{O}_2$, we take $\Ga_t = \Ga(\ee^{\pi \ii} \infty, -b_2) \cup J_{t, 1} \cup I_t \cup J_{t, 2} \cup \Ga(b_2 , \ee^{0\pi \ii}\infty)$, see Figure \ref{fig:eta-sign-chart} for a schematic.
\end{theorem}
The support $J_t$ and sets $\{z \ : \  \mc U(z; t)<0\}$ were numerically computed for various choices $t$ in \cite{BGM}*{Figures 3 - 8}.
\begin{figure}[t]
	\centering
        \includegraphics[scale = 0.25]{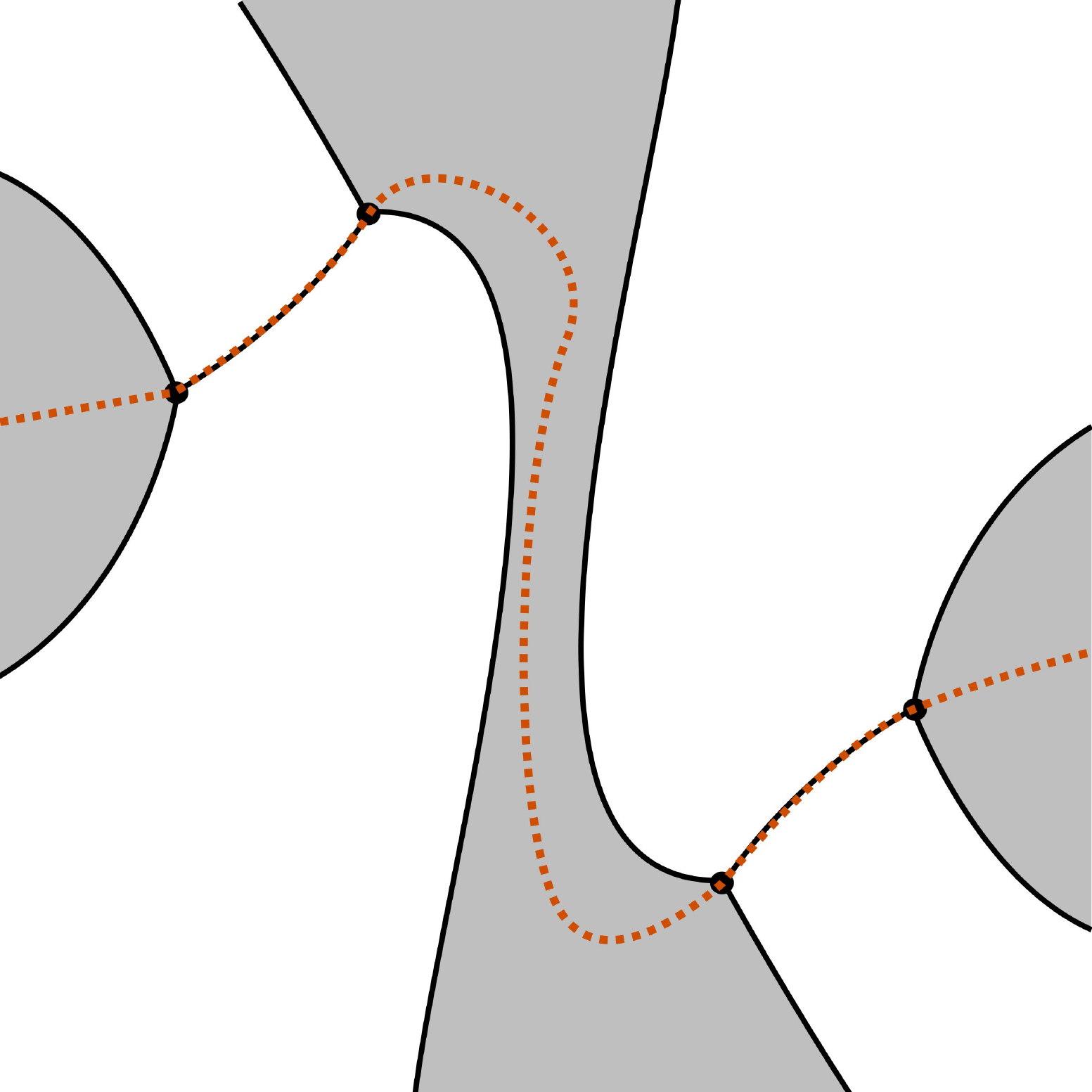}
        \put(-150, 110){$-b_2$}
        \put(-150, 150){$-a_2$}
        \put(-55, 30){$a_2$}
        \put(-45, 70){$b_2$}
        \caption{Schematic representation of the trajectories of $-Q(z;t)\dd z$ with $Q(z;t)$ as in \eqref{Q-two-cut} and $t \in \Oo_2$. The shaded region is the set $\{z \ : \ \mc U(z; t) <0\}$. The dashed contour is a sample choice of $\Gamma_t$.}
	\label{fig:eta-sign-chart}
\end{figure}
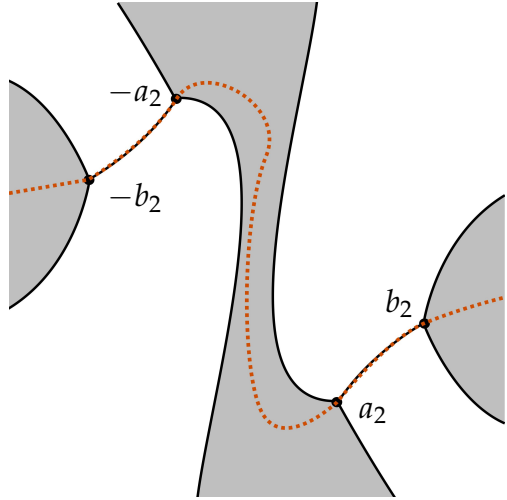

\subsection{\texorpdfstring{Construction of the $g$-function}{Construction of the g-function}}
\label{sec:g-fun}
The first step in Section \ref{sec:irhp} will be to normalize the initial Riemann-Hilbert problem \hyperref[rhy]{\rhy} at $z = \infty$, for which we require a function $g(z)$, colloquially known as the \emph{$g$-function}: 
\begin{equation}
\label{g-def}
    g(z; t):= \int_{J_{t}} \log(z - x) \dd \mu_t(x),
\end{equation}
where, for a fixed $x \in J_t$, where $\log(\diamond - s)$  is the principle branch with a branch cut starting at $s$ and extending along $\Gamma_t$ to $\ee^{\pi \ii} \infty$. It follows from the definition of $g$ that 
\begin{equation}
\label{g-potential}
  U^{\mu_t}(z) = - \Re ( g(z;t) ) 
\end{equation}
and we can now rewrite the Euler-Lagrange conditions:
\begin{equation}
\label{euler-lagrange-g}
    \begin{aligned}
    -\Re g(z;t) + \dfrac{1}{2} \Re V(z;t) = \left\{ \begin{array}{ll} = \ell_{t}, & z \in J_{t}, \\ \geq \ell_{t}, & z \in \Gamma_t \setminus J_t. \end{array} \right.
    \end{aligned}
\end{equation}
Furthermore, we note that 
\begin{equation}
\label{resolvant}
    g'(z;t) = \int_{J_t} \dfrac{\dd \mu_t(x)}{z - x}. 
\end{equation}
For $t \in \Oo_2$, the $g$-function was constructed in \cite{BGM}; its relevant properties are recalled in Section \ref{sec:g-szego}.

\subsection{Szeg\H{o} function}
To handle the general double-scaling limit where $\{N_n\}_{n=1}^\infty \subset \N$ with $|N_n - n| < C$ for a fixed $C \geq 0$\footnote{As opposed to only the case $N=n$, which does not require the construction in this section.}, we will need the \emph{Szeg\H{o} function} of $\ee^{V(z;t)}$. Let
\begin{equation}
\label{Q-R}
    Q^{1/2}(z;t) = \dfrac{1}{2} z R^{1/2}(z;t)
\end{equation}
where $R^{1/2}(z;t) = \sqrt{(z^2 - a_2^2)(z^2 - b_2^2)}$ is the branch which is analytic in $\C \setminus J_t$ and satisfies $R^{1/2}(z) = z^2 + \Oo(z)$ as $z \to \infty$. The following proposition defines and summarizes the properties of the Szeg\H{o} function. 

\begin{proposition}
\label{Szego-prop}
Let $R^{1/2}(z;t)$ be is as in \eqref{Q-R}. Then, the Szeg\H{o} function
\begin{equation}
\label{Szego-fun}
    \mathcal{D}(z;t) := \exp \left( \dfrac{1}{2}V(z;t) - \dfrac{1}{8} (z^2 + t) R^{1/2}(z;t)\right). 
\end{equation}
is analytic in $\overline{\C}\setminus J_t$ and has continuous boundary values on $J_t \setminus \{\pm a_2, \pm b_2\}$ which satisfy
\begin{equation}
\label{Szego-jump}
    \mathcal{D}_+(x; t) \mathcal{D}_-(x; t)  = \ee^{V(x;t)}, \quad x \in J_t,
\end{equation}
where $\pm$-signs are with respect to the orientation of $J_t$ defined in \eqref{eq:jt-arcs} and Notation \ref{note:notation}. Furthermore, $\mathcal{D}(z;t)$ is bounded as $z \to e \in \{\pm a_2, \pm b_2\}$, and there exists $\ell_*(t) \in \C$ (cf. \eqref{ell}) such that
\begin{equation}
\label{Szego-lim}
\mathcal{D}(z;t) = \ee^{-\ell_*(t)/2} \left( 1 + \dfrac{1}{4z^4} + \Oo\left( {z^{-6}} \right) \right) \qasq z \to \infty.
\end{equation}

\end{proposition}
Proposition \ref{Szego-prop} is proved in Section \ref{sec:szego-proof}. It will be convenient to define the \emph{normalized Szeg\H{o} function} as
\begin{equation}
    D(z;t) := \ee^{\ell_*(t)/2}\mathcal{D}(z;t).
    \label{eq:normalized-szego-fun}
\end{equation}
\subsection{Riemann surface}
\label{subsec:RS}
We will need an arsenal of functions that naturally live on an associated Riemann surface, which we introduce here. Let $R(z;t) = (z^2 - a_2^2(t))(z^2 - b_2^2(t))$ and
\[
\RS := \left \{(z, w) \ : \ w^2 = R(z; t) \right\}.
\]
Denote by \( \pi:\RS\to\overline\C \) the natural projection \( \pi(\z) =z \) and by $\diamond^*: \RS \to \RS$ the natural involution $\z = (z, w) \implies \z^* = (z, w)^* = (z, -w)$.   We use notation \( \z,\x,\boldsymbol a \) for points on \( \RS \) with natural projections \( z,x,a \). 

The function \( w(\z) \), defined by \( w^2(\z)=R(z; t) \), is meromorphic on \( \RS \) with simple zeros at the ramification points \( \boldsymbol E_t =\{ \pm \boldsymbol a_2,\pm \boldsymbol b_2 \} \), double poles at the points on top of infinity, and is otherwise non-vanishing and finite.
Set
\[
\boldsymbol\Delta:=\pi^{-1}(J_t) \quad \text{and} \quad \RS = \RS^{(0)}\cup\boldsymbol\Delta\cup \RS^{(1)},
\]
where the domains \( \RS^{(k)} \) project onto \( \overline\C\setminus J_t \) with labels chosen so that \( w(\z) = (-1)^kz^2 + \Oo(z) \) as \( \z \) approaches the point on top of infinity within \( \RS^{(k)} \). For \( z\in\overline\C\setminus J_t \) we let \( z^{(k)} \) stand for the unique \( \z\in \RS^{(k)} \) with \( \pi(\z)=z \). Since the restrictions $\pi|_{\RS^{(k)}}$ are bijections, given a function $F: \RS \to \C$ we define $F^{(k)}: \C \setminus J_t \to \C$ to be the pull-back $F\left( ( \pi|_{\RS^{(k)}})^{-1}(z) \right)$.

Next, we define a homology basis on \( \RS \) in the following way: Theorem \ref{thm:2-cut} and the discussion preceding it, we let 
\[ 
\boldsymbol\alpha := \pi^{-1}(I_t) \quad \text{and} \quad \boldsymbol\beta := \pi^{-1}(J_{t,1}),
\]
where \( \boldsymbol\alpha \) is oriented towards \( -\boldsymbol a_2 \) within \( \RS^{(0)} \) and \( \boldsymbol\beta \) is oriented so that \( \boldsymbol\alpha,\boldsymbol\beta \) form the right pair at \( -\boldsymbol a_2 \), see Figure \ref{fig:planar-model}. We will often use the notation $\RS_{\ualpha, \ubeta} := \RS \setminus \{ \ualpha, \ubeta\}$.

\begin{figure}[t]

\begin{tikzpicture}[scale = 1.0]
%shading
\draw[fill=black!30] (0, 0) -- (0, 2) -- (4, 2) -- (4, -2) -- (0, -2) -- cycle;
%
%nodes
%a_1
\node at (4, 2){\textbullet};\node[above right] at (4, 2){$- \boldsymbol a_2$};
\node at (4, -2){\textbullet};\node[below right] at (4, -2){$-\boldsymbol a_2$};
\node at (-4, 2){\textbullet};\node[above left] at (-4, 2){$-\boldsymbol a_2$};
\node at (-4, -2){\textbullet};\node[below left] at (-4, -2){$-\boldsymbol a_2$};
%a_2
\node at (-4, 0){\textbullet};
\node[left] at (-4, 0){$-\boldsymbol b_2$};
\node at (4, 0){\textbullet};
\node[right] at (4, 0){$-\boldsymbol b_2$};
%a_3
\node at (0, 0){\textbullet }; \node [above right] at (0, 0){$\boldsymbol b_2$};
%a_4
\node at (0, 2){\textbullet}; \node [above] at (0, 2){$\boldsymbol a_2$};
\node at (0, -2){\textbullet}; \node [below] at (0, -2){$\boldsymbol a_2$};

%alpha
\draw (-4, -2) -- (4, -2); \node at (-2, -2){$<$};\node at (2, -2){$<$};
\node[below] at (-2, -2.2){$\boldsymbol \alpha$};\node[below] at (2, -2.2){$\boldsymbol \alpha$};
\draw (-4, 2) -- (4, 2); \node at (-2, 2){$<$};\node at (2, 2){$<$};
\node[above] at (-2, 2.2){$\boldsymbol \alpha$};\node[above] at (2, 2.2){$\boldsymbol \alpha$};
%beta
\draw (-4, -2) -- (-4, 2); \node[rotate = 90] at (-4, -1){$<<$};\node[rotate = 90] at (-4, 1){$<<$};
\node[left] at (-4.2, -1){$\boldsymbol \beta$};\node[left] at (-4.2, 1){$\boldsymbol \beta$};
\draw (4, -2) -- (4, 2); \node[rotate = 90] at (4, -1){$<<$};\node[rotate = 90] at (4, 1){$<<$};
\node[right] at (4.2, -1){$\boldsymbol \beta$};\node[right] at (4.2, 1){$\boldsymbol \beta$};
%sheets
\node at (-2, 0){$\RS^{(0)}$};
\node at (2, 0){$\RS^{(1)}$};
\end{tikzpicture}
\caption{Schematic representation of the surface \( \RS \) (shaded region represents \( \RS^{(1)} \)), which topologically is a torus, and the homology basis \(\boldsymbol \alpha,\boldsymbol \beta\) .}
\label{fig:planar-model}
\end{figure}
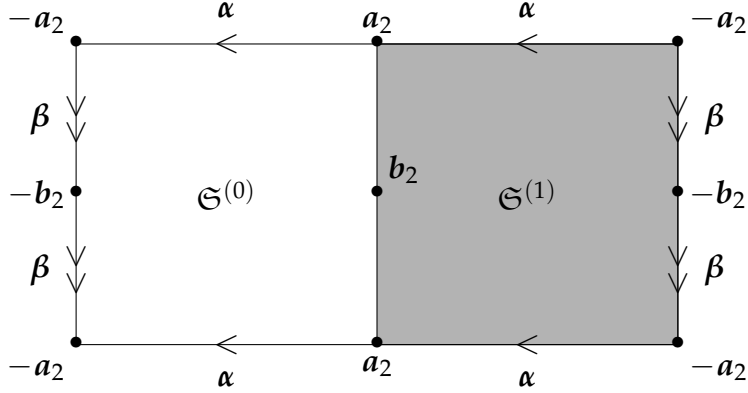
Surface $\RS$ is genus 1, and so the vector space of holomorphic differentials is one dimensional and we choose the following basis 
\begin{equation}
\label{hd}
    \mathcal{H}(z) = \left( \int_{\boldsymbol \alpha}\dfrac{\dd x}{w(\x)} \right)^{-1} \dfrac{\dd z}{w(\z)},
\end{equation}
which is normalized to have period 1 on $\boldsymbol \alpha$; it follows (see, e.g. \cite{MR583745}) that when $t \in \mathcal{O}_2$,
\[
\mathrm{B}(t) := \int_{\boldsymbol \beta} \hd \implies \Im \mathrm B(t) > 0.
\]
In what follows, we suppress the dependence of $\B(t)$ on $t$ and write $\B(t) = \B$.
\subsubsection{Abel's map} Let $\A: \RSab \to \C$ be the abelian integral
\begin{equation*}
\mathcal{A}(\z) := \int_{\b_2}^\z \hd, \quad \z \in \RS_{\ualpha, \ubeta},
\end{equation*}
where the path of integration is taken in $\RS_{\ualpha, \ubeta}$. $\A$ is holomorphic in $\RSab$, and it follows from the definition of $\B$ and the normalization of $\hd$ that
\begin{equation}
\label{abel-jumps}
    (\A_+ - \A_-)(\x) = \left\{ \begin{array}{rl} -\B, & \x \in \ualpha\setminus \{- \a_2\}, \medskip \\ 1, & \x \in \ubeta\setminus \{- \a_2\}.\end{array} \right.
\end{equation}

Let $\jac(\RS) = \C / \{\Z + \B \Z\}$ denote the Jacobi variety of $\RS$, whose elements are the following equivalence classes, given $x \in \C$, 
\[
[x] := \{x + k + j \B \in \C \ : x \in \ k, j \in \Z \}.
\]
It follows from \eqref{abel-jumps} that the \emph{Abel's map} $\z \mapsto [\A(\z)]$ is well-defined. Since $\RS$ has genus 1, the Abel map is an isomorphism \cite{MR583745}*{Section III.6}. We will use this to formulate a Jacobi inversion problem in Section \ref{sec:jip}.

\subsection{Jacobi inversion problem} 
\label{sec:jip}
In view of the Abel map being an isomorphism, let $\z_{[x]}$ denote the unique solution to the equation 
\[
\left[\A \left(\z_{[x]} \right) \right] = [x]. 
\]
The problem of determining $\z_{[x]}$ for a given $x \in \C$ is known as the Jacobi inversion problem. We will be interested in two problems, namely 
\begin{equation*}
\label{jip}
    \left[\A \left(\z_{k, [x_*]} \right) \right] = [x_*] \quad \text{ where } \quad x_* = \A \left(0^{(k)}\right) + \dfrac{1}{2}, \quad  k \in \{0, 1\}.
\end{equation*}
Due to various symmetries of $\RS$, it turns out we can solve this problem exactly. This is the content of the following proposition.
\begin{proposition}
\label{prop-jip}
Let $x_*$ be as above, and $k \in \{0, 1\}$ then $\z_{k, [x_*]} = 0^{(1-k)}$.
\end{proposition}
\begin{proof}
Let $\tilde{\ualpha}$ be an involution-symmetric cycle homologous to $\ualpha$ containing $0^{(0)}, 0^{(1)}$, and note the following two symmetries:
\begin{equation}
\label{A-syms}
\int_{\pm \boldsymbol a_2}^{\z} \hd = - \int_{\pm \boldsymbol a_2}^{\z^*} \hd, \qandq \int_{\pm \boldsymbol a_2}^{0^{(k)}} \hd = \int_{0^{(k)}}^{\mp \boldsymbol a_2} \hd.
\end{equation}
Indeed, the first equality follows from $w(\z^*) = -w(\z^*)$ while the other follows from $w(\pi(\z)) = w(-\pi(\z))$. Using the first symmetry and normalization \eqref{hd}, for a path contained in $\RS^{(1)}$ we have that 
\[
\int_{-\boldsymbol a_2}^{\boldsymbol a_2} \hd = \dfrac{1}{2}.
\]
Using this and \eqref{A-syms}, we have
\begin{multline*}
\dfrac{1}{2} = \int_{-\boldsymbol a_2}^{\boldsymbol a_2} \hd = \int_{-\boldsymbol a_2}^{0^{(1)}} \hd + \int_{0^{(1)}}^{\boldsymbol a_2} \hd = -\int_{-\boldsymbol a_2}^{0^{(0)}} \hd + \int_{0^{(1)}}^{\boldsymbol a_2} \hd \\
= \int_{0^{(0)}}^{-\boldsymbol a_2} \hd + \int_{0^{(1)}}^{\boldsymbol a_2} \hd 
= \int^{0^{(0)}}_{\boldsymbol a_2} \hd + \int_{0^{(1)}}^{\boldsymbol a_2} \hd.
\end{multline*}
So, $\left[\A\left( 0^{(0)} \right) - \A\left(0^{(1)} \right) \right] = \left[1/{2}\right]$. Noting that $[1/2] = [-1/2]$ yields the desired result.
\end{proof}
In fact, the proof of the above proposition yields a slightly better identity, namely that 
\begin{equation}
    \A(0^{(0)}) = \A(0^{(1)}) + \dfrac{1}{2},
    \label{eq:abel-zero-identity}
\end{equation}
which we will use in the following subsection.

\subsection{Theta functions} Let \( \theta(x;\B) \) be the Riemann theta function associated with \( \B \), i.e.,
\begin{equation}
\label{Rtheta}
\theta(x;\B) = \sum_{k\in\Z}\exp\left( \pi\ii \mathsf \B k^2+2\pi\ii xk\right), \quad x\in\C.
\end{equation}
The function $\theta(x)$ is holomorphic in $\C$ and enjoys the following periodicity properties:
\begin{equation*}
\label{ab4}
\theta(x+j+\B m) = \exp\left\{-\pi\ii \B m^2-2\pi\ii x m\right\}\theta(x), \qquad j,m\in\Z.
\end{equation*}
It is also known that $\theta(x)$ vanishes only at the points of the lattice \( \left[\frac{\mathsf B+1}2\right] \). 

Let
\begin{equation}
\label{theta-k}
\Theta(\z) = \frac{\theta\left(\A(\z) - {\A}\left(0^{(1)}\right) - \frac{\mathsf B+1}2\right)}{\theta\left(\A(\z) - {\A}\left(0^{(0)}\right) - \frac{\mathsf B+1}2\right)}.
\end{equation}
The function $\Theta(\z)$ is meromorphic on \( \RS\setminus \boldsymbol\alpha \) with exactly one simple pole at $\z = 0^{(0)}$, and exactly one simple zero at \( \z = 0^{(1)} \) ($\Theta$ can be analytically continued to multiplicatively multi-valued functions on the whole surface \( \RS \); thus, we can talk about simplicity of a pole or zero whether it belongs to the cycles of a homology basis or not). Moreover, according to \eqref{abel-jumps}, \eqref{eq:abel-zero-identity}, \eqref{Rtheta}, and \eqref{theta-k}, $\Theta(\z)$ has continuous boundary values on $\boldsymbol\alpha,\boldsymbol\beta$ away from \( -\boldsymbol a_2 \) that satisfy
\begin{equation}
\label{theta-jump}
\Theta_{+}(\x) = \ee^{\pi \ii }\Theta_{-}(\x), \quad \x \in \ualpha \setminus \{- \boldsymbol a_2\}.
\end{equation}
\subsection{More auxiliary functions}
Let 
\begin{equation}
\label{gamma}
    \gamma(z) := \left( \dfrac{z - b_2}{z + b_2} \dfrac{z + a_2}{z - a_2} \right)^{1/4}, \quad z \in \overline{\C} \setminus J_t,
\end{equation}
be the branch analytic outside of $J_t$ and satisfying $\gamma(\infty) = 1$. Furthermore, let 
\begin{equation}
\label{A-B-def}
    A(z) := \dfrac{\gamma(z) + \gamma^{-1}(z)}{2}, \quad B(z):=  \dfrac{\gamma(z) - \gamma^{-1}(z)}{2\ii}, \quad z \in \overline{\C} \setminus J_t.
\end{equation}
Both $A(z), B(z)$ are analytic in $\overline{\C} \setminus J_t$ and satisfy 
\begin{equation}
\label{A-B-lim}
A(\infty) = 1, \qandq B(\infty) = 0.
\end{equation}
Furthermore, they satisfy 
\begin{equation}
\label{A-B-jump}
A_{+}(x) =  B_{-}(x) \qandq A_{-}(x) = -B_+(x), \quad x \in J_t \setminus \{\pm a_2, \pm b_2\}.
\end{equation}
Observe that with the choice of branch of $\gamma$, we have that $B(0) = 0$. Indeed, the equation $(AB)(z) = 0$ is equivalent to $\gamma^4(z) = 1$, which has two solutions in $\overline{\C}$, namely $z = \infty$ and $z = 0$. To see that $z = 0$ is a root of $B(z)$ instead of $A(z)$, we will show that $\gamma(0) = 1$. To this end, let $w = \gamma^4(z)$ and observe that $\gamma^4(J_{t, 1})$ is a contour connecting $w=0$ to $w=\infty$. Observe that
\[
\gamma^4(z) = \gamma^4(x) \implies z = x \quad \text{or} \quad z = -\dfrac{a_2 b_2}{x}.
\]
Let $L$ denote a contour connecting $z = \infty$ to $z = 0$ while avoiding $J_{t, 1}$ and its image under the map $z \mapsto -a_2 b_2/z$. Then, $\gamma^4(L)$ is a closed contour containing $w = 1$ and not winding around the origin. Hence, the continuation of the principle branch of $\diamond^{1/4}$ along $\gamma^4(L)$ is single-valued, and produces $\gamma(L)$. Finally, since $\gamma(\infty) = 1$, we have $\gamma(0) = 1$ as desired. 

\section{Main results on asymptotics of \texorpdfstring{$P_n(z;t, N)$}{orthogonal polynomials} and \texorpdfstring{$\gamma_n^2(t, N)$}{recurrence coefficients}}
\label{sec:main2}
To simplify the formulation of the following theorems, we make the following definition
\begin{definition}
    An estimate $f_n(z, t) = \Oo(n^{-\alpha})$ holds $(z, t)$-\emph{locally uniformly} for $z \in V$ as $n \to \infty$ if for each compact $T \subset \Oo_2$ there is a constant $C(V; T) <\infty$ such that for all $z \in V,  t \in T$ and $n$ large enough
    \[
    |f_n(z; t)| \leq C(V;T) n^{-\alpha}.
    \]
\end{definition}
We are now ready to state our main results regarding the asymptotic behavior of the orthogonal polynomials and their recurrence coefficients and normalizing constants. 
\begin{theorem}
\label{thm:asymptotics-poly}
Let $t \in \Oo_2$ and $\{N_n\}_{n\in \N}$ be a sequence of integers such that there exists $C \in \N$ with $|N_n - n| < C$ for all $n \in \N$. Then, 
    \begin{equation}
    \label{p-asymp-outside}
    P_n(z) = \ee^{n g(z)} \left( D^{N_n - n}(z)A(z)\cdot \left\{ \begin{array}{ll} 1, & n \in 2\N, \medskip \\ \dfrac{\Theta^{(0)}(\infty)}{\Theta^{(0)}(z)} , & n \in \N\setminus 2\N \end{array} \right \} + \Oo\left( {n^{-1}} \right) \right), \qasq n \to \infty,
    \end{equation}
    $(z,t)$-locally uniformly in $\C \setminus J_t$, where $g(z), D(z), A(z)$ are as in \eqref{g-def}, \eqref{eq:normalized-szego-fun}, \eqref{A-B-def}, respectively, and the superscripts $\diamond^{(0)}$ are as in Section \ref{subsec:RS}.
\end{theorem}
\begin{remark}
    Note that the oddness of the function $P_{2n+1}(z;t, N)$ is reflected in the fact that the leading term in \eqref{p-asymp-outside} vanishes at $z = 0$ since $\Theta^{(0)}(z)$ possesses a simple pole there.
\end{remark}

Formulas similar to \eqref{p-asymp-outside} can easily be deduced in the interior of $J_t$, but since we do not use these we do not state them here. One immediate consequence of Theorem \ref{thm:asymptotics-poly} and \eqref{g-log-requirement} is that for all $t \in \Oo_2$ and $n$ large enough, $\deg P_n = n$. In view of the discussion at the beginning of Section \ref{sec:asymptotic-analysis-prelim} we have the following corollary. 

\begin{corollary}
    \label{cor:zero-free}
    Let $K \subset \Oo_2$ be a compact set and $\{N_n\}_{n\in \N}$ be as in the statement of Theorem \ref{p-asymp-outside}. Then, there exists $n_*\in \N$ large enough so that for all $n >n_*$, $H_n(t, N_n)$ is non-vanishing for $t\in K$, and $h_n(t, N_n), \gamma_{n}^2(t, N_n)$ are non-vanishing and bounded for $t \in K$.
\end{corollary}

\begin{remark}
    The same boundedness and non-vanishing result is already known for $K \subset \Oo_1$; this follows directly from the asymptotic analysis in \cite{BGM}. In fact, motivated by Figure \ref{fig:polest}, one might suspect that the results can be extended to \emph{closed} subsets of $\Oo_1, \Oo_2$. An analogous story takes place in the cubic model, where boundedness on compact subsets of the one-cut region was shown in \cite{MR3607591} and upgraded to closed subsets in \cite{BBDY}.
\end{remark}
The asymptotic analysis in Section \ref{sec:rh} also yields asymptotic formulas for $\gamma_n^2(t,N), h_n(t,N)$, which are stated in the following theorem.
\begin{theorem}
    \label{thm:asymptotics-gamma-h}
    Let $t \in \Oo_2$, $\{N_n\}_{n \in \N}$ be as in Theorem \ref{thm:asymptotics-poly}, and $a_2(t), b_2(t)$ be as in \eqref{a-b-2-cut}. Then, 
    \begin{equation}
    \label{gamma-n-final}
    \gamma^2_n(t, N_n) =  \left\{  \begin{array}{ll}
    \dfrac{1}{4}(a_2(t) - b_2(t))^2, & n \in 2\N, \medskip \\  \dfrac{4}{(a_2(t) - b_2(t))^2} , & n \in \N \setminus 2\N.
    \end{array}\right \} + \Oo\left( n^{-1} \right) \qasq n\to \infty.
    \end{equation}
    Furthermore, with $\ell_*$ given by \eqref{ell}, we have
    \begin{equation}
    h_n(t, N_n) = 2\pi \ee^{N_n\ell_*} \left\{  \begin{array}{ll}
     \dfrac{1}{2}(a_2(t) - b_2(t)), & n \in 2\N, \medskip \\  \dfrac{2}{a_2(t) - b_2(t)} , & n \in \N \setminus 2\N.
    \end{array}\right\} + \Oo\left( n^{-1} \right) \qasq n\to \infty.
    \label{eq:hn-asymp-final}
    \end{equation}
    Finally, for all $n \in \N$ large enough, we have 
    \begin{align}
        \mf p_{2n, 2n - 2}(t, N_{2n}) &= {nt} +  \dfrac{(a_2 - b_2)^2}{8 } + \Oo \left( n^{-1} \right) \qasq n\to \infty, \label{eq:sub-leading-1}\\
    \mf p_{2n-1, 2n - 3}(t, N_{2n}) &= {nt} + \dfrac{3a_2^2 + 2a_2 b_2 +3b_2^2}{8} + \Oo\left( n^{-1} \right) \qasq n\to \infty. \label{eq:sub-leading-2}
    \end{align}
\end{theorem}
% Corollary \ref{cor:asymp-p4} is an immediate consequence of Theorems \ref{thm:main2} and \ref{thm:asymptotics-gamma-h}. 

\begin{remark}
A particular instance of \eqref{gamma-n-final} has already appeared in the literature; when $t \in (-\infty, -2)$, \eqref{gamma-n-final} reduces to
\begin{equation*}
\label{gamma-n-alg}
\gamma^2_n(t, N_n) = \left \{ \begin{array}{ll}\dfrac{- t - \sqrt{t^2 - 4 }}{2}, & n \in 2\N, \medskip \\ \dfrac{- t + \sqrt{t^2 - 4 }}{2}, & n \in \N \setminus 2\N, \end{array}\right\} +  \Oo\left( n^{-1} \right) \qasq n\to \infty,
\end{equation*}
which agrees with \cite{MR1715324}*{Theorem 1.1} for all $n \in \N$ large enough. 
\end{remark}
Next, we will use Theorem \ref{thm:asymptotics-gamma-h} to obtain asymptotics of special function solutions of Painlev\'e-IV. 
%%%%%%%%%%%%%%%%%%%%%%%%%%
%%%%%%%%%%%%%%%%%%%%%%%%%%
\section{Application to Painlev\'e-IV}
\label{sec:p4}
To state the consequences of Theorem \ref{thm:asymptotics-gamma-h} in terms of solutions to Painlev\'e-IV, we recall two useful reformulations.
\subsection{Symmetric form}
To state our results, we will need to recall two reformulations of \eqref{eq:p4-ab}. The first is as a symmetric system which first appeared in \cite{MR1251164} and has since been used in various works; see e.g. \cites{Adler, N, FW}. We will follow the notation of \cite{FW}.
\begin{theorem}[\cite{MR1251164}]
    Let $\{\alpha_j\}_{j = 0}^2 \subset \C$ be arbitrary constants satisfying 
    \begin{equation}
        \alpha_0 + \alpha_1 + \alpha_2 = 1.
        \label{eq:alpha-plane}
    \end{equation}
    The system 
    \begin{equation}
    \begin{aligned}
        \dod{f_0}{x} &= f_0(f_1 - f_2) + 2\alpha_0, \\
        \dod{f_1}{x} &= f_1(f_2 - f_0) + 2\alpha_1,\\
        \dod{f_2}{x} &= f_2(f_0 - f_1) + 2\alpha_2,
    \end{aligned}
    \label{eq:p4-sym}
    \tag{$\psym{4}$}
    \end{equation}
    subject to the constraint 
    \begin{equation}
        f_0(x) + f_1(x) + f_2(x) = 2x,
        \label{eq:p4-sym-constraint}
    \end{equation}
    is equivalent to \eqref{eq:p4-ab}. That is, for $j = 0, 1, 2$, $u(x) = -f_j(x)$ solves \eqref{eq:p4-ab} with parameters\footnote{Here, the indices are understood cyclically, i.e. $\alpha_3 \equiv \alpha_0$.} 
    \begin{equation*}
        \Theta_\infty= \frac{1}{2}\left(\alpha_{j - 1} - \alpha_{j + 1} +1\right), \quad \Theta_0^2 = \frac{1}{4}\alpha_j^2.
        \label{eq:a-b}
    \end{equation*}
    Furthermore, given a solution of \eqref{eq:p4-ab}, $(u(x), \Theta_0, \Theta_\infty)$, there exists a triple $(f_0(x), f_1(x), f_2(x))$ and parameters $(\alpha_0, \alpha_1, \alpha_2)$ satisfying \eqref{eq:p4-sym} such that $u(x) = -f_1(x)$.
    \label{thm:sym}
\end{theorem}
It is clear that permuting the indices of $(f_0(x), f_1(x), f_2(x))$ maps a triple which solves \eqref{eq:p4-sym} to another triple which solves the same system but with correspondingly permuted parameters $\alpha_j$. Less trivially, the following was observed in \cite{MR1251169}.
\begin{proposition}[\cite{MR1251169}]
    Let $\{\alpha_j\}_{j = 0}^2$ and $\{f_j(x)\}_{j = 0}^2$ be as above. Define the Cartan matrix $\mb A$ and orientation matrix $\mb U$ by
    \begin{equation*}
        \mb A = [a_{ij}]_{i, j = 0}^2 = \left[\begin{array}{rrr}
            2 & -1 & -1\\ -1 & 2 & -1 \\ -1 & -1 & 2
        \end{array} \right], \quad \mb U = [u_{ij}]_{i, j = 0}^2= \left[\begin{array}{rrr}
            0 & 1 & -1\\ -1 & 0 & 1 \\ 1 & -1 & 0
        \end{array} \right], 
        \label{eq:cartan}
    \end{equation*}
    and consider the functions given by transformations
    \begin{equation*}
       s_i(f_j(x)) = f_j(x) + u_{ij}\dfrac{2\alpha_i}{f_i(x)}, \quad i, j \in \{0, 1, 2\}.
    \end{equation*}
    Then, the triples $s_i(f_0(x), f_1(x), f_2(x)) := (s_i(f_0(x)), s_i(f_1(x)), s_i(f_2(x)))$ satisfy \eqref{eq:p4-sym-constraint} and solve \eqref{eq:p4-sym} with parameters \sloppy $(s_i(\alpha_0), s_i(\alpha_1), s_i(\alpha_2))$ given by\footnote{We slightly abuse notation by using $s_i$ to denote the transformation of both the functions $f_j$ and parameters $\alpha_j$.\label{footnote:notation-abuse}} 
    \begin{equation*}
        s_i(\alpha_j) := \alpha_j - a_{ij} \alpha_i.
    \end{equation*}
\end{proposition}
To construct families of special solutions, we will use the following transformations which were introduced in \cite{NY}: 
\begin{equation*}
    T_1 := \pi s_2 s_1 , \quad T_2 := s_1 \pi s_2 , \quad T_3 :=  s_2 s_1\pi,
    \label{eq:T-def}
\end{equation*}
where $\pi$ denotes the shift
\begin{equation*}
    \pi(f_j(x)) := f_{j-1}(x), \quad \pi(\alpha_j) = \alpha_{j-1}.
\end{equation*}
\begin{table}
    \centering
    \begin{tabular}{c|ccc}
         & $\alpha_0$ & $\alpha_1$ & $\alpha_2$ \\ \hline
        $T_1$ & $\alpha_0 - 1$ & $\alpha_1 + 1$ & $\alpha_2$ \\
        $T_2$ & $\alpha_0$ & $\alpha_1 - 1$ & $\alpha_2 + 1$ \\
        $T_3$ & $\alpha_0 + 1$ & $\alpha_1$ & $\alpha_2 - 1$\\
    \end{tabular}
    \caption{Action of $T_i$ on parameters $\alpha_j$.}
    \label{tab:T-parameters}
\end{table}
Transformations $T_j$ map triples $(f_0(x), f_1(x), f_2(x))$ which solve \eqref{eq:p4-sym} and satisfy \eqref{eq:p4-sym-constraint} to triples which again satisfy \eqref{eq:p4-sym} and \eqref{eq:p4-sym-constraint}. It is easy to check that these transformations commute, and their effect on the parameters $\alpha_j$ is summarized in Table \ref{tab:T-parameters}. To give the reader a sense of what these iterates look like, we find it instructive to include the following example.
\begin{example}
    Given a triple $(f_0(x), f_1(x), f_2(x))$ which solves \eqref{eq:p4-sym}, a direct calculation yields 
    \begin{multline}
         T_1(f_0(x), f_1(x), f_2(x)) \\ =\left( f_{2}(x) + \dfrac{2 \alpha_1}{f_1(x)}, f_0(x) - \dfrac{2\alpha_1}{f_1(x)} + \dfrac{2(\alpha_1 + \alpha_2)f_1(x)}{(f_1f_2)(x) + 2\alpha_1}, f_1(x) -  \dfrac{2(\alpha_1 + \alpha_2)f_1(x)}{(f_1f_2)(x) + 2\alpha_1} \right),
         \label{eq:T1-f}
    \end{multline}
which solves \eqref{eq:p4-sym} with parameters $(\alpha_0 - 1, \alpha_1 + 1, \alpha_2)$. In a similar fashion, one can check that 
\begin{multline*}
    T_1^{-1}(f_0(x), f_1(x), f_2(x)) \\
    = \left( f_1(x) + \dfrac{2\alpha_0}{f_0(x)} - \dfrac{2(\alpha_0 + \alpha_2)f_0(x)}{(f_0f_2)(x) - 2\alpha_0}, f_2(x) - \dfrac{2\alpha_0}{f_0(x)}, f_0(x) + \dfrac{2(\alpha_0 + \alpha_2)f_0(x)}{(f_0f_2)(x) - 2\alpha_0} \right)
    \label{eq:T1-inv}
\end{multline*}
solves \eqref{eq:p4-sym} with parameters $(\alpha_0 + 1, \alpha_1 - 1, \alpha_2)$. 
\end{example}

\begin{remark} \label{remark:indeterminacy}
    It can occur that either $f_1(x) = 0$ or $(f_1f_2)(x) + 2\alpha_1 = 0$ holds identically, which causes an indeterminacy in \eqref{eq:T1-f}. We highlight that the requirement that $T_1(f_0, f_1, f_2)$ be determinate is really a condition on $f_1$. Indeed, given that $(f_0, f_1, f_2)$ satisfies \eqref{eq:p4-sym} and \eqref{eq:p4-sym-constraint}, $f_2$ is given in terms of $f_1$ as
    \begin{equation}
    f_2(x) = \dfrac{1}{2f_1(x)}\left(-2\alpha_1 + 2xf_1(x) - f_1^2(x) + f_1'(x) \right).
        \label{eq:f-2}
    \end{equation}
    Hence, for $T_1(f_0, f_1, f_2)$ to be determinate we must require $f_1(x) \not \equiv 0$ and  
    \begin{equation}
     f_1'(x) - f_1^2(x) + 2xf_1(x) +2\alpha_1 \not \equiv 0.
    \label{eq:T1-det} 
    \end{equation}
    In the following section we will check (see Proposition \ref{prop:determinate} below) that the particular family of solutions we study in this paper is determinate.
\end{remark}

\subsection{Hamiltonian form}
The second reformulation of \eqref{eq:p4-ab} we use is as a Hamiltonian system:
    \begin{align}
    \dod{q}{x} = \dpd{H}{p},  \quad  \dod{p}{x} = - \dpd{H}{q}, 
    \label{eq:H-system}
    \end{align}
Each of the Painlev\'e equations admit such a representation with a polynomial Hamiltonian, see e.g. \cite{MR0581468}. In the case of \eqref{eq:p4-ab} the Hamiltonian is given by \cite{FW}*{Equation (2.13)}
\begin{equation}
    \begin{aligned}
    H_{IV}(x; p,q; \alpha_1, \alpha_2) = (2p - q - 2x)pq - 2\alpha_1p - \alpha_2 q.
    \label{eq:H}
    \end{aligned}
\end{equation}
It is a simple calculation to verify that, for a given $\alpha_1, \alpha_2$, if $p(x), q(x)$ satisfy \eqref{eq:H-system} with Hamiltonian \eqref{eq:H} and $\alpha_0$ is defined by \eqref{eq:alpha-plane}, then setting
\begin{equation*}
    f_1(x) = -q(x), \quad f_2(x) = 2p(x), \quad f_0(x) = 2x +q(x) - 2p(x),
    \label{eq:q-p-f}
\end{equation*}
yields a solution of \eqref{eq:p4-sym} which satisfies constraint \eqref{eq:p4-sym-constraint}. The Hamiltonian evaluated at such a solution defines the \emph{Jimbo-Miwa-Okamoto $\sigma$-function}, i.e.
\begin{equation}
    \sigma(x) := H(x; p(x), q(x); \alpha_1, \alpha_2) = \frac12 (f_0f_1f_2)(x) + \alpha_2 f_1(x) - \alpha_1f_2(x).
    \label{eq:sigma-f}
\end{equation}
Function \eqref{eq:sigma-f} is known to satisfy the so-called $\sigma$-form of Painlev\'e-IV (see e.g. \cite{FW}*{Proposition 4}) given by 
\begin{equation*}
    \left(\dod[2]{\sigma}{x}\right)^2 - 4\left(x \dod{\sigma}{x} - \sigma(x)\right) + 4\dod{\sigma}{x} \left(\dod{\sigma}{x} + 2\alpha_1\right)\left(\dod{\sigma}{x} - 2\alpha_2\right) = 0.
    \label{eq:p4-sigma}
\end{equation*}
Solutions of \eqref{eq:p4-ab} are known to be meromorphic functions in $\C$ \cite{GLS}*{Theorem 4.1} and in view of \eqref{eq:sigma-f}, it follows that $\sigma(x)$ is a meromorphic function in $\C$ as well. A simple local analysis of \eqref{eq:p4-sigma} implies that poles of $\sigma(x)$ are all simple and with residue $1$. Therefore, the \emph{Jimbo-Miwa-Okamoto $\tau$-function}, defined (up to a constant) by 
\begin{equation}
    \sigma(x) =: \dod{}{x} \log \tau(x),
    \label{eq:tau-def}
\end{equation}
is an entire function with at most simple zeros. In the sequel, we will be interested in iteratively applying transformation $T_j$ to a ``seed" triple $((f_0)_{0,0,0}(x), (f_1)_{0,0,0}(x), (f_2)_{0,0,0}(x))$ and recording the resulting $\sigma$- and $\tau$-functions. To this end, we will use the notation 
\begin{equation}
    T_3^kT_2^mT_1^{n}(K_{0,0,0}(x)) = K_{m, n, k}(x), \quad K \in \{f_0, f_1, f_2, \sigma, \tau\}.
    \label{eq:iterate-notation}
\end{equation}
The constant of integration in \eqref{eq:tau-def} can be chosen so that the well-known Toda equations hold. 
\begin{proposition}[\cite{O}, see also \cite{FW}*{Proposition 5}] \label{prop:toda}
    Let $\tau_{0,0,0}(x)$ be an arbitrary $\tau$-function defined by \eqref{eq:tau-def} and chosen so that for a given $n, m, k \in \Z$, the functions $\tau_{n, m, k}(x), \tau_{n \pm 1, m, k}(x)$ are well-defined. Then the functions $ \tilde{\tau}_{n,m,k}(x) := \ee^{x^2(\alpha_1 + n)} \tau_{n,m,k}(x)$ satisfy the Toda equation:
    \begin{equation}
        \dod[2]{}{x}\log \tilde{\tau}_{n,m,k}(x) = \dfrac{\tilde{\tau}_{n+1,m,k}(x)\tilde{\tau}_{n-1,m,k}(x)}{\tilde{\tau}^2_{n,m,k}(x)}.
        \label{eq:toda}
    \end{equation}
    Furthermore, if $\tau_{n, m\pm 1, k}(x)$ are well-defined, then there exists $K_{n, m, k} \in \C \setminus \{0\}$ such that 
    \begin{equation*}
        \dod[2]{}{x}\log {\tau}_{n,m,k}(x) = K_{n,m,k}\dfrac{{\tau}_{n,m+1,k}(x){\tau}_{n,m-1,k}(x)}{{\tau}^2_{n,m,k}(x)}.
        \label{eq:toda-T2}
    \end{equation*}
\end{proposition}
\begin{remark}
As noted above, functions $\tau_{n, m, k}(x)$ are entire with only simple zeros. Writing out \eqref{eq:toda} in the form 
\[
\tilde{\tau}_{n,m,k}''\tilde{\tau}_{n,m,k} - (\tilde{\tau}_{n,m,k}')^2 = \tilde{\tau}_{n + 1,m,k} \tilde{\tau}_{n - 1,m,k},
\]
one can see that $\tilde{\tau}_{n,m,k}(x), \tilde{\tau}_{n+1,m,k}(x)$ (and hence ${\tau}_{n,m,k}(x), {\tau}_{n+1,m,k}(x)$) do not share any common roots. Similar considerations imply the analogous fact about $\tau_{n,m,k}, \tau_{n,m+1,k}$ and $\tau_{n,m,k}, \tau_{n,m,k+1}$.
\label{remark:tau-zeros}
\end{remark}

It is well-known (see \cite{GLS} and references therein) that \eqref{eq:p4-ab} possesses solutions which can be written in terms of parabolic cylinder functions. One way to see this is to require that \eqref{eq:p4-ab} be consistent with a Riccati equation
\begin{equation}
    \dod{u}{x} = k_2(x) u^2(x) + k_1(x) u(x) + k_0(x).
    \label{eq:ric-1}
\end{equation}
The compatibility of \eqref{eq:ric-1} with \eqref{eq:p4-ab} determines the coefficients $k_j(x)$ up to some choices of signs and one finds that \eqref{eq:ric-1} can be solved in terms of parabolic cylinder functions. For the remainder of the paper, we will be working with the particular seed triple
\begin{equation}
    (f_0(x), f_1(x), f_2(x))_{0,0,0} = \left(0,  \frac{\varphi'(x)}{\varphi(x)} , 2x -\frac{\varphi'(x)}{\varphi(x)} \right), \quad \varphi(x) = \ee^{\frac12 x^2} D_{-\frac12}(\sqrt{2} x), 
    \label{eq:seed}
\end{equation}
which solves \eqref{eq:p4-sym} with parameters 
\begin{equation}
     (\alpha_0, \alpha_1, \alpha_2) = \left(0, \frac12, \frac12 \right)
     \label{eq:seed-parameters}
\end{equation}
and satisfies \eqref{eq:p4-sym-constraint}. We now address the indeterminacy problem raised in Remark \ref{remark:indeterminacy}. 

\begin{proposition}[\cite{BM}*{Proposition 11}]
    Let $\left((f_0)_{0, 0, 0}, (f_1)_{0, 0, 0}, (f_2)_{0,0, 0} \right)$ and $(\alpha_0, \alpha_1, \alpha_2)$ be as in \eqref{eq:seed}, \eqref{eq:seed-parameters}, respectively, and recall notation \eqref{eq:iterate-notation}. Then, \sloppy $\left((f_0)_{n, 0, 0}, (f_1)_{n, 0, 0}, (f_2)_{n,0, 0} \right)$ and $\left((f_0)_{n, -1, 0}, (f_1)_{n, -1, 0}, (f_2)_{n, -1, 0} \right)$ are well-defined for all $n\geq 0$. 
    \label{prop:determinate}
\end{proposition}
\begin{proof}
    Let $u(x) := -(f_1)_{0,0,0}(x)$. By Theorem \ref{thm:sym}, $u(x)$ solves \eqref{eq:p4-ab} with parameters $\Theta_\infty = \frac{1}{4}$ and $\Theta_0^2 = \frac{1}{16}$. Choosing $\Theta_0 = \frac14$ and letting $u_{\searrow}(x) := -T_1((f_1)_{0,0,0}(x))$, a direct computation using \eqref{eq:T1-f} and \eqref{eq:p4-sym-constraint} yields 
    \begin{equation}
        u_{\searrow}(x) = \dfrac{(2\alpha_1)^2 -8(\alpha_1 + \alpha_2)u^2(x) - 4x^2u^2(x) - 4xu^3(x)-u^4(x)-4\alpha_1u'(x) + u'(x)^2}{2u(x)(-2\alpha_1 +2xu(x) + u^2(x) + u'(x))},
        \label{eq:T1-u}
    \end{equation}
    which, upon recalling the definitions of $\alpha_i$'s and the constraint \eqref{eq:alpha-plane}, agrees with \cite{BM}*{Equation 3.50}. Observe that if $u_{\searrow}(x)$ is well-defined, then so is the triple $T_1((f_0)_{0,0,0}, (f_1)_{0,0,0}, (f_2)_{0,0,0})$. Indeed, the conditions $(f_1)_{0,0,0} \not \equiv 0$ and \eqref{eq:T1-det} are equivalent to the denominator of \eqref{eq:T1-u} not vanishing identically. Since $\Theta_0(\Theta_\infty - \Theta_0 - 1) \neq 0$, by \cite{BM}*{Proposition 11}, we have that $u_{\searrow}(x)$ is well-defined and solves \eqref{eq:p4-ab} with parameters
    \[
    \Theta_{\infty, \searrow} = \Theta_\infty - \frac12, \quad \Theta_{0, \searrow} = \Theta_0 + \frac12.
    \]
    The first claim now follows by inductively applying \cite{BM}*{Proposition 11} since, for $n \geq 0$,
    \[
    \left(\Theta_0 + \frac{n}{2} \right)(\Theta_\infty - \Theta_0 - 1) \neq 0.
    \]
    For the second claim, we first note that by \cite{DLMF}*{12.8.1-2},
    \begin{equation}
    \left(f_0(x), f_1(x), f_2(x) \right)_{0, -1, 0} = \left( 0, 2x + \dfrac{\phi'(x)}{\phi(x)} , -\dfrac{\phi'(x)}{\phi(x)} \right), \quad \phi(x) = \ee^{\frac12 x^2}D_{-\frac32}(\sqrt{2} x),
    \label{eq:seed-2}
    \end{equation}
    solves \eqref{eq:p4-sym} with parameters $(0, \frac32, -\frac12)$ and satisfies \eqref{eq:p4-sym-constraint}. This implies that $-(f_1)_{0, -1, 0}(x)$ solves \eqref{eq:p4-ab} with $\Theta_0 = \Theta_\infty = \frac34$. Repeating the argument above yields the result.
\end{proof}

\begin{remark}
    A given solution of \eqref{eq:p4-ab} $(u(x),  \Theta_0, \Theta_\infty)$ can be embedded into at least two distinct pairs of solutions of \eqref{eq:p4-sym} by setting $u(x) = -f_1(x)$ and making the choice $\alpha_1 = \pm 2\Theta_0$. This then fixes the triple $(f_0, f_1, f_2)$ uniquely by \eqref{eq:f-2} and \eqref{eq:p4-sym-constraint}. In the proof of Proposition \ref{prop:determinate}, the effect of making the choice $\Theta_0 = -\frac{1}{4}$ is that $u_{\swarrow}(x) := -T_1(f_1(x))$ solves \eqref{eq:p4-ab} with parameters
    \[
    \Theta_{\infty, \swarrow} = \Theta_\infty - \frac12, \quad \Theta_{0, \swarrow} = \Theta_0 - \frac12.
    \]
    This again agrees with \cite{BM}*{Equation 3.49}.
\end{remark}

The $\tau$-functions corresponding to the triples \eqref{eq:seed} and \eqref{eq:seed-2} can be calculated explicitly. Indeed, from \eqref{eq:tau-def} we have $\tau_{0,0,0}(x) := \tau(x) = \ee^{-\alpha_1 x^2} \varphi(x)$ or $\tilde{\tau}_{0,0,0}(x) = \varphi(x)$. Setting $\tilde{\tau}_{-1,0,0}(x) \equiv 1$\footnote{It can be directly verified that this is consistent with applying $T_1^{-1}$ to the triple \eqref{eq:seed}}, the Desnanot-Jacobi identity implies that the functions
    \begin{equation}
    \tilde{\tau}_{n,0,0}(x) = \det \left[ \dod[i+j]{}{x}\left( \ee^{\frac12 x^2}D_{-\frac12}(\sqrt{2} x) \right)\right ]_{i, j = 0}^{n}
    \label{eq:tau-n-special}
    \end{equation}
satisfy \eqref{eq:toda}. Similarly, recalling \eqref{eq:seed-2}, the definition of $\tilde{\tau}_{n, m, k}(x)$, and that $T_j's$ commute, we find
\begin{equation}
    \tilde{\tau}_{n, -1, 0}(x) = T_2^{-1}T_1^n(\tilde{\tau}_{0,0,0}(x)) = T_1^n(\tilde{\tau}_{0,-1,0}(x)) = \det \left[ \dod[i+j]{}{x}\left( \ee^{\frac12 x^2}D_{-\frac32}(\sqrt{2} x) \right)\right ]_{i, j = 0}^n. 
    \label{eq:tau-n-special-2}
\end{equation}
This particular family appears as recurrence coefficients of the orthogonal polynomials satisfying \eqref{eq:ortho}, as can be seen from the following theorem. 

\begin{theorem}[\cite{MR3494152}*{Section 4.2}]
    \label{thm:main2}
    Let $\left((f_0)_{0, 0, 0}, (f_1)_{0, 0, 0}, (f_2)_{0,0, 0} \right)$ and $(\alpha_0, \alpha_1, \alpha_2)$ be as in \eqref{eq:seed}, \eqref{eq:seed-parameters}, respectively. With notation as in \eqref{eq:iterate-notation}, we have
    \begin{align}
        % \tau_{n-1, 0, 0}(x) &= \ee^{-(n-\frac12)x^2} H_n^{(e)}(2N^{-\frac12}x, N) = \ee^{-(n-\frac12)x^2} H_n^{(e)}(2x, 1) \label{eq:tau-n-0-0}, \medskip \\
        % \tau_{n-1, -1, 0}(x) &= \ee^{-(n+\frac12)x^2} H_n^{(o)}(2N^{-\frac12}x, N) = \ee^{-(n+\frac12)x^2} H_n^{(o)}(2x, 1)\label{eq:tau-n-1-0},\medskip \\
        (f_0)_{n, 0, 0}(x) &= - N^{\frac12}\gamma_{2n}^2(2N^{-\frac12}x, N) = - \gamma_{2n}^2(2x, 1), \label{eq:f0-n-0-0} \medskip\\
        (f_1)_{n, 0, 0}(x) &= - N^{\frac12}\gamma_{2n+1}^2(2N^{-\frac12}x, N) = - \gamma_{2n+1}^2(2x, 1), \label{eq:f1-n-0-0} \medskip\\
        (f_2)_{n, 0, 0}(x) &= 2x + N^{\frac12}(\gamma_{2n}^2(2N^{-\frac12}x, N) +\gamma_{2n+1}^2(2N^{-\frac12}x,N) \label{eq:f2-n-0-0} \medskip\\
        &= 2x + \gamma_{2n}^2(2x, 1) +\gamma_{2n+1}^2(2x,1).\notag
    \end{align}
    Furthermore, writing\footnote{Since $V(z;t)$ is an even function, we have $P_n(z;t, N)$ are even(odd) functions whenever $\deg P_n$ is even(odd)}$$P_n(z; t, N) := z^n + \sum_{k = 1}^{\lfloor \frac{n}{2} \rfloor} \mathfrak{p}_{n, n - 2k}(t, N) z^{n - 2k},$$ we have
    \begin{align}
        \sigma_{n-1, 0, 0}(x) &= N^{\frac12} \mf p_{2n, 2n-2}(2N^{-\frac12}x, N) - (2n-1)x \label{eq:sigma-n-0-0} \medskip \\
        &= \mf p_{2n, 2n-2}(2x, 1) - (2n-1)x, \notag \medskip\\
        \sigma_{n-1, -1, 0}(x) &= N^{\frac12} \mf p_{2n+1, 2n-1}(2N^{-\frac12}x, N) - (2n+1)x\label{eq:sigma-n-1-0} \medskip \\
        &=  \mf p_{2n+1, 2n-1}(2x, 1) - (2n+1)x. \notag
    \end{align}
\end{theorem}
\begin{proof}
    We first note that a re-scaling in \eqref{eq:ortho} yields 
\begin{equation*}
        P_n(z; t, N) = N^{-\frac{1}{4}\deg P_n} P_n(N^{\frac14}z; N^{\frac12} t, 1).
        \label{eq:pn-rescaled}
\end{equation*}
This implies 
\begin{equation*}
\begin{aligned}
    %H_n^{(e, o)}(2N^{-\frac12}t, N) &= H_n^{(e, o)}(2t, 1), \\
    N^{\frac12}\gamma_{n}^2(2N^{-\frac12}t, N) &= \gamma_{n}^2(2t, 1),\\
     N^{\frac12} \mf p_{n, n-2}(2N^{-\frac12} t, N) &= \mf p_{n, n-2}(2t, 1).
\end{aligned}
\label{eq:scaled-quantities}
\end{equation*}
%Equation \eqref{eq:tau-n-0-0} (resp. \eqref{eq:tau-n-1-0}) follows by simply comparing \eqref{eq:He-def} (resp. \eqref{eq:Ho-def}) with \eqref{eq:tau-n-special} (resp. \eqref{eq:tau-n-special-2}). 
The proof of \eqref{eq:f0-n-0-0} - \eqref{eq:f2-n-0-0} is an exercise in comparing notation with \cite{MR3494152} which we now carry out. Writing the three-term recurrence relation for $S_n(x; t, \lambda)$ as
\[
xS_{n}(x; t, \lambda) = S_{n+1}(x; t, \lambda) + \beta_n(t, \lambda) S_{n - 1}(x;t, \lambda),
\]
relation \eqref{eq:freud-connection} implies
\begin{equation*}
    \gamma_n^2(t, 1) = 2\beta_n\left(-t, -\frac{1}{2}\right).
\end{equation*}
This and \cite{MR3494152}*{Lemma 4} imply \eqref{eq:f0-n-0-0} - \eqref{eq:f2-n-0-0}. Finally, analogs of \eqref{eq:sigma-n-0-0} - \eqref{eq:sigma-n-1-0} are not explicitly stated in \cite{MR3494152}, but are easily deduced. Setting 
$$S_n\left(z; t, -\frac12 \right) := z^n + \sum_{k = 1}^{\lfloor \frac{n}{2} \rfloor} \mathfrak{s}_{n, n - 2k} \left(t,-\frac12 \right) z^{n - 2k},$$
then it follows from \eqref{eq:freud-connection} that 
\begin{equation}
    \mathfrak{p}_{n, n-2}(t, 1) = 2\mathfrak{s}_{n, n-2}\left(-t, -\frac12 \right). 
    \label{eq:p-s}
\end{equation}
Furthermore, the three-term recurrence relation implies
\begin{equation}
    \beta_n\left(t, -\frac12 \right) = \mathfrak{s}_{n, n-2}\left(t, -\frac12 \right) - \mathfrak{s}_{n+1, n-1}\left(t, -\frac12 \right).
    \label{eq:beta-subleading}
\end{equation}
Using \eqref{eq:beta-subleading}, we have the identities
\begin{equation}
    \sum_{j = 0}^n \beta_j\left(t, -\frac12 \right) = -\mathfrak s_{n+1, n-1}\left(t, -\frac12 \right).
    \label{eq:beta-sum}
\end{equation}
Combining \eqref{eq:p-s}, \eqref{eq:beta-sum} with \cite{MR3494152}*{Lemma 5} yields \eqref{eq:sigma-n-0-0} - \eqref{eq:sigma-n-1-0}.
\end{proof}

An immediate corollary of Theorem \ref{thm:asymptotics-gamma-h} and Theorem \ref{thm:main2} is the following.
\begin{corollary}
    \label{cor:asymp-p4}
    Let $\left((f_0)_{0, 0, 0}, (f_1)_{0, 0, 0}, (f_2)_{0,0, 0} \right)$ and $(\alpha_0, \alpha_1, \alpha_2)$ be as in \eqref{eq:seed}, \eqref{eq:seed-parameters}, respectively and recall \eqref{eq:iterate-notation}. Then, as $n \to \infty$, 
    \begin{equation*}
    \begin{aligned}
        2^{\frac12}n^{-\frac12}(f_0)_{n, 0,0} \left(2^{-\frac12}n^{\frac12}x \right) &=  x + \sqrt{x^2 - 4} + \Oo\left( n^{-1} \right), \\
        2^{\frac12}n^{-\frac12}(f_1)_{n, 0,0} \left(2^{-\frac12}n^{\frac12}x \right) &= \dfrac{4}{x + \sqrt{x^2 - 4}} + \Oo\left(  n^{-1} \right), \\
        2^{\frac12}n^{-\frac12}(f_2)_{n, 0,0} \left(2^{-\frac12}n^{\frac12}x \right) &=  \Oo\left(  n^{-1} \right),
        \label{eq:fi-thm}
    \end{aligned}
    \end{equation*}
    and
    \begin{equation*}
    \begin{aligned}
        2^{\frac12}n^{-\frac12}\sigma_{n-1, 0, 0} \left(2^{-\frac12}n^{\frac12}x \right) &= \dfrac{x - \sqrt{x^2 - 4}}{2} + \Oo\left(  n^{-1} \right), \\
        2^{\frac12}n^{-\frac12}\sigma_{n-1, -1, 0} \left(2^{-\frac12}n^{\frac12}x \right) &= \dfrac{-x + \sqrt{x^2-4}}{2}+ \Oo\left(  n^{-1} \right),
        \label{eq:sigma-thm}
    \end{aligned}
    \end{equation*}
    locally uniformly for $x \in \Oo_2$. In all cases, $\sqrt{x^4 - 4}$ is the branch analytic in $\C\setminus [-2, 2]$ and satisfies $\sqrt{x^2-4} \sim -x$ as $x \to \infty$.
\end{corollary}
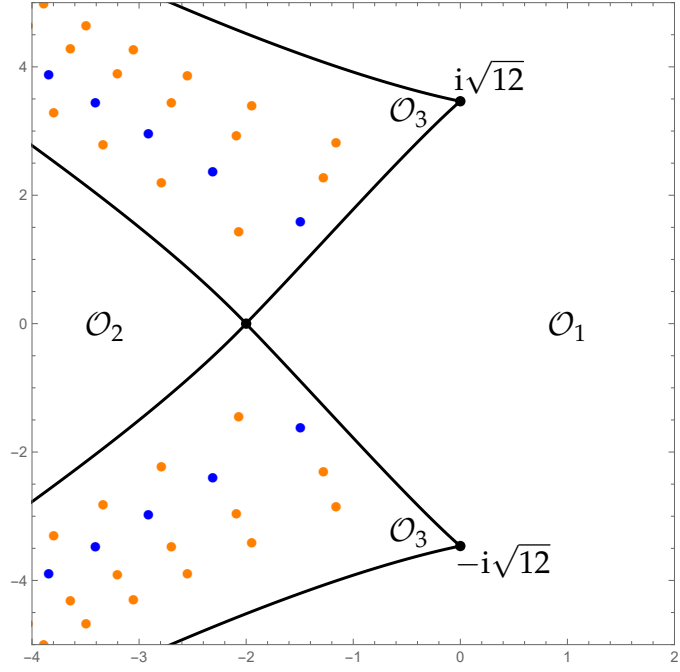
\begin{figure}[t]
    \centering
    \includegraphics[width = 0.5 \textwidth]{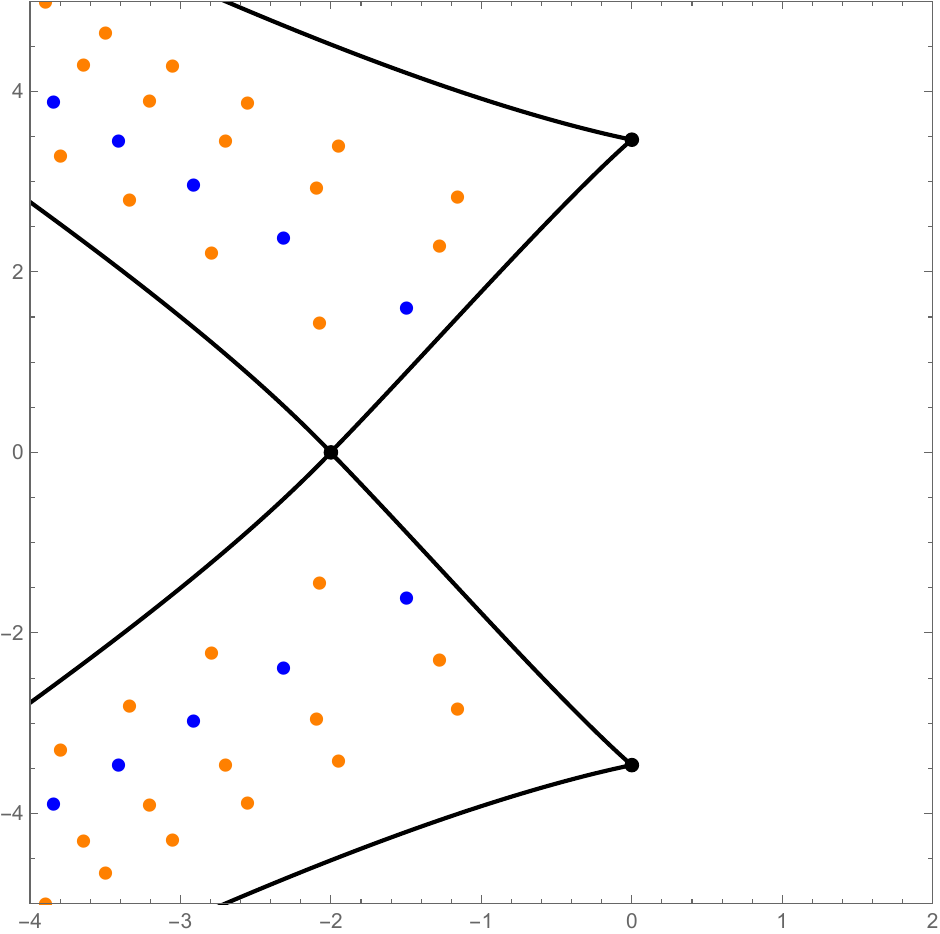}
    \put(-50,125){$\Oo_1$}
    \put(-225,125){$\Oo_2$}
    \put(-110,206){$\Oo_3$}
    \put(-85,218){$\ii \sqrt{12}$}
    \put(-85,35){$-\ii \sqrt{12}$}
    \put(-110,47){$\Oo_3$}
    \caption{Poles of $(f_1)_{n, 0, 0}(x)$ for $n = 0$ (blue) and $n = 1$ (orange). The black contours are the boundary of the open sets $\Oo_1, \Oo_2, \Oo_3$.} 
    \label{fig:polest}
\end{figure}
An obvious consequence of Corollary \ref{cor:asymp-p4} is that for $n$ large enough, the triple $\left((f_0)_{0, 0, 0}, (f_1)_{0, 0, 0}, (f_2)_{0,0, 0} \right)$ is free of poles in the region $2^{-1/2}n^{1/2}\Oo_2$, cf. Figure \ref{fig:polest}. 

The remainder of the paper is dedicated to proving Theorem \ref{thm:asymptotics-poly} and Theorem \ref{thm:asymptotics-gamma-h}.

\section{\texorpdfstring{On functions $g(z;t)$ and $\mc D(z;t)$}{the g-function and Szeg\H{o} function}}
\label{sec:g-szego}
In this section, we recall useful properties of $g(z;t)$ defined in \eqref{g-def} and prove Proposition \ref{Szego-prop}. 

\subsection{The \texorpdfstring{$g$-function}{g-function} in the two-cut case} 
\label{sec:two-cut-g}
Combining \eqref{em5}, \eqref{Q-two-cut}, and \eqref{resolvant}, we conclude that there exists a constant $\ell_*(t)$ such that \footnote{Comparing \eqref{g-V-eta} with \eqref{euler-lagrange-g}, we note that $\Re( \ell_*(t)) =-2\ell_t$. In \cite{BGM}, $\ell_*(t)$ was denoted $\ell^{(2)}_*(t)$ and $\eta(z; t)$ was denoted $\eta_2(z; t)$.}
\begin{equation}
\label{g-V-eta}
    g(z;t) = \dfrac{V(z;t) + \ell_*(t) + \eta(z;t) }{2}, \quad z \in \C\setminus \Gamma_t (-\infty, b_2],
\end{equation}
where 
\begin{equation}
\label{eta-2}
\eta(z; t) := -2\int_{b_2}^z Q^{1/2}(x;t) \dd x, 
\end{equation}
and the path of integration is taken to be in $\C \setminus \Ga_t(-\infty, b_2]$. It will be convenient to introduce the following notation: for $e\in \{\pm a_2, \pm b_2\}$, let
\begin{equation}
\label{eta-e}
\eta_{e}(z) := -2 \int_{e}^z Q^{1/2}(x) \dd x, \quad 
\end{equation}
where the path of integration is taken in
\begin{itemize}
\item $\C \setminus \Ga_t(-\infty, b_2]$ when $e = b_2$, 
\item $\C \setminus \Ga_t[-b_2, \infty)$ when $e = -b_2$, 
\item $\C \setminus \Ga_t(-\infty, -a_2] \cup \Ga[a_2, \infty)$ when $e \in \{\pm a_2\}$.
\end{itemize}
Then, it follows from definition \eqref{eta-2} that $\eta(z) = \eta_{b_2}(z)$. Note that since $\mu_t$ in \eqref{em6} is a probability measure, and using the evenness of $Q^{1/2}_+(z;t)$, we have 
\[
\dfrac{1}{\pi \ii}\int_{J_{t, 1}}Q^{1/2}_+(x) \dd x =\dfrac{1}{\pi \ii} \int_{J_{t, 2}} Q^{1/2}_+(x) \dd x = \dfrac{1}{2}.
\] 
Using this, we can deduce that
\begin{equation}
\label{eta-e-b}
    \eta_{b_2}(z;t) = \left\{ \begin{array}{l} \eta_{a_2}(z;t) \pm \pi \ii, \\
    \eta_{-a_2}(z; t) \pm \pi \ii, \\
    \eta_{-b_2}(z;t) \pm 2\pi \ii \end{array}  \right. \quad z \in \C \setminus \Ga,
\end{equation}
where the plus/minus signs correspond to $z$ on the left/right of $\Ga_t$, respectively.

It turns out that $\eta(z;t)$ (and consequently, all $\eta_{e}(z;t)$) can be written in terms of elementary functions. Indeed, one can check that 
\begin{equation}
    \eta(z;t) = -\dfrac{1}{4} (z^2 + t) R^{1/2}(z) + \log \left( \dfrac{2z^2 + t + R^{1/2}(z)}{2} \right), \quad z \in \C \setminus \Ga_t(-\infty, b_2].
    \label{eq:g-explicit}
\end{equation}
Then, using \eqref{g-V-eta}, the $g$ functions can be written explicitly 
\begin{equation*}
    g(z;t) = \dfrac{1}{2} \left( \dfrac{t }{2}z^2 + \dfrac{1}{4}z^4 + \ell_*(t) \right) - \dfrac{1}{8} (z^2 + t) R^{1/2}(z) + \dfrac{1}{2} \log \left( \dfrac{2z^2 + t + R^{1/2}(z)}{2} \right).
\end{equation*}
Combining this expression with the requirement that 
\begin{equation}
\label{g-log-requirement}
    g(z;t) = \log z + \Oo(z^{-1})
\end{equation}
yields
\begin{equation}
\label{ell}
\ell_*(t) = \dfrac{t^2}{4} - \dfrac{1}{2}.
\end{equation}
Furthermore, $g(z;t)$ has the following jumps
\begin{equation}
    (g_+ - g_-)(z;t) = \left\{ \begin{array}{ll}
    0, & z \in \Gamma[b_2, \infty), \\ \eta_{ +}(z; t), & z\in \Gamma[a_2, b_2],\\ \pi \ii, & z \in \Ga[-a_2, a_2],\\ \eta_{+}(z;t), & z \in \Ga[-b_2, -a_2], \\ 2\pi \ii, & z \in \Ga(-\infty, -b_2],
    \end{array}\right.
\label{g-jump}
\end{equation}
and 
\begin{equation}
\left(g_+ + g_- -V - \ell_*\right)(z;t) = \left\{ \begin{array}{ll}
    \eta(z;t), & z \in \Gamma[b_2, \infty), \\ 0, & z\in \Gamma[a_2, b_2],\\ \eta_{ \pm}(z; t) \mp \pi \ii, & z \in \Ga[-a_2, a_2],\\ 0, & z \in \Ga[-b_2, -a_2], \\ \eta_{ \pm}(z; t) \mp 2\pi \ii, & z \in \Ga(-\infty, -b_2],
    \end{array}\right.
\label{g-V-jump}
\end{equation}
where the second-to-last line in both formulas uses the fact that 
\[
\int_{-a_2}^{a_2} x \sqrt{(x^2 - a_2^2)(x^2 - b_2^2)} \dd x = 0.
\]
The arcs $\Ga[a_2, b_2] \cup \Ga[-b_2, -a_2]$ are solutions of $\Re[\eta(z; t)] = 0$; in fact, it was shown in \cite{BGM}*{Section 3.2} that for all $t \in \mathcal{O}_2$, there exists a (non-unique, see Remark \ref{remark:contour-freedom}) choice of $\Gamma_t$ so that 
\begin{equation}
\label{re-eta}
\Re[\eta(z; t)] < 0, \quad z \in \Ga_t \setminus J_t.
\end{equation}
Furthermore, it follows from \eqref{em0}, \eqref{g-potential} , and the subharmonicity of $\mathcal{U}(z;t)$ that $\Re[\eta(z;t)]$ is subharmonic at any point $z_0 \in J_t$. We then deduce from the maximum principle for subharmonic functions that every $z \in J_{t, 1} \cup J_{t, 2}$ is on the boundary of the set $\Re[\eta(z;t)] > 0$, cf. Figure \ref{fig:eta-sign-chart}. 

\subsection{Proof of Proposition \ref{Szego-prop}}
\label{sec:szego-proof}
From \eqref{Q-R} it follows that $R^{1/2}(z;t)$ is analytic in $\C \setminus J_t$ and for $z \in J_t\setminus \{\pm a_2, \pm b_2\}$ we have $R^{1/2}(z; t)_+ = -R^{1/2}(z; t)_-$. This implies \eqref{Szego-jump}. At $z = \infty$, we note that by \eqref{em5} and the choice if the square root, we have 
    \begin{equation}
    Q^{1/2}(z, t) = \dfrac{V'(z, t)}{2} - \dfrac{1}{z} +  \Oo \left( {z^{-3}} \right) \qasq z \to \infty.
    \label{eq:Q-root-series}
    \end{equation}
    By definition of $R^{1/2}(z;t)$ and \eqref{eq:Q-root-series} we find 
    \begin{equation*}
        (z^2 + t) R^{1/2}(z;t) = 4V(z;t) + 4\ell_*(t) + \Oo \left( {z^{-2}} \right) \qasq z \to \infty.
    \end{equation*}
    Hence, $\mc D(z;t)$ has a convergent Laurent expansion centered at $z = \infty$ of the form \eqref{Szego-lim}.

\section{Proofs of Theorem \ref{thm:asymptotics-poly} and Theorem \ref{thm:asymptotics-gamma-h}} \label{sec:rh}

\subsection{Initial Riemann-Hilbert problem}
\label{sec:irhp}
Henceforth, we omit the dependence on \( t \) whenever it does not cause ambiguity. In what follows, we will denote matrices with bold, capitalized symbols (e.g. $\mathbf{Y}$), with the exception of the Pauli matrices:
\[
\mathbb{I}:=\begin{bmatrix} 1 & 0 \\ 0 & 1 \end{bmatrix}, \quad \sigma_1:=\begin{bmatrix} 0 & 1 \\ 1 & 0 \end{bmatrix},\quad  \sigma_2:=\begin{bmatrix} 0 & -\ii \\ \ii & 0 \end{bmatrix}\qandq \sigma_3:=\begin{bmatrix} 1 & 0 \\ 0 & -1 \end{bmatrix},
\]
and will use the notation 
\[
f^{\sigma_3} := \begin{bmatrix} f & 0 \\ 0 & f^{-1} \end{bmatrix}.
\]
We are seeking solutions of the following sequence of Riemann-Hilbert problems for $2\times2$ matrix functions (\rhy):
\begin{enumerate}
\label{rhy}
\item[(a)] $\mathbf Y_n(z)$ is analytic in $\C\setminus \Ga$ and $\lim_{\C\setminus\Ga \ni z\to\infty}\mathbf Y_n(z)z^{-n\sigma_3}=\mathbf I$;
\item[(b)] $\mathbf Y_n(z)$ has continuous boundary values on $\Ga$ that satisfy
\[
\mathbf Y_{n, +}(x) = \mathbf Y_{n,-}(x) \begin{bmatrix}1& \ee^{-N_nV(x)}\\0&1\end{bmatrix},
\]
where, as before, \( V(z) \) is given by \eqref{potential} and the sequence \( \{N_n\} \) is such that \( |n-N_n| < C \) for some \( C>0 \).
\end{enumerate}
If the solution of \hyperref[rhy]{\rhy} exists, then it is necessarily of the form
\begin{equation}
\label{rh1}
\mathbf Y_n(z) = \begin{bmatrix}
P_n(z) & \mathcal{C}\big(P_n\ee^{-N_nV}\big)(z) \medskip \\
-\frac{2\pi\ii}{h_{n-1}}P_{n-1}(z) & -\frac{2\pi\ii}{h_{n-1}}\mathcal{C}\big(P_{n-1}\ee^{-N_nV}\big)(z)
\end{bmatrix},
\end{equation}
where $P_n(z)=P_n(z;t,N_n)$ are the polynomial satisfying orthogonality relations \eqref{eq:ortho}, $h_n=h_n(t,N_n)$ are the constants defined in \eqref{eq:normalizing}, and $\mathcal{C}f(z)$ is the Cauchy transform of a function $f$ given on $\R$, i.e.,
\[
\mathcal{C}(f)(z) := \frac1{2\pi\ii}\int_\R\frac{f(x)}{x-z}\dd x.
\]
The connection of \hyperref[rhy]{\rhy} to orthogonal polynomials was first demonstrated by Fokas, Its, and Kitaev in \cite{FIK2}. Below, we show the solvability of \hyperref[rhy]{\rhy} for all $n\in\N$ large enough following the method of non-linear steepest descent introduced by Deift and Zhou \cite{MR1207209}. 

\subsection{Normalized Riemann-Hilbert problem and lenses}

Suppose that $\mathbf Y_n(z)$ is a solution of \hyperref[rhy]{\rhy}. Let
\begin{equation}
\label{rh2}
\mathbf T_n(z) := \ee^{-\frac{n}{2}\ell_*\sigma_3}\mathbf Y_n(z) \ee^{-n(g(z)-\frac{1}{2}\ell_*)\sigma_3},
\end{equation}
where the function $g(z)$ is defined by \eqref{g-def} and \( \ell_* \) was introduced in \eqref{g-V-eta}. Then
\[
\mathbf T_{n,+}(x) = \mathbf T_{n,-}(x)\begin{bmatrix} \ee^{-n(g_+(x)-g_-(x))} & \ee^{n(g_+(x)+g_-(x)-V(x)-\ell_*)+(n-N_n)V(x)} \\ 0 & \ee^{-n(g_-(x)-g_+(x))} \end{bmatrix},
\]
$x\in\Gamma$, and therefore we deduce from \eqref{g-log-requirement}, \eqref{g-jump}, and \eqref{g-V-jump} that $\mathbf T_n(z)$ solves \rht:
\begin{itemize}
\label{rht}
\item[(a)] $\mathbf T_n(z)$ is analytic in $\C\setminus\Gamma$ and $\lim_{\C\setminus\Ga\ni z\to\infty}\mathbf T_n(z)=\mathbb I$;
\item[(b)] $\mathbf T_n(z)$ has continuous boundary values on $\Gamma\setminus\{\pm a_2,\pm b_2\}$ that satisfy
\[
\mathbf T_{n,+}(x) = \mathbf T_{n,-}(x)  \left\{
\begin{array}{ll}
\begin{bmatrix} 1 & \ee^{n(\eta_{+}(x)+(n-N_n)V(x)}\\0&1\end{bmatrix}, & x\in\Gamma(e^{\ii\pi}\infty,-b_2), \medskip \\
\begin{bmatrix} 1 & \ee^{n\eta(x)+(n-N_n)V(x)} \\ 0 & 1 \end{bmatrix}, & x\in\Gamma(b_2,{\ee^{0\pi\ii}}\infty), \medskip \\
\begin{bmatrix} \ee^{n\pi \ii } & \ee^{n(\eta_{ +}(x) - \pi \ii) +(n-N_n)V(x)}\\ 0 & \ee^{-n\pi\ii } \end{bmatrix}, & x\in\Gamma(-a_2,a_2) ,\medskip \\
\begin{bmatrix} \ee^{-n\eta_{ +}(x)} & \ee^{(n-N_n)V(x)} \\ 0 & \ee^{n\eta_{+}(x)} \end{bmatrix}, & x\in \Ga(-b_2, -a_2) \cup \Gamma(a_2,b_2),
\end{array}
\right.
\]

\end{itemize}

Clearly, if \hyperref[rht]{\rht} is solvable and $\mathbf T_n(z)$ is the solution, then by inverting \eqref{rh2} one obtains a matrix $\mathbf Y_n(z)$ that solves \hyperref[rhy]{\rhy}.

To proceed, we observe the factorization of the jump matrices on $J_{t}$
\begin{multline*}
\begin{bmatrix} \ee^{-n\eta_{ +}(x)} & \ee^{(n-N_n)V(x)} \\ 0 & \ee^{n\eta_{+}(x)} \end{bmatrix} \\= \begin{bmatrix} 1 & 0 \\ \ee^{-n\eta_{ -}(x) - (n - N_n)V(x)} & 1 \end{bmatrix} \begin{bmatrix} 0 & \ee^{(n - N_n)V(x)}\\ -\ee^{-(n - N_n)V(x)} & 0 \end{bmatrix} \begin{bmatrix} 1 & 0 \\ \ee^{-n\eta_{ +}(x) - (n - N_n)V(x)} & 1 \end{bmatrix}.
\end{multline*}
This motivates the following transformation. Denote by $J_{\pm }$ smooth contours connecting $\pm a_2$ and $\pm b_2$ and otherwise remaining in the region $\Re[\eta(z)] >0$, as shown in Figure \ref{fig:two-cut-schematic}.
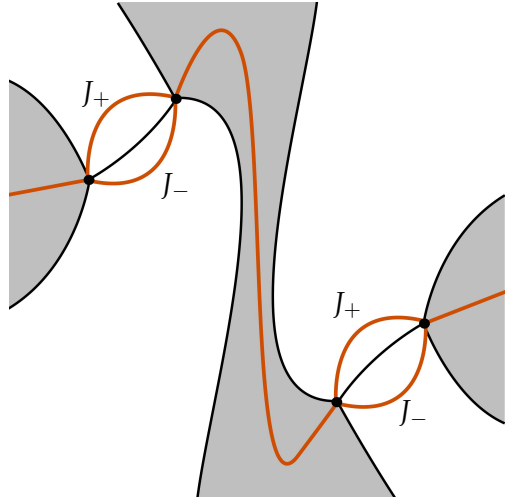
\begin{figure}
    \centering
    \includegraphics[scale=0.25]{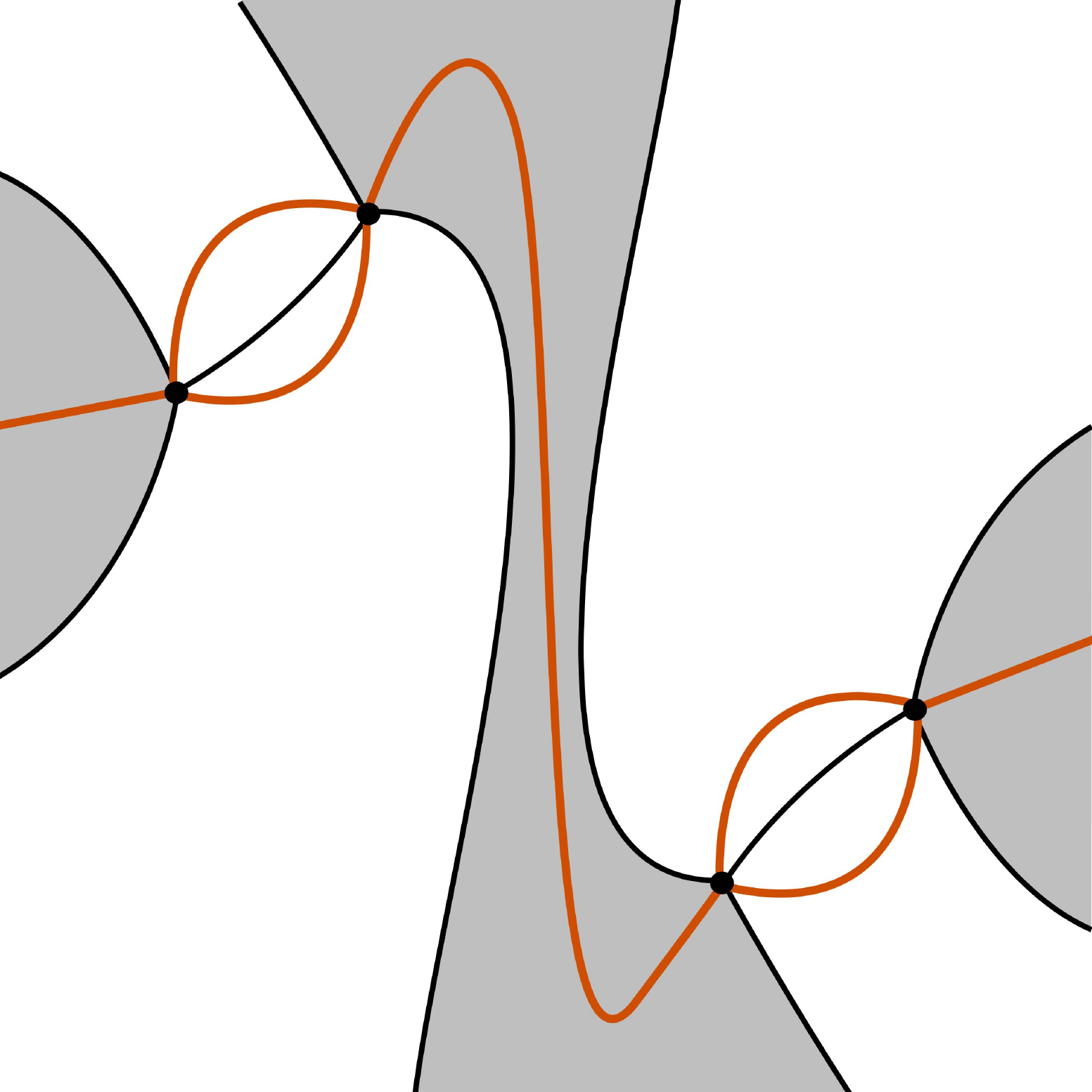}
    \put(-160,150){\small $J_+$}
    \put(-130,115){\small $J_-$}
    \put(-65,70){\small $J_+$}
    \put(-40,29){\small $J_-$}
    \caption{Schematic representation of the jump contour of $\mb S_n(z)$. The black contours and shaded regions are as in Figure \ref{fig:eta-sign-chart}.}
    \label{fig:two-cut-schematic}
\end{figure}
Denote the regions bounded by $J_{t}$ and $J_{\pm}$ by $L_{\pm}$ and define
\begin{equation}
\label{rh4}
    \mathbf{S}_n(z):=\mathbf T_n(z) \left\{ \begin{array}{ll}
       \begin{bmatrix} 1 & 0 \\ \mp \ee^{-n \eta(z) - (n - N_n)V(z)} & 1\end{bmatrix}  & z \in L_{\pm},   \medskip \\
        \mathbb{I}, & \text{otherwise}.
    \end{array} \right.
\end{equation}
Then, $\mathbf S_n(z)$ solves \rhs:
\begin{enumerate}[label=(\alph*)]\label{rhs}
    \item $\mathbf S_n (z)$ is analytic in $\C \setminus (\Ga \cup J_{\pm})$ and $\lim_{\C \setminus \Ga \ni z \to \infty} \mathbf S_n (z) = \mathbb{I}$, 
    \item $\mathbf S_n (z)$ has continuous boundary values on $\Ga \cup J_{\pm} \setminus \{\pm a_2, \pm b_2\}$ which satisfy 
    \[
    \mathbf S_{n,+}(x) = \mathbf S_{n,-}(x) \left\{
    \begin{array}{ll}
    \begin{bmatrix} 1 & \ee^{n(\eta_{ +}(x)+(n-N_n)V(x)}\\0&1\end{bmatrix}, & x\in\Gamma(e^{\ii\pi}\infty,-b_2), \medskip \\
    \begin{bmatrix} 1 & \ee^{n\eta(x)+(n-N_n)V(x)} \\ 0 & 1 \end{bmatrix}, & x\in\Gamma(b_2,{\ee^{0\pi\ii}}\infty), \medskip \\
    \begin{bmatrix} \ee^{n\pi \ii } & \ee^{n(\eta_{ +}(x) - \pi \ii) +(n-N_n)V(x)}\\ 0 & \ee^{-n\pi\ii } \end{bmatrix}, & x\in\Gamma(-a_2,a_2) ,\medskip \\
    \begin{bmatrix} 0 & \ee^{(n-N_n)V(x)} \\ -\ee^{-(n-N_n)V(x)} & 0 \end{bmatrix}, & x\in \Ga(-b_2, -a_2) \cup \Gamma(a_2,b_2), \medskip \\
    \begin{bmatrix} 1 & 0 \\  \ee^{-n \eta(x) - (n - N_n)V(x)} & 1\end{bmatrix}, & x \in J_\pm. 
    \end{array}
    \right.
    \]
\end{enumerate}

\subsection{Global parametrix} \label{sec:global}
If follows from the choice of contour $\Ga$ and the discussion at the end of Section \ref{sec:two-cut-g} (see e.g. Figure \ref{fig:eta-sign-chart}) that jumps of $\mathbf S_n(z)$ on $\Ga\setminus \Ga(-b_2, b_2)$ are exponentially\footnote{However, this exponential estimate is not uniform; we will handle this in the next subsection.} close to the identity, which motivates introducing the following model Riemann-Hilbert problem, \rhn. 
\begin{enumerate}[label=(\alph*)] \label{rhn}
\item $\mathbf N_n (z)$ is analytic in $\overline{\C} \setminus \Ga[-b_2, b_2]$ and $\mathbf{N}(\infty) = \mathbb{I}$, 
\item $\mathbf N_n (z)$ has continuous boundary values on $\Ga(-b_2, b_2) \setminus \{ \pm a_2 \}$ which satisfy
\[
\mathbf{N}_{n,+}(x) = \mathbf{N}_{n,-}(x) \left\{ \begin{array}{ll}
\begin{bmatrix} 0 & \ee^{(n-N_n)V(x)} \\ -\ee^{-(n-N_n)V(x)} & 0 \end{bmatrix}, & x\in \Ga(-b_2, -a_2) \cup \Gamma(a_2,b_2), \medskip \\
\begin{bmatrix} \ee^{n\pi \ii } & 0\\ 0 & \ee^{-n\pi\ii } \end{bmatrix}, & x\in I_t.
\end{array}\right.
\]
\end{enumerate}
 
In the next subsection, we solve \hyperref[rhn]{\rhn} explicitly.

\subsubsection{$\mathbf{N}_n(z)$ when $n \in 2\N$} Let 
\begin{equation}
\label{N-even}
    \mathbf N_n(z) := \mathcal{D}^{-(N_n - n) \sigma_3}(\infty)\begin{bmatrix} A(z) & -B(z) \\ B(z) & A(z) \end{bmatrix} \mathcal{D}^{(N_n - n) \sigma_3}(z),
\end{equation}
where $A(z), B(z)$, and $\mathcal{D}(z)$ are defined in \eqref{A-B-def} and Proposition \ref{Szego-prop}, respectively. \hyperref[rhn]{\rhn}(a) follows from the analyticity properties of $A(z), B(z), D(z)$, and equations \eqref{A-B-lim} 
 and \eqref{Szego-lim}. \hyperref[rhn]{\rhn}(b) follows from \eqref{A-B-jump} and \eqref{Szego-jump}. 

\subsubsection{$\mathbf{N}_n(z)$ when $n \in \N \setminus 2\N$} Let 
\[
\mathbf{M}_n(z) := \begin{bmatrix} 
(A/\Theta^{(0)})(z) & -(B/\Theta^{(1)})(z)\\ (\Theta^{(0)}B)(z) & (\Theta^{(1)}A)(z)
\end{bmatrix}  \mathcal{D}^{(N_n - n) \sigma_3}(z)
\]
where $\Theta(\z)$ is defined in \eqref{theta-k} and superscripts $(k)$ are explained in Section \ref{subsec:RS}. It follows from analyticity properties of $\Theta(\z)$ and Proposition \ref{prop-jip} that $\Theta^{(0)}(\infty)$ and consequently $\Theta^{(1)}(\infty)$ are finite and non-vanishing. Then, combining this with \eqref{A-B-lim} and boundedness/non-vanishing of $\mathcal{D}(\infty)$, we have that $\mathbf{M}(\infty)$ exists and 
\[
\lim_{z \to \infty} \det \mathbf M_n(z) =  \dfrac{ \Theta^{(1)}(\infty)}{\Theta^{(0)}(\infty)} \neq 0.
\]
Hence, the matrix
\begin{equation}
\label{N-odd}
\mathbf N_n(z) := \mathbf M_n^{-1}(\infty) \mathbf{M}_n(z) ,
\end{equation}
exists for all $n \in \N\setminus 2\N$. Furthermore, $\mathbf N_n(z)$ satisfies \hyperref[rhn]{\rhn}(a) due to the analyticity properties of $\Theta(\z)$, $A(z), B(z),$ and $D(z)$. \hyperref[rhn]{\rhn}(b) follows from \eqref{Szego-jump}, \eqref{A-B-jump}, and \eqref{theta-jump}. 

It follows from \eqref{N-even}, \eqref{N-odd} that
\begin{enumerate}[label = (\roman*)]
\item for any $e \in \{\pm a_2, \pm b_2 \}$, we have that \( \mathbf N_n(z) = \Oo\left( |z - e|^{-1/4} \right)\) as $z \to e$ (entry-wise), and
\item $\det \left(\mathbf N_n(z) \right) = 1$.
\end{enumerate}
Indeed, (i) follows immediately from the definition of $A(z), B(z), D(z)$ and $\Theta(\z)$. For (ii), note that $\det\left(\mathbf N_n(z) \right)$ is an analytic function in $\C \setminus \{\pm a_2, \pm b_2\}$. Furthermore, $\det \left(\mathbf N_n(z) \right)$ has at most square root growth at each of $\{ \pm a_2, \pm b_2\}$, making them removable singularities. Liouville's theorem implies $\det \left(\mathbf N_n(z) \right)$ is constant, and the normalization at infinity implies (ii).

\subsection{Local parametrices} To handle the fact that the jumps removed in \hyperref[rhn]{\rhn} are not uniformly close to identity, we construct parametrices near the endpoints $\{\pm a_2, \pm b_2\}$ which solve \hyperref[rhs]{\rhs} in a neighborhood $U_e, \ e \in \{\pm a_2, \pm b_2\}$ \emph{exactly}. Furthermore, we require that this local solution ``matches" the global parametrix on $\partial U_e$. Specifically, fix $e \in \{\pm a_2, \pm b_2\}$ and let 
\[
U_e := \{z \ : \ |z - e| < \delta_{e}(t) \},
\]
where $\delta_e(t)$ is a radius to be determined. We seek a solution to the following Riemann-Hilbert problem (\rhp)
\begin{enumerate}[label = (\alph*)]
\label{rhp}
\item $\mathbf P_{e,n}(z)$ has the same analyticity properties as $\mathbf S_n(z)$ restricted to $U_e$, see \hyperref[rhs]{\rhs}(a);
\item $\mathbf P_{e,n}(z)$ satisfies the same jump relations as $\mathbf S_n(z)$ restricted to $U_e$, see \hyperref[rhs]{\rhs}(b);
\item $\mathbf P_{e,n}(z)=\mathbf N_n(z)\left(\mathbb I+\Oo(n^{-1})\right)$ holds uniformly on $\partial U_e$ as $n\to\infty$.
\end{enumerate}

\subsubsection{Model Parametrix.}
Let \( L_\theta := \left \{ r \ee^{\ii \theta} \ : \ r >0 \right \} \). We will solve \hyperref[rhp]{\rhp} with the help of the Airy matrix $\mb A(\zeta)$, see e.g. \cites{MR1702716,MR1858269} (see also \cite{MR1677884}), whose construction is by now standard and which solves the following Riemann-Hilbert problem:
\begin{enumerate}[(a)]
    \item $\mb A(\zeta )$ is analytic in $\C \setminus (\R \cup L_{\pm 2\pi /3})$, 
    \item $\mb A(\zeta )$ has continuous boundary values on $\R \cup L_{\pm 2\pi /3}$ which satisfy
    \[
    \mathbf A_+(s) = \mathbf A_-(s) \left\{
    \begin{array}{ll}
    \begin{bmatrix} 0 & 1 \\ -1 & 0 \end{bmatrix}, & s\in (-\infty,0), \medskip \\
    \begin{bmatrix} 1 & 0 \\ 1 & 1 \end{bmatrix}, & s\in L_{\pm2\pi /3}, \medskip \\
    \begin{bmatrix} 1 & 1 \\ 0 & 1 \end{bmatrix}, & s\in (0,\infty),
    \end{array}
    \right.
    \]
    where the real line is oriented from $-\infty$ to $\infty$ and the rays $L_{\pm2\pi /3}$ are oriented towards the origin. 
    \item $\mb A(\zeta)$ has the following asymptotic expansion\footnote{We say that a function \( F(p) \), which might depend on other variables as well, admits an asymptotic expansion \( F(p) \sim \sum_{k=1}^\infty c_k \psi_k(p) \), where the functions \( \psi_k(p) \) depend only on \( p \) while the coefficients \( c_k \) might depend on other variables but not \( p \), if for any natural number \( K \) it holds that \( F(p)-\sum_{k=0}^{K-1}c_k \psi_k(p) =  \Oo_K(\psi_K(p)) \).} at infinity:
    \begin{equation}
    \label{rh6}
    \mathbf A(\zeta)\ee^{\frac23\zeta^{3/2}\sigma_3} \sim \frac{\zeta^{-\sigma_3/4}}{\sqrt2}\sum_{k=0}^\infty \begin{bmatrix} s_k & 0 \\ 0 & t_k \end{bmatrix} \begin{bmatrix} (-1)^k & \mathrm i \\ (-1)^k \mathrm i & 1 \end{bmatrix}  \left(\frac23\zeta^{3/2}\right)^{-k},
    \end{equation}
    uniformly in $\C\setminus\big(\R \cup L_{\pm 2\pi /3}\big)$, where
    \[
    s_0=t_0=1, \quad s_k=\frac{\Gamma(3k+1/2)}{54^kk!\Gamma(k+1/2)}, \quad t_k=-\frac{6k+1}{6k-1}s_k, \quad k\geq1.
    \]
\end{enumerate}

\subsubsection{Local coordinates.}
In this section, it will be convenient to denote
\begin{equation*}
\label{J-I-e}
    J_e := J_t \cap U_e, \quad I_e := \left(\Ga \cap U_e\right) \setminus J_e.
\end{equation*}
To transform \hyperref[rhp]{\rhp} to the model problem above, we will need to study the function $\eta_{ e}(z)$ as $ z\to e$. To this end, observe that the choice of of $\Ga_t$ in Theorem \ref{thm:2-cut} implies 
\[
\eta_{ e}(z) < 0 \qforq z \in I_e,
\]
and that 
\[
\eta_{ e}(z) = \pm 2\pi \ii \varsigma(e) \mu_t \left( \Ga_{t}[e, z]\right) = \ee^{\mp \varsigma(e) \frac{3\pi \ii}{2}} 2\pi  \mu_t \left( \Ga_{t}[e, z]\right), \quad z \in J_{e}
\]
where 
\begin{equation}
\label{var-sigma}
\varsigma(e) = \left\{ \begin{array}{ll} -1, & e \in \{-b_2, a_2\}, \\ 1, & e\in \{b_2, -a_2\}.  \end{array} \right.
\end{equation}
In particular, since $\eta_{ e}(z) \sim |z - e|^{\frac32}$ as $z \to e$, we may define a holomorphic branch of $\zeta_e(z):=(-3n \eta_{ e}(z)/4)^{\frac23}$ so that $\eta_{e}(I_e) \subset (0, \infty)$, $\eta_{e}(J_{e}) \subset (-\infty, 0)$, and $\eta_{e}$ is conformal in a sufficiently small neighborhood of $e$; we choose $\delta_{e}(t)>0$ so that $\eta_{e}$ is conformal in $U_e$. 

\begin{remark}
    We note that for any given compact $K \subset \Oo_2$, we may (and will) choose $\delta_e(t)$ so that they are continuous and separated from zero on $K$ by, say, choosing $2 \delta_e(t) = \left( \min \{2\pi , |\eta_{e}(0)|\}\right)^{2/3}$. This follows from the Basic Structure Theorem \cite{MR0096806}*{Theorem 3.5} and arguments carries out in \cite{MR3607591}*{Section 6.4}. One of the technical difficulties in extending our asymptotic results to hold uniformly on unbounded subsets of $\Oo_2$ is choosing such radii that remain bounded away from zero as $t \to \infty$. This was done in \cite{MR3607591}*{Section 6.4} and requires careful estimates of the quantity $\delta_e(t)$ above. While similar analysis can be carried out here, for the sake of brevity we settle for locally uniform asymptotics. 
\end{remark}

\subsubsection{Solution of \hyperref[rhp]{\rhp}.} Let 
\begin{equation}
    \mathbf{P}_{e, n}(z) := \mathbf E_e(z) \mathbf{A}_e(\zeta_e(z)) \ee^{\frac23 \zeta^{3/2}_{e}(z) \sigma_3} \mathbf J_e(z),
    \label{eq:local-parametrix}
\end{equation}
where 
\begin{itemize}
\item $\mathbf E_e(z)$ is a holomorphic prefactor to be determined, 
\item $\mathbf A_e(\zeta) := \mathbf{A}(\zeta)$ for $e \in \{-a_2, b_2\}$ and $\mathbf A_e(\zeta) := \sigma_3\mathbf{A}(\zeta) \sigma_3$ otherwise. The conjugation by $\sigma_3$ has the effect of reversing the orientation of $\R \cup L_\pm$,
\item we define  
\begin{equation}
    \mathbf J_e (z) := \ee^{\frac12 (N_n - n)V(z)\sigma_3} \left\{ \begin{array}{ll}
    \mathbb{I}, & e \in\{ \pm b_2 \}, \\
    \ee^{\pm \frac{\pi \ii}{2} n \sigma_3}, & e \in \{ \pm a_2 \},
    \end{array}\right.
    \label{eq:sectional-fun}
\end{equation}
and we choose the $+/-$ sign when $z$ is to the left/right of $\Ga$, respectively.
\end{itemize}
Indeed, \hyperref[rhp]{\rhp}(a) follows directly from the analyticity properties of $\mathbf A_e(\zeta)$, $\mathbf{J}_e(z)$, and $\zeta_e(z)$. Using \eqref{g-V-eta}, \eqref{g-jump}, \eqref{g-V-jump}, and \eqref{eta-e-b}, we see that for any $e \in \{\pm a_2, \pm b_2 \}$
\begin{equation*}
(g_+ + g_- - V - \ell_*)(z)  = \dfrac{1}{2} \left( \eta_{e, +} + \eta_{e, -} \right)(z), \quad z \in \Ga
\end{equation*}
and 
\begin{equation}
\label{eta-e-jump}
(g_+ - g_-)(z) = \dfrac{1}{2}\left( \eta_{e, +} - \eta_{e, -} \right)(z) + \left\{ \begin{array}{ll}
0, & e = b_2, \\
\pi \ii, & e\in \{ \pm a_2 \}, \\
2\pi \ii, & e =-b_2,
\end{array}\right. \quad z \in \Ga.
\end{equation}

Using \eqref{eta-e-jump} along with \eqref{g-jump}, \eqref{g-V-jump}, one may directly verify \hyperref[rhp]{\rhp}(b). Finally, to satisfy \hyperref[rhp]{\rhp}(c), we choose 
\begin{equation}
\mathbf E_e (z) := \mathbf N_n(z) \mathbf J_e^{-1}(z) \dfrac{1}{\sqrt{2}}\begin{bmatrix} 1 & - \ii \varsigma(e) \\ - \ii \varsigma(e) & 1  \end{bmatrix} \zeta_e^{1/4}(z),
    \label{eq:hol-prefactor}
\end{equation}
where $\varsigma(e)$ is as as in \eqref{var-sigma}. It follows from the definition of $\mathbf N_n(z), \mathbf J_e(z)$, and jump conditions \hyperref[rhn]{\rhn}(b) that $\mathbf E_e(z)$ is analytic in $U_e \setminus \{e\}$. At $z=e$, it follows from the discussion in Section \ref{sec:global} and the conformality of $\zeta_e(z)$ that $\mathbf E_e(z)$ has at most a square-root singularity at $e$, making it a removable singularity and yielding the desired analyticity. \\

Using \eqref{rh6}, the definition of $\mathbf P_{ e, n}(z)$, and  the definition of $\mathbf E_e(z)$, it follows that for $z \in \partial U_e$,\footnote{Strictly speaking, this is not an asymptotic expansion in the sense of the previous subsection since coefficients $\mathbf{C}_{e, k}$ still depend on $n$. However, for any $K \in \N$, the estimate $\mathbf{P}_{e, n}(z) - \sum_{k = 0}^{K-1} n^{-k} \mathbf{C}_{e, k}(z) = \Oo_K(n^{-K})$ still holds. \label{footnote:asymptotic-expansion}}
\begin{equation}
\label{local-parametrix-asymptotics}
\mathbf{P}_{e, n}(z) \sim \mathbf{N}_n(z)
 \left( \mathbb{I} + \dfrac{1}{n} \sum_{k = 0}^\infty  \dfrac{\mathbf{C}_{e,k}(z;t)}{n^k}\right) \qasq n \to \infty,
 \end{equation}
 where $\mathbf C_{e, k}(z; t)$ can be calculated using \eqref{rh6}, \eqref{eq:local-parametrix}, \eqref{eq:sectional-fun}, and \eqref{eq:hol-prefactor}. We will not use these formulas and so omit them, but we do point out that $\mathbf{C}_{e,k}(z; t)$ depend on $t$ via $\eta_{ e}(z; t)$ and $V(z; t)$.

\subsection{Small-norm Riemann-Hilbert problem}
\begin{figure}
    \centering
    \includegraphics[scale=0.25]{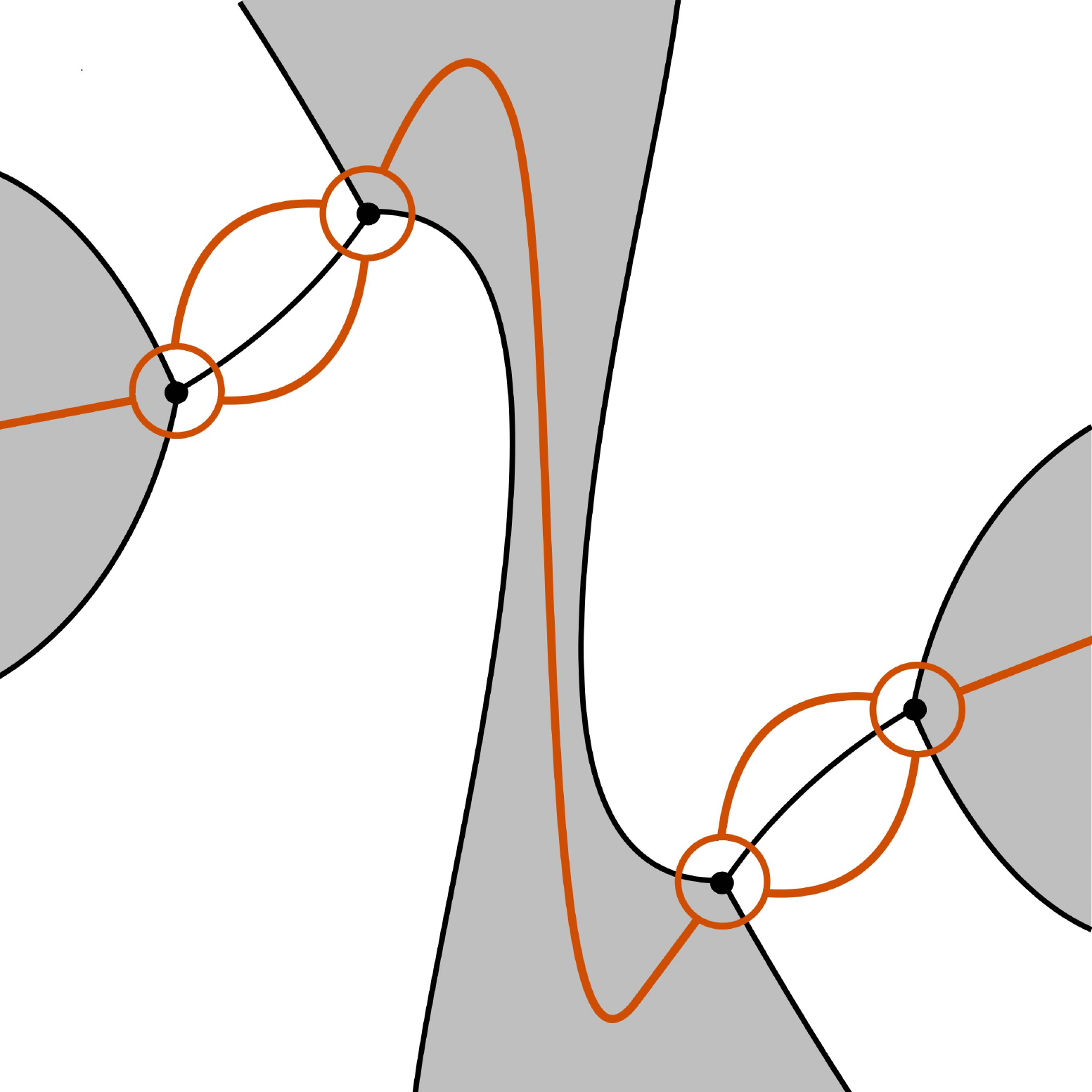}
    \caption{Schematic representation of $\Sigma$ (orange). The black contours and shaded regions are as in Figure \ref{fig:eta-sign-chart}.}
    \label{fig:small-norm-jump}
\end{figure}
Let $$\Sigma = \left[\left(\Ga \cup J_t \cup J_+ \cup J_- \right) \cap (\C \setminus \cup_e U_e) \right] \cup \left( \cup_e \partial U_e \right),$$ shown schematically in Figure \ref{fig:small-norm-jump}. Consider the following Riemann-Hilbert problem \rhr:
\begin{enumerate}[label=(\alph*)] \label{rhr}
    \item $\mathbf R_n(z)$ is holomorphic in $\C \setminus \Sigma$ and $\lim_{\C \setminus \Ga \ni z \to \infty} \mathbf R_n (z) = \mathbb{I}$, 
    \item $\mathbf R_n(z)$ has continuous boundary values on points of $\Sigma$ with a well-defined tangent and 
    \[
    \mathbf R_{n,+}(x) =  \mathbf R_{n,-}(x) \left\{ \begin{array}{ll}
    \mathbf P_{e, n}(x) \mathbf N_n^{-1}(x), & x \in \partial U_e, \medskip \\
    \mathbf N_n(x) \begin{bmatrix} 1 & 0 \\ \ee^{-n \eta(x) - (n - N_n)V(s)} & 1 \end{bmatrix} \mathbf N_n^{-1}(x), & s \in J_{\pm} \cap \C \setminus \cup_e U_e.
    \end{array}\right.
    \]
    where $\partial U_e$ are oriented clockwise, and (recall that $\det \mathbf N_n(z) = 1$, see Section \ref{sec:global})
    \[
    \mathbf R_{n,+}(x) =  \mathbf R_{n,-}(x) \left\{ \begin{array}{ll}
\mathbf N_n(x) \begin{bmatrix} 1 & \ee^{n(\eta_{ +}(x)+(n-N_n)V(x)}\\0&1\end{bmatrix} \mathbf N_n^{-1}(x) , & x\in\Gamma(e^{\ii\pi}\infty,-b_2), \medskip \\
\mathbf N_n(x) \begin{bmatrix} \ee^{-n\eta_{ +}(x)} & \ee^{(n-N_n)V(x)} \\ 0 & \ee^{n\eta_{+}(x)} \end{bmatrix}\mathbf N_n^{-1}(x) , & x\in \Ga(-b_2, -a_2),  \medskip \\
\mathbf N_n(x) \begin{bmatrix} 1 & \ee^{n\eta(x)+(n-N_n)V(x)} \\ 0 & 1 \end{bmatrix} \mathbf N_n^{-1}(x) , & x\in\Gamma(b_2,{\ee^{0\pi\ii}}\infty), \medskip \\
\mathbf N_{n,-}(x) \begin{bmatrix} \ee^{n\pi \ii } & \ee^{n(\eta_{ +}(x) - \pi \ii) +(n-N_n)V(x)}\\ 0 & \ee^{-n\pi\ii } \end{bmatrix}\mathbf N^{-1}_{n, +}(x) , & x\in I_t ,\medskip \\
\mathbf N_n(x) \begin{bmatrix} \ee^{-n\eta_{ +}(x)} & \ee^{(n-N_n)V(x)} \\ 0 & \ee^{n\eta_{+}(x)} \end{bmatrix}\mathbf N_n^{-1}(x) , & x\in  \Gamma(a_2,b_2).
\end{array} \right.
    \]
\end{enumerate}

We now show that $\mathbf R_n(z)$ satisfies a so-called "small-norm" Riemann-Hilbert problem. The estimates in this section are fairly well-known at this point, but since we also consider the dependence on the parameter $t$, we point to \cite{MR3607591}*{Section 7.6} as a general reference. Consider 
\begin{equation}
\label{delta-def}
\mathbf \Delta(x;t) := (\mathbf R_{n, -}^{-1} \mathbf R_{n,+} )(x) - \mathbb{I}, \quad x \in \Sigma.
\end{equation}

It follows from the explicit formulas in Section \ref{sec:global} that $\mathbf N_n(z)$ is uniformly bounded on $\Sigma$. This, the choice of $\Ga$ (see \eqref{re-eta} and the surrounding discussion), the fact that $|n - N_n|$ is bounded for all $n\in \N$, and the jump of $\mathbf N_n(z)$ on $I_t$ implies that there exists $c(t; \delta)>0$, where $\delta = \min_{e} \delta_e$, so that as $n \to \infty$
\begin{equation}
\mathbf \Delta (x; t) = \Oo \left(\ee^{-c(t; \delta) n} \right) \qforq x \in (J_\pm \cap (\C \setminus \cup_e U_e) ) \cup I_t \cup \Ga(-\infty, -b_2) \cup \Ga(b_2, \infty).
\label{eq:estimate-lenses}
\end{equation}
Since $\mb \Delta(x;t)$ is continuous in $t$, we can choose the exponent $c(t; \delta)$ to be continuous in $t$ as well, and by compactness we find that \eqref{eq:estimate-lenses} holds $(x, t)$-locally uniformly. Similarly, by \eqref{local-parametrix-asymptotics} and the definition of $\mathbf{\Delta}(x;t)$, it follows that for $x \in \partial U_e$
\begin{equation}
\mathbf{\Delta}(x; t) \sim \sum_{k = 1}^\infty \dfrac{\mathbf N_n(x; t) \mathbf C_{e,k}(x; t) \mathbf N_n^{-1}(x; t)}{n^k} \qasq n \to \infty.
\label{eq:estimate-boundary-U}
\end{equation}
Once again, by continuity we have that estimate \eqref{eq:estimate-boundary-U} holds $(x,t)$-locally uniformly. Put together, the above estimates yield that $\| \mathbf{\Delta} \|_{L^\infty(\Sigma)} = \Oo\left(n^{-1} \right)$ $t$-locally uniformly. Furthermore, on the unbounded components of $\Sigma$ we have 
\[
\mathbf \Delta(x;t) = \mathbf N_n(x;t) \begin{bmatrix} 1 & \ee^{n(\eta_{ +}(x;t)+(n-N_n)V(x;t)}\\0&1\end{bmatrix} \mathbf N_n^{-1}(x;t),
\]
and since $\mathbf N_n(z)$ is bounded as $x \to \infty$ and $\ee^{n\eta(x;t)}$ is exponentially small as $x \to \infty$, it follows that 
\begin{equation*}
\label{eq:Delta-est}
\|\mathbf{\Delta} \|_{L^\infty(\Sigma) \cap L^2(\Sigma)} = \Oo\left(n^{-1} \right).
\end{equation*}
Hence, applying \cite{MR1677884}*{Corollary 7.108}, yields that for $n$ large enough, a solution $\mathbf R_n(z)$ exists and 
\begin{equation}
\label{R-estimate}
\mathbf R_n(z) = \mathbb I + \Oo \left( n^{-1}\right) \qasq n \to \infty,
\end{equation}
$(z, t)$-uniformly in $\C$. In the next subsection, we use this to conclude asymptotic formulas for polynomials $P_n(z;t, N_n)$ and recurrence coefficients $\gamma^2_n(t, N_n)$.

\subsection{Asymptotics of \texorpdfstring{$P_n(z;t)$}{the Orthogonal Polynomials}.}
Given $\mathbf R_n(z)$, \( \mathbf N_n(z) \), and \( \mathbf P_{e,n}(z) \), solutions of \hyperref[rhr]{\rhr}, \hyperref[rhn]{\rhn}, and \hyperref[rhp]{\rhp}, respectively, one can verify that \hyperref[rhs]{\rhs} is solved by
\begin{equation}
\label{rh17}
\mathbf S_n(z) = \left\{
\begin{array}{ll}
(\mathbf R_n\mathbf N_n)(z), &  \quad (\C \setminus \overline{U}_e)\setminus[(\Ga\setminus J_t)\cup J_+\cup J_-], \medskip \\
(\mathbf R_n \mathbf P_{e, n})(z), &  \quad U_e,~~~e\in\{\pm a_2,\pm b_2\}.
\end{array}
\right.
\end{equation}
Let $K$ be a compact set in $\C\setminus\Ga$. By choosing $J_{\pm}$ appropriately, we can arrange for $K$ to lie entirely withing the unbounded component of $\C \setminus \Sigma$. Then it follows from \eqref{rh2}, \eqref{rh4}, and \eqref{rh17} that
\begin{equation}
\label{rh18}
\mathbf Y_n(z) = \ee^{\frac{n}{2}\ell_*\sigma_3}(\mathbf{R}_n \mathbf{N}_n)(z) \ee^{n(g(z)-\frac12 \ell_*)\sigma_3}, \quad z\in K.
\end{equation}
Subsequently, by using \eqref{rh1} and \eqref{g-V-eta}, we see that 
\[
P_n(z) = [\mathbf Y(z)]_{11} = \ee^{ \frac n2(\eta(z) + V(z)+\ell_*)}\left([\mathbf R_n(z)]_{11}[\mathbf N_n(z)]_{11}+[\mathbf R_n(z)]_{12}[\mathbf N_n(z)]_{21}\right).
\]
Using formulas \eqref{N-even}, \eqref{N-odd}, estimate \eqref{R-estimate}, and the definitions of the entries of $\mb N_n(z)$ yields \eqref{p-asymp-outside}.

\subsection{Asymptotics of \texorpdfstring{$\gamma_n^2(t)$}{Recurrence Coefficients}} \label{sec:gamma-asymp}
It follows from Theorem \ref{thm:asymptotics-poly} that for $n$ large enough, $\deg P_n = n$, and in this section we will use the notation $$P_n(z;t, N) = z^n +  \sum_{k = 0}^{n - 1} \mathfrak{p}_{n, k}(t, N) z^k.$$ Furthermore, we will drop the $n$ subscript (e.g. $\mathbf N_n \mapsto \mathbf N$) and use notation
\[
\mathbf K(z)  =  \mathbb I + z^{-1} \mathbf K_1(n,t,N) + z^{-2} \mathbf K_2(n,t,N) + z^{-3} \mathbf K_3(n,t,N) + \mathcal{O}\big(z^{-4}\big) \quad \text{as} \quad z \to \infty,
\]
where \(  \mathbf K \in \{ \mathbf M, \mathbf N, \mathbf R, \mathbf T ,\mathbf Y z^{-n\sigma_3}  \}  \). It follows from the orthogonality relations \eqref{eq:ortho} and \eqref{eq:normalizing} that
\begin{align*}
z^n\big(\mathcal{C}P_n \ee^{-NV}\big)(z) &= -\frac{1}{z}\frac1{2\pi\ii} \int_\Ga x^nP_n(x) \ee^{-NV(x)}\dd x - \frac{1}{z^2}\frac1{2\pi\ii} \int_\Ga x^{n+1}P_n(x)\ee^{-NV(x)}\dd x + \mathcal O\left( z^{-3}\right) \\
&=  \frac{h_n(t,N)}{2\pi\ii}\left(- \frac1z + \frac{ \mathfrak{p}_{n+1, n}(t, N)}{z^2}\right) + \mathcal O\left( z^{-3}\right),
\end{align*}
where the second expression for the coefficient next to \( z^{-2} \) follows from the observation
\begin{multline*}
\int_\Ga x^{n+1}P_n(x) \ee^{-NV(x)}\dd x \\
= \int_\Ga \left(x^{n + 1} - P_{n + 1}(x)\right) P_n(x) \ee^{-NV(x)}\dd x = \mathfrak{p}_{n+1, n}(t, N) \int x^nP_n(x) \ee^{-NV(x)}\dd x,
\end{multline*}
where all polynomials correspond to the same parameter \( N \). However, since the weight of orthogonality is an even function, we have $P_n(-z;t, N) = (-1)^n P_n(z;t,N)$, i.e. polynomials $P_n$ are even/odd functions. In particular $\mathfrak{p}_{n+1, n}(t, N) = 0$. Hence, it follows from \eqref{rh1} that 
\begin{equation*}
\begin{aligned}
\label{rc2}
{\mathbf Y}_1(n,t,N) &= \begin{bmatrix} 0 & - \frac{h_n(t,N)}{2\pi \ii} \medskip \\  - \frac{2\pi \ii}{h_{n - 1}(t,N)} & 0  \end{bmatrix} \\
\mathbf Y_2(n,t, N) &= \begin{bmatrix} \mf p_{n,n-2}(t, N) &  0 \medskip \\ 0 & *  \end{bmatrix}, \\
\mathbf Y_3(n,t, N) &= \begin{bmatrix} 0 &  * \medskip \\ - \frac{2\pi \ii}{h_{n - 1}(t,N)} \mf p_{n-1,n-3}(t, N) & *  \end{bmatrix}.
\end{aligned}
\end{equation*}
This along with \eqref{eq:gamma} yields the formulae
\begin{equation*}
\label{rc3}
\begin{cases}
h_n(t, N) & = -2\pi \ii [\mathbf Y_1(n,t, N)]_{12} , \medskip \\
\gamma_n^2 (t, N) & = [\mathbf Y_1(n,t, N)]_{12} [\mathbf Y_1(n, t,N)]_{21}, \medskip \\
\mf p_{n, n-2}(t, N) & = [\mb Y_2(n, t, N)]_{11}, \medskip\\
\mf p_{n-1, n-3}(t, N) & = \dfrac{[\mb Y_3(n, t, N)]_{21}}{[\mathbf Y_1(n, t,N)]_{21}} .
\end{cases}
\end{equation*}

Next, using \eqref{eq:g-explicit} we have
\begin{equation}
\label{rc4}
\ee^{-ng(z) \sigma_3} = z^{-n\sigma_3} \left( \mathbb I  - \dfrac{nt}{2z^2} \sigma_3 + {\Oo} \left( z^{-4} \right) \right).
\end{equation}
Combining \eqref{rc4} with \eqref{rh2} yields
\begin{equation*} 
\label{Y-T}
\begin{cases}
	[\mathbf T_1(n,t, N)]_{12} &= \ee^{-n \ell_*} [\mathbf Y_1(n, N)]_{12}, \medskip \\
	[\mathbf T_1(n,t, N)]_{21} &= \ee^{n \ell_*} [\mathbf Y_1(n,t, N)]_{21}, \medskip\\
 [\mathbf T_1 (n,t, N)]_{11} &= [\mathbf Y_1 (n,t, N)]_{11}, \medskip\\
 [\mathbf T_2(n,t, N)]_{11} &= [\mathbf Y_2(n,t, N)]_{11} - \dfrac{nt}{2}, \medskip\\
  [\mathbf T_3(n,t, N)]_{21} &= \ee^{n\ell_*} \left( [\mathbf Y_3(n,t, N)]_{21} - \dfrac{nt}{2} [\mathbf Y_1(n,t, N)]_{21}\right).
\end{cases}
\end{equation*}
This, in turn, leads to
\begin{equation}
\label{rc6}
\left\{   
\begin{array}{lll}
	h_n(t, N) & = & -2\pi \ii \ee^{n \ell_*} [\mathbf T_1(n, N)]_{12} , \medskip \\
	\gamma_n^2 (t, N) & = & [\mathbf T_1(n,t, N)]_{12} [\mathbf T_1(n, t, N)]_{21}, \medskip\\
 \mf p_{n, n-2}(t, N) & = & [\mb T_2(n, t, N)]_{11} + \dfrac{nt}{2} ,\medskip \\
 \mf p_{n-1, n-3}(t, N) & = & \dfrac{[\mb T_3(n, t, N)]_{21}}{[\mb T_1(n, t, N)]_{21}} + \dfrac{nt}{2} .
 
\end{array}
\right.
\end{equation}
It was shown in \cite{MR1711036}*{Theorem 7.10}\footnote{In the reference, fractional powers of $n$ appear because of the dependence of the conformal maps $\varphi_n$ (analog of our $\eta_{ e}$) on $n$. This is not the case in, say, \cite{MR1702716}*{Eq. 4.115}.} that $\mathbf R(z)$ admits an asymptotic expansion (in the sense of Footnote \ref{footnote:asymptotic-expansion})
\begin{equation}
	\label{r-expansion}
	\mathbf R(z) \sim \mathbb I + \sum_{k = 1}^{\infty} \dfrac{\mathbf R^{(k)}(z, n)}{n^k}, \quad n \to \infty \quad  \text{ and } \quad z \in \C \setminus \cup_e \partial U_e. 
\end{equation}

The matrices $\mathbf R^{(k)}(z)$ satisfy an additive Riemann-Hilbert Problem: 
\begin{enumerate}
	\item[(a)] $\mathbf R^{(k)}(z)$ is analytic in $\C \setminus \cup_e \partial U_e$,
	\item[(b)] $\mathbf R_+^{(k)}(z) = \mathbf R_-^{(k)}(z) + \sum_{j = 1}^{k} \mathbf R_-^{(k - j)}(z) \mathbf \Delta_j(z)$ for $z \in \cup_e \partial U_e$, where we write
 \[
\mathbf \Delta (z; t) \sim \sum_{k = 1}^\infty \dfrac{\mathbf \Delta_k(z;t)}{n^k} \qasq n \to \infty, 
 \]
 and $\mathbf \Delta(z; t)$ is as in \eqref{delta-def}, and
	\item[(c)] as $z \to \infty,$ $\mathbf R^{(k)}(z)$ admits an expansion of the form
	\begin{equation}
		\label{rk-expansion}
		\mathbf R^{(k)}(z) = \dfrac{\mathbf R_1^{(k)}}{z} + \dfrac{\mathbf R_2^{(k)}}{z^2} + \Oo\left( z^{-3} \right) .
	\end{equation}
\end{enumerate}
Recall that $\mathbf T(z) = (\mathbf R \mathbf N)(z)$ for $z$ in the vicinity of infinity by \eqref{rh2} and \eqref{rh18}. It also follows from their definitions that matrices \( \mathbf N(z) \) form a normal family in \( n\in\N \) in the vicinity of infinity. Plugging expansions \eqref{r-expansion} and \eqref{rk-expansion} into the product $\mathbf T(z) = (\mathbf R \mathbf N)(z)$ yields
\begin{equation}
\label{rc7}
\begin{aligned}
    \mathbf T_1(n, t, N_n) &= \mathbf N_1(n,t, N_n) + \dfrac{\mathbf R^{(1)}_1(t)}{n} +  \Oo\left( n^{-2} \right), \medskip \\
    \mathbf T_2(n, t, N_n) &= \mathbf N_2(n,t, N_n) + \dfrac{\mathbf R_1^{(1)}(t) \mathbf N_1(n,t, N_n) + \mathbf R_2^{(1)}(t)}{n}  +  \Oo \left( n^{-2} \right), \medskip \\
    \mathbf T_3(n, t, N_n) &= \mb N_3(n, t, N_n) + \dfrac{\mb R_1^{(1)}(t) \mb N_2(n,t,N_n) + \mb R_2^{(1)}(t) \mb N_1(n, t, N_n) + \mb R_3^{(1)}(t) }{n} + \Oo\left( n^{-2} \right).
    \end{aligned}
\end{equation}
where normality of $\mathbf N$ is used to ensure boundedness of coefficients $\mathbf N_k(n, t, N)$. To use \eqref{rc7}, we will need to compute $\mathbf N_1, \mathbf N_2, \mb N_3$. To this end, note that by \eqref{Szego-lim}, 
\begin{equation}
\label{eq:cal-D-exp}
\mathcal{D}^{(N-n) \sigma_3}(z) = \ee^{-\frac{1}{2}(N - n) \ell_* \sigma_3} \left( \mathbb{I} + \dfrac{N-n}{4z^4} \sigma_3 + \Oo \left( z^{-6}\right)\right).
\end{equation}
Furthermore, it follows from the definition of $\gamma(z)$ in \eqref{gamma} that
\begin{align*}
\gamma(z) &= 1 + \dfrac{ a_2 - b_2 }{2z} + \dfrac{(a_2 - b_2)^2}{8z^2} +\dfrac{(a_2 - b_2)(3 a_2^2 + 2 a_2 b_2 + 3 b_2^2)}{16z^3} + \Oo \left( z^{-3} \right)  \qasq z \to \infty, \\
\gamma^{-1}(z) &= 1 - \dfrac{a_2 - b_2}{2z} + \dfrac{(a_2 - b_2)^2}{8z^2} - \dfrac{(a_2 - b_2)(3 a_2^2 + 2 a_2 b_2 + 3 b_2^2)}{16z^3} + \Oo \left( z^{-4} \right) \qasq z \to \infty,
\end{align*}
which, in turn, yields 
\begin{equation}
\begin{aligned}
\label{A-B-expansion}
A(z) &= 1 + \dfrac{(a_2 - b_2)^2}{8z^2} + \Oo \left( z^{-4} \right), \\
B(z) &= \dfrac{ a_2 - b_2 }{2\ii z} + \dfrac{(a_2 - b_2)(3a_2^2 + 2a_2 b_2 + 3b_2^2)}{16\ii z^3} + \Oo \left( z^{-5} \right), \qasq z \to \infty.
\end{aligned}
\end{equation}
Put together, the expansions \eqref{eq:cal-D-exp}, \eqref{A-B-expansion} along with the definition of $\mathbf N$ yield that 
\begin{equation}
\label{N-1-even}
\mathbf N_1(n,t, N) = \dfrac{b_2 - a_2}{2 }\ee^{(N-n)\ell_* \sigma_3} \sigma_2 ,
\end{equation}
when $n \in 2\mathbb{N}$. Recalling Proposition \ref{prop-jip}, we can write 
\begin{equation}
\label{N-1-odd}
\mathbf N_1(n,t, N) = \ee^{\frac{1}{2}(N-n)\ell_* \sigma_3} \begin{bmatrix} 
-\dfrac{\dd}{\dd x} \left( \log \Theta^{(0)}(1/x) \right)\biggl|_{x = 0} &  -\dfrac{a_2 - b_2}{2 \ii}\dfrac{\Theta^{(0)}(\infty)}{\Theta^{(1)}(\infty)} \medskip \\  
	\dfrac{a_2 - b_2}{2 \ii}\dfrac{\Theta^{(0)}(\infty)}{\Theta^{(1)} (\infty)} &  \dfrac{\dd}{\dd x} \left( \log \Theta^{(1)}(1/x) \right)\biggl|_{x = 0}
\end{bmatrix} \ee^{-\frac{1}{2}(N-n)\ell_* \sigma_3},
\end{equation}
when $n \in \N \setminus 2\N$. Using \eqref{N-1-even}, \eqref{N-1-odd}, and \eqref{rc7} in \eqref{rc6} yields 
\begin{equation}
h_n(t, N_n) = \pi (a_2 - b_2) \ee^{n \ell_*} \ee^{(N_n-n)\ell_*} \left\{  \begin{array}{ll}
1, & n \in 2\N, \medskip \\  \dfrac{\Theta^{(0)} (\infty)}{\Theta^{(1)}(\infty)} , & n \in \N \setminus 2\N.
\end{array}\right\} + \Oo\left( n^{-1} \right) \qasq n\to \infty.
\label{eq:hn-asymp}
\end{equation}
Similarly, 
\begin{equation}
\label{gamma-n-asymptotic}
\gamma^2_n(t, N_n) = \dfrac{(a_2 - b_2)^2}{4} \left\{  \begin{array}{ll}
1, & n \in 2\N, \medskip \\  \left(\dfrac{ \Theta^{(0)}(\infty)}{  \Theta^{(1)} (\infty)} \right)^2 , & n \in \N \setminus 2\N.
\end{array}\right \} + \Oo\left( n^{-1} \right) \qasq n\to \infty.
\end{equation}

Note that by using formulas \eqref{a-b-2-cut} in \eqref{gamma-n-asymptotic} when $n \in 2\N$, we immediately recover formulas obtained in \cite{MR1715324}*{Theorem 1.1}\footnote{There, the authors consider a more general weight of orthogonality and use the notation $R_n = \gamma_n^2$. The specialization is: $g \mapsto 1, \ t \mapsto t$, and $\lambda \mapsto 1$. }, namely, when $n \in 2\N$,
\[
\gamma_n^2(t, N_n) = \dfrac{-t - \sqrt{t^2 - 4}}{2} + \Oo\left( n^{-1} \right) \qasq n \to \infty, \quad t \in (-\infty, -2).
\]
When $n \in \N \setminus 2\N$, however, matching our result with the above reference is less obvious. 
\begin{lemma}
    Let $\Theta(\z)$ be as in \eqref{theta-k} and $a_2, b_2$ be as in \eqref{a-b-2-cut}. Then, the following identity holds:
    \begin{equation}
     \label{theta-identity}
     \dfrac{\Theta^{(0)}(\infty)}{\Theta^{(1)}(\infty)} = \dfrac{4}{(a_2 - b_2)^2}. 
    \end{equation}
    \label{lemma:theta-id}
\end{lemma}
\begin{proof}
Consider the auxiliary function $G: \RS_{\ualpha, \ubeta} \to \overline{\C}$ defined by $G(z^{(0)}) = G^{-1}(z^{(1)})$ and 
\begin{equation*}
G(z^{(0)}) := \dfrac{1}{\mathcal{D}(z)} \exp\left( \frac{1}{2} \left( V(z) + \eta_{ b_2}(z) \right) \ \right),
\end{equation*}
where $\mathcal{D}(z), \eta(z)$ are as in \eqref{Szego-fun} and \eqref{eta-e}, respectively. Then, it follows from \eqref{g-jump}, \eqref{g-V-jump}, and \eqref{Szego-jump} that $G$ is well defined on $\RS \setminus \{\ualpha\}$ and for $\z \in \ualpha \setminus \pm \boldsymbol a_2$,
\[
G_+(\z) = -G_-(\z).
\]
Furthermore, it follows from \eqref{g-log-requirement}, \eqref{g-V-eta}, and \eqref{Szego-lim} as $z \to \infty$, 
\[
G(z^{(0)}) = z + \Oo(1). 
\]
Hence, the function $(G\Theta)(\z)$ is rational on $\RS$ with divisor 
\(
0^{(1)} + \infty^{(1)} - 0^{(0)} - \infty^{(0)}. 
\) This is also true of the function 
\[
F(\z) := \left\{ \begin{array}{ll} 
(A/B)(z), & \z \in \RS^{(0)}, \\ -(B/A)(z), & \z \in \RS^{(1)},   \end{array} \right. 
\]
where $A(z), B(z)$ are defined in \eqref{A-B-def}. Hence, there exists a constant $C$ so that $F(\z) = C(G\Theta)(\z)$. In particular, using \eqref{A-B-expansion} yields
\[
\Theta^{(0)}(\infty) = -C \dfrac{2\ii }{a_2 - b_2 }, \quad \Theta^{(1)}(\infty) = C \dfrac{a_2 - b_2 }{2\ii }  .
\]
Taking the ratio of the two identities yields \eqref{theta-identity}.
\end{proof}

Plugging \eqref{theta-identity} into \eqref{gamma-n-asymptotic} and \eqref{eq:hn-asymp}, we find \eqref{gamma-n-final} and \eqref{eq:hn-asymp-final}, respectively. 
Finally, we note that by combining \eqref{eq:cal-D-exp}, \eqref{A-B-expansion} with the definition of $\mb N$ we find that when $n \in 2\N$
\begin{equation*}
    \mb N_2(n, t, N) = \dfrac{(a_2 - b_2)^2}{8 } \I, \qandq \mb N_3(n, t, N) = \dfrac{(b_2 - a_2)(3a_2^2 + 2a_2 b_2 + 3b_2^2)}{16} \ee^{(N-n)\ell_* \sigma_3}\sigma_2
\end{equation*}
which, combined with \eqref{rc7} and \eqref{rc6}, yield \eqref{eq:sub-leading-1}, \eqref{eq:sub-leading-2}.

\addtocontents{toc}{\setcounter{tocdepth}{1}}
\section{Conclusion}
\label{sec:conclusion}

In this work, we studied orthogonal polynomials with a generalized Freud weight depending on a parameter $t$ and whose recurrence coefficients are parabolic cylinder solutions of Painlev\'e-IV. We obtain leading term asymptotics of the orthogonal polynomials and their recurrence coefficients in the so-called two-cut region and, as a consequence, established asymptotically pole-free regions for the corresponding special function solutions of Painlev\'e-IV.

\subsection{Future Work}
The most obvious first step is to carry out asymptotic analysis similar to Section \ref{sec:rh} for $t \in \Oo_3$. Other problems that remain open and deserve attention include:
\begin{itemize}
    \item A more careful analysis of the matrix $\mb R_n(z;t)$, it can be shown than $\gamma_n(t)^2$ admits a full asymptotic expansion
    \begin{equation*}
    \label{gamma-full-expansion}
    \gamma_n^2(t, N_n) \sim  \left\{ \begin{array}{ll}   \displaystyle \sum_{j = 0}^\infty\dfrac{e_j(t)}{n^j}, & n \in 2\N, \medskip \\ \displaystyle \sum_{j = 0}^\infty \dfrac{o_j(t)}{n^j}, & n \in \N \setminus 2\N.  \end{array} \right. \qasq n \to \infty,
    \end{equation*}
    where $e_j(t), o_j(t)$ are analytic functions of $t$ for $t \in \mathcal{O}_2$ and $e_0, o_0$ are as in equations \eqref{gamma-n-final}, and the remaining $o_j, e_j$ can be computed recursively using the String equation \eqref{string-eq}. This will give rise to two (in principle) different asymptotic expansions of the free energy. It would be interesting to understand if these expansions are different, and whether or not the enumeration property discussed in the introduction carries over to the two-cut region. 

    \item In this work, we concerned ourselves with B\"acklund iterates of the seed solution 
    \[
    \varphi(x) = \ee^{\frac12 x^2} D_{-\frac12}(\sqrt{2}x)
    \]
    since the corresponding tau functions appear in the quartic model. One could instead consider a generalization of this model where the starting point is the partition function 
    \[
    \hat Z_{n}(t; N) := \int_{\Ga} \cdots \int_{\Ga} \prod^n_{1 \leq k < j \leq n} (\zeta_k - \zeta_j)^2 \prod_{j = 1}^n |\zeta_j|^{2a}\ee^{-N\left(\frac{t}{2} \zeta_j^2 + \frac{1}{4}\zeta^4_j \right)} \dd \zeta_1 \cdots \dd \zeta_n,
    \]
    where $\Re(a)>-\frac12$, $\Ga = C_1\R + C_2 \ii \R$ and $C_1, C_2 \in \R$ are auxiliary parameters. Then, the corresponding measure of orthogonality is supported on $\Ga$ with density $|z|^{2a} \ee^{-NV(z; t)}$. Since $V(\ii z; t) = V(z; -t)$, a calculation similar to \eqref{eq:moment-0} yields 
    \begin{multline*}
    \int_{\Ga} |z|^{2a} \ee^{-NV(z;t)} \dd z \\
    = \left( \frac{2}{N}\right)^{\frac12 a + \frac14} \Gamma\left(a + \frac12\right) \ee^{\frac{N}{8}t^2} \left(C_1 D_{-a - \frac12}\left ( {N^{\frac12}2^{-\frac12}}t\right) + \ii C_2 D_{-a - \frac12}\left ( -{N^{\frac12}2^{-\frac12}}t\right)  \right),
    \end{multline*}
    which, modulo the restriction on $a$, corresponds to a general parabolic cylinder seed function. Removing the restriction on $a$ completely requires further deformation of the contour of integration. 
    
    \item It is known \cite{DK} that when $t \to \pm \ii \sqrt{12}$, the recurrence coefficients $\gamma_n^2(t, N)$ have asymptotics described in terms of {tronqu\'ee} solutions of Painlev\'e-I. This suggests that one might be able to obtain tronqu\'ee solutions of Painlev\'e-I as limits of special function solutions of Painlev\'e-IV. This type of critical Painlev\'e-I behavior is not unique to the quartic model, and can also be exhibited in the cubic model \cite{BD2}. In the setting of the cubic model, it is known that the relevant recurrence coefficients are related to Airy solutions of Painlev\'e-II \cite{BBDY}. This suggests that one might be able to obtain {tronqu\'ee} solutions of Painlev\'e-I as limits of special function solutions of Painlev\'e-II. It would be interesting to work out these limits in detail and identify the limiting solution of Painlev\'e-I. 
    \end{itemize}

\appendix

\section{\texorpdfstring{$H_{n}(t, N)$}{Hankel determinants} as a product of \texorpdfstring{$\tau$}{tau}-functions} \label{appendix-a}

In Theorem \ref{thm:main2}, we omitted an interpretation of the functions $\tau_{n, 0, 0}(x), \tau_{n, -1, 0}(x)$ in terms of the orthogonal polynomials $P_n(z; t, N)$ and related quantities since we did not need such formulas. We record these here, and as a consequence deduce a theorem on the structure of degree degeneration of polynomials $P_n(z; t, N)$, see Proposition \ref{prop:deg}. We start with the following lemma. 

\begin{proposition}
\label{lemma:hankel-factorization}
    Let $H_n(t, N)$ be as in \eqref{eq:hankel-det-def}. Set $H_0^{(e, o)}(t, N) \equiv 1$ and 
\begin{align}
    H^{(e)}_{n}(t,N) &:= \det \left[ \dod[i+j]{}{t} \left(\ee^{\frac{N}{8}t^2} D_{-\frac12}\left({N^{\frac12} 2^{-\frac12} t}\right)\right) \right]_{i, j = 0}^{n - 1}, \label{eq:He-def} \\
    H^{(o)}_{n}(t,N) &:= \det \left[ \dod[i+j]{}{t} \left(\ee^{\frac{N}{8}t^2} D_{-\frac32}\left({N^{\frac12} 2^{-\frac12} t}\right)\right) \right]_{i, j = 0}^{n - 1} \label{eq:Ho-def},
\end{align}
where $D_\nu(\diamond)$ is the parabolic cylinder functions (cf. \cite{DLMF}*{Section 12}). Then, 
\begin{align*}
    H_{2n - 1}(t, N) &= \dfrac{\pi^{n} 2^{2n(n - 1)}}{N^{n(2n - 1)}} H^{(e)}_{n}(t,N) H^{(o)}_n(t, N), \\
    H_{2n}(t, N) &= \dfrac{\pi^{n+\frac12} 2^{2n^2 + \frac14}}{N^{2n^2 + n + \frac14}} H^{(e)}_{n+1}(t,N) H^{(o)}_n(t, N).
\end{align*}
\end{proposition}

To prove Proposition \ref{lemma:hankel-factorization}, we will make use of the following algebraic identity.
\begin{lemma}
    Let $\{c_j\}_{j = 0}^\infty$ be a sequence of variables and set $c_{k/2} = 0$ for all $k \in \N \setminus 2\N$. Then
    \begin{align}
    \det [ c_{\frac{i+j}{2}} ]_{i, j = 0}^{2n - 1} &= \det [c_{i+j}]_{i,j = 0}^{n - 1} \cdot \det [c_{i+j+1}]_{i,j = 0}^{n - 1}, \label{eq:det-id-1}\\
    \det [ c_{\frac{i+j}{2}} ]_{i, j = 0}^{2n} &= \det [c_{i+j}]_{i,j = 0}^{n } \cdot \det [c_{i+j+1}]_{i,j = 0}^{n - 1}.\label{eq:det-id-2}
    \end{align}
    \label{prop:hankel}
\end{lemma}
\begin{proof}[Proof of Proposition \ref{prop:hankel}]
    We begin by proving \eqref{eq:det-id-1}. Denote $H := [ c_{\frac{i+j}{2}} ]_{i, j = 0}^{2n - 1}$ and its entries by $H_{i,j} = c_{\frac{i+j}{2}}$.
    Let $P_{2n} \subset S_{2n}$ be the subgroup defined by the property that $j$ and $\sigma(j)$ have the same parity for all $i = 0, 1, ..., 2n-1$. Note that such permutations factor into cycles which act on even numbers and cycles which act on odd numbers. That is, for all $\sigma \in P_{2n}$, there exist $\sigma_e, \sigma_o \in S_{n}$ such that\footnote{One should interpret elements of $S_{2n}$ as acting on the alphabet $0, 1, ..., 2n - 1$, $\sigma_e \in S_n$ as a permutation acting on $0, 2, ..., 2n - 2$, and $\sigma_o\in S_n$ acting on $1, 3, ..., 2n - 1$.} $\sigma = \sigma_e \sigma_o$. With this in mind, the definition of the determinant reads
    \begin{multline*}
        \det H = 
    \sum_{\sigma \in S_{2n}} \text{sgn}(\sigma) \prod_{i = 0}^{2n - 1} H_{i,\sigma(i)} \\= \left( \sum_{\sigma_e \in S_{n}} \text{sgn}(\sigma_e) \prod_{i = 0}^{n - 1} H_{2i,\sigma_e(2i)}\right) \left( \sum_{\sigma_o \in S_{n}} \text{sgn}(\sigma_o) \prod_{i = 0}^{n - 1} H_{2i+1,\sigma_o(2i+1)} \right).
    \end{multline*}
    The terms in the last expression can now readily be interpreted as the determinants in the right hand side of \eqref{eq:det-id-1}. The proof of \eqref{eq:det-id-2} follows using the same argument and noting that the set $\{0, 1, ..., 2n\}$ has one more even number than odd numbers.
\end{proof}

\begin{proof}[Proof of Lemma \ref{lemma:hankel-factorization}]\label{sec:hankel-factorization-proof}
In order to apply Proposition \ref{prop:hankel} to determinants $H_n(t, N)$, we set 
\[
c_{i + j} = \int_\R s^{2(i+j)} \ee^{-NV(s;t)} \dd s = (-1)^{i+j}\dfrac{2^{i+j + \frac14}\pi^{1/2}}{N^{i+j +\frac14}} \dod[{i+j}]{}{t} \left(  \ee^{\frac{N}{8}t^2} D_{-\frac12}\left({N^{\frac12} }{2^{-\frac12}}t \right) \right) 
\]
and note that by \cite{DLMF}*{Eq. 12.8.9} we have the identity 
\begin{equation*}
     \dod{}{t}\left(\ee^{\frac{N}{8}t^2} D_{-\frac12}\left({N^{\frac12} 2^{-\frac12} t}\right)\right) = -{N^{\frac12}}{2^{-\frac32}} \ee^{\frac{N}{8}t^2} D_{-\frac32}\left({N^{1/2} 2^{-1/2} t}\right)
\end{equation*}
In particular, \eqref{eq:even-moments} gives
\begin{equation*}
    c_{i+j+1} = -\frac{2}{N} \dod{c_{i + j}}{t} = (-1)^{i+j}{2^{-\frac12}N^{-\frac12}}\dfrac{2^{i+j + \frac14}\pi^{1/2}}{N^{i+j +\frac14}} \dod[{i+j}]{}{t} \left(  \ee^{\frac{N}{8}t^2} D_{-\frac32}\left({N^{\frac12}}{2^{-\frac12}}t \right) \right).
\end{equation*}
Pulling out common factors from each row/column and recalling \eqref{eq:He-def}, \eqref{eq:Ho-def}, we find
\begin{align}
    \det [c_{i+j}]_{i,j = 0}^{n - 1} &= \dfrac{\pi^{n/2} 2^{n/4}}{N^{n/4}} \left(\prod_{k = 0}^{n - 1} (-1)^k\dfrac{2^k}{N^k}\right)^2 H_n^{(e)}(t,N) = \dfrac{\pi^{n/2} 2^{n(n - 3/4)}}{N^{n(n - 3/4)}}  H_n^{(e)}(t,N), \label{eq:hid-1} \\
    \det [c_{i+j+1}]_{i,j = 0}^{n - 1} &= N^{-\frac{n}{2}} 2^{-\frac{n}{2}} \dfrac{\pi^{n/2} 2^{n(n - 3/4)}}{N^{n(n - 3/4)}}  H_n^{(o)}(t,N) = \dfrac{\pi^{n/2} 2^{n(n - 5/4)}}{N^{n(n - 1/4)}}  H_n^{(o)}(t,N) \label{eq:hid-2}.
\end{align}
Finally, we plug \eqref{eq:hid-1}, \eqref{eq:hid-2} into \eqref{eq:det-id-1}, \eqref{eq:det-id-2} to arrive at the result.
\end{proof}
By simply comparing \eqref{eq:He-def} (resp. \eqref{eq:Ho-def}) with \eqref{eq:tau-n-special} (resp. \eqref{eq:tau-n-special-2}) and using the identity 
\[
    H_n^{(e, o)}(2N^{-\frac12}t, N) = H_n^{(e, o)}(2t, 1),
\]
we find that
\begin{align}
    \tau_{n-1, 0, 0}(x) &= \ee^{-(n-\frac12)x^2} H_n^{(e)}(2N^{-\frac12}x, N) = \ee^{-(n-\frac12)x^2} H_n^{(e)}(2x, 1) \label{eq:tau-n-0-0}, \medskip \\
    \tau_{n-1, -1, 0}(x) &= \ee^{-(n+\frac12)x^2} H_n^{(o)}(2N^{-\frac12}x, N) = \ee^{-(n+\frac12)x^2} H_n^{(o)}(2x, 1)\label{eq:tau-n-1-0}. 
\end{align}

\begin{remark}
    Using Theorem \ref{thm:main2}, one can write the partition function and the (derivative of the) free energy of the quartic one-matrix model mentioned in Section \ref{sec:intro} in terms of $\tau$- and $\sigma$-functions of \eqref{eq:p4-ab}. Since we do not use this connection in this work, we settle for stating the formulas without further comment. Let $Z_{n,N}(t), F_{n,N}(t)$ be the partition function and free energy, respectively, defined in \cite{BGM}*{Equations (1.17) - (1.18)}. Then, 
    \begin{equation*}
        Z_{n,N}(t) = \dfrac{N! \pi^{n} 2^{2n(n - 1)}}{N^{n(2n - 1)}}\ee^{\frac{nN}{2}t^2}\begin{cases}
            (\tau_{n-1, 0, 0} \tau_{n-1, -1, 0})( 2^{-1}N^{\frac12} t), & k = 2n, \medskip \\
           \dfrac{\pi^{\frac12} 2^{2n+ \frac14}}{N^{2n + \frac14}} \ee^{\frac{Nt^2}{4}}  (\tau_{n, 0, 0} \tau_{n-1, -1, 0})( 2^{-1}N^{\frac12} t), & k = 2n+1.
        \end{cases}
    \end{equation*}
    Furthermore,
    \begin{equation*}
        (2n)^2 \dod{}{t} F_{2n,N}(t) = nNt + \dfrac{N^{\frac12}}{2}\left( 
            \sigma_{n-1, 0, 0}(2^{-1}N^{\frac12}t) + \sigma_{n-1, -1, 0}(2^{-1}N^{\frac12}t) \right),
    \end{equation*}
    and 
    \begin{equation*}
        (2n+1)^2 \dod{}{t} F_{2n+1,N}(t) = \left(n -\frac12\right)Nt + \dfrac{N^{\frac12}}{2}\left(\sigma_{n, 0, 0}(2^{-1}N^{\frac12}t) + \sigma_{n-1, -1, 0}(2^{-1}N^{\frac12}t) \right).
    \end{equation*}
    \label{remark:partition-tau}
\end{remark}
Finally, note that combining Lemma \ref{lemma:hankel-factorization}, equations \eqref{eq:tau-n-0-0} - \eqref{eq:tau-n-1-0}, and Remark \ref{remark:tau-zeros} yields the following result.
\begin{proposition} \label{prop:deg}
    For any fixed $t \in \C$, $N \geq 0$, and all $n \in \N$, it holds that 
    \begin{itemize}
        \item if $\deg P_{n} = n$ then $\deg P_{n - 1} = n - 1$ or $\deg P_{n+1} = n+1$, 
        \item if $\deg P_{n} < n$ then exactly one of the following must hold: 
        \begin{enumerate}[(i)]
            \item $\deg P_{n - 1} = \deg P_n = n - 2 $ and $\deg P_{n +1} = n+1$, or
            \item $\deg P_{n+1} =\deg P_n = n-1$ and $\deg P_{n - 1} = n-1$.
        \end{enumerate}
    \end{itemize}
\end{proposition}

\bibliographystyle{amsrefs}
\bibliography{bibliography}

\begin{bibdiv}
\begin{biblist}

\bib{MR1251169}{article}{
      author={Adler, V.~\`E.},
       title={Cutting of polygons},
        date={1993},
        ISSN={0374-1990,2305-2899},
     journal={Funktsional. Anal. i Prilozhen.},
      volume={27},
      number={2},
       pages={79\ndash 82},
         url={https://doi.org/10.1007/BF01085984},
      review={\MR{1251169}},
}

\bib{Adler}{article}{
      author={Adler, V.~\`E.},
       title={Nonlinear chains and {P}ainlev\'{e} equations},
        date={1994},
        ISSN={0167-2789,1872-8022},
     journal={Phys. D},
      volume={73},
      number={4},
       pages={335\ndash 351},
         url={https://doi.org/10.1016/0167-2789(94)90104-X},
      review={\MR{1280883}},
}

\bib{MR3562403}{article}{
      author={Balogh, Ferenc},
      author={Bertola, Marco},
      author={Bothner, Thomas},
       title={Hankel determinant approach to generalized {V}orob'ev-{Y}ablonski
  polynomials and their roots},
        date={2016},
        ISSN={0176-4276,1432-0940},
     journal={Constr. Approx.},
      volume={44},
      number={3},
       pages={417\ndash 453},
         url={https://doi.org/10.1007/s00365-016-9328-4},
      review={\MR{3562403}},
}

\bib{BBDY2}{unpublished}{
      author={Barhoumi, Ahmad},
      author={Bleher, Pavel},
      author={Dea\~{n}o, Alfredo},
      author={Yattselev, Maxim},
       title={On {A}iry solutions of {PII} and complex cubic ensemble of random
  matrices, {II}},
        note={To appear in: Contemp. Math. Arxiv:
  \href{https://arxiv.org/abs/2403.03023}{2403.03023}},
}

\bib{BBDY}{unpublished}{
      author={Barhoumi, Ahmad},
      author={Bleher, Pavel},
      author={Dea\~{n}o, Alfredo},
      author={Yattselev, Maxim},
       title={On {A}iry solutions of {PII} and complex cubic ensemble of random
  matrices, {I}},
        date={2023},
        note={To appear in: Orthogonal Polynomials, Special Functions and
  Applications - Proceedings of the 16th International Symposium, Montreal,
  Canada, In honor to Richard Askey. Arxiv:
  \href{https://arxiv.org/abs/2310.14898}{2310.14898}},
}

\bib{BCH}{article}{
      author={Bassom, Andrew~P.},
      author={Clarkson, Peter~A.},
      author={Hicks, Andrew~C.},
       title={B\"{a}cklund transformations and solution hierarchies for the
  fourth {P}ainlev\'{e} equation},
        date={1995},
        ISSN={0022-2526,1467-9590},
     journal={Stud. Appl. Math.},
      volume={95},
      number={1},
       pages={1\ndash 71},
         url={https://doi.org/10.1002/sapm19959511},
      review={\MR{1342071}},
}

\bib{BT}{article}{
      author={Bertola, M.},
      author={Tovbis, A.},
       title={Asymptotics of orthogonal polynomials with complex varying
  quartic weight: global structure, critical point behavior and the first
  {P}ainlev\'{e} equation},
        date={2015},
        ISSN={0176-4276,1432-0940},
     journal={Constr. Approx.},
      volume={41},
      number={3},
       pages={529\ndash 587},
         url={https://doi.org/10.1007/s00365-015-9288-0},
      review={\MR{3346719}},
}

\bib{MR3431594}{article}{
      author={Bertola, Marco},
      author={Bothner, Thomas},
       title={Zeros of large degree {V}orob'ev-{Y}ablonski polynomials via a
  {H}ankel determinant identity},
        date={2015},
        ISSN={1073-7928,1687-0247},
     journal={Int. Math. Res. Not. IMRN},
      number={19},
       pages={9330\ndash 9399},
         url={https://doi.org/10.1093/imrn/rnu239},
      review={\MR{3431594}},
}

\bib{MR3589917}{article}{
      author={Bertola, Marco},
      author={Tovbis, Alexander},
       title={On asymptotic regimes of orthogonal polynomials with complex
  varying quartic exponential weight},
        date={2016},
        ISSN={1815-0659},
     journal={SIGMA Symmetry Integrability Geom. Methods Appl.},
      volume={12},
       pages={Paper No. 118, 50},
         url={https://doi.org/10.3842/SIGMA.2016.118},
      review={\MR{3589917}},
}

\bib{BIZ}{article}{
      author={Bessis, D.},
      author={Itzykson, C.},
      author={Zuber, J.~B.},
       title={Quantum field theory techniques in graphical enumeration},
        date={1980},
        ISSN={0196-8858,1090-2074},
     journal={Adv. in Appl. Math.},
      volume={1},
      number={2},
       pages={109\ndash 157},
         url={https://doi.org/10.1016/0196-8858(80)90008-1},
      review={\MR{603127}},
}

\bib{BD2}{article}{
      author={Bleher, Pavel},
      author={Dea\~{n}o, Alfredo},
       title={Painlev\'{e} {I} double scaling limit in the cubic random matrix
  model},
        date={2016},
        ISSN={2010-3263,2010-3271},
     journal={Random Matrices Theory Appl.},
      volume={5},
      number={2},
       pages={1650004, 58},
         url={https://doi.org/10.1142/S2010326316500040},
      review={\MR{3493550}},
}

\bib{MR3607591}{article}{
      author={Bleher, Pavel},
      author={Dea\~{n}o, Alfredo},
      author={Yattselev, Maxim},
       title={Topological expansion in the complex cubic log-gas model: one-cut
  case},
        date={2017},
        ISSN={0022-4715},
     journal={J. Stat. Phys.},
      volume={166},
      number={3-4},
       pages={784\ndash 827},
         url={https://doi.org/10.1007/s10955-016-1621-x},
      review={\MR{3607591}},
}

\bib{BGM}{article}{
      author={Bleher, Pavel},
      author={Gharakhloo, Roozbeh},
      author={McLaughlin, Kenneth T-R},
       title={Phase diagram and topological expansion in the complex quartic
  random matrix model},
        date={2024},
        ISSN={0010-3640,1097-0312},
     journal={Comm. Pure Appl. Math.},
      volume={77},
      number={2},
       pages={1405\ndash 1485},
         url={https://doi.org/10.1002/cpa.22164},
      review={\MR{4673885}},
}

\bib{MR1715324}{article}{
      author={Bleher, Pavel},
      author={Its, Alexander},
       title={Semiclassical asymptotics of orthogonal polynomials,
  {R}iemann-{H}ilbert problem, and universality in the matrix model},
        date={1999},
        ISSN={0003-486X},
     journal={Ann. of Math. (2)},
      volume={150},
      number={1},
       pages={185\ndash 266},
         url={https://doi.org/10.2307/121101},
      review={\MR{1715324}},
}

\bib{MR1949138}{article}{
      author={Bleher, Pavel},
      author={Its, Alexander},
       title={Double scaling limit in the random matrix model: the
  {R}iemann-{H}ilbert approach},
        date={2003},
        ISSN={0010-3640,1097-0312},
     journal={Comm. Pure Appl. Math.},
      volume={56},
      number={4},
       pages={433\ndash 516},
         url={https://doi.org/10.1002/cpa.10065},
      review={\MR{1949138}},
}

\bib{BD}{article}{
      author={Bleher, Pavel~M.},
      author={Dea\~{n}o, Alfredo},
       title={Topological expansion in the cubic random matrix model},
        date={2013},
        ISSN={1073-7928,1687-0247},
     journal={Int. Math. Res. Not. IMRN},
      number={12},
       pages={2699\ndash 2755},
         url={https://doi.org/10.1093/imrn/rns126},
      review={\MR{3071662}},
}

\bib{BI}{article}{
      author={Bleher, Pavel~M.},
      author={Its, Alexander~R.},
       title={Asymptotics of the partition function of a random matrix model},
        date={2005},
        ISSN={0373-0956,1777-5310},
     journal={Ann. Inst. Fourier (Grenoble)},
      volume={55},
      number={6},
       pages={1943\ndash 2000},
         url={http://aif.cedram.org/item?id=AIF_2005__55_6_1943_0},
      review={\MR{2187941}},
}

\bib{BV}{article}{
      author={Boelen, Lies},
      author={Van~Assche, Walter},
       title={Discrete {P}ainlev\'{e} equations for recurrence coefficients of
  semiclassical {L}aguerre polynomials},
        date={2010},
        ISSN={0002-9939,1088-6826},
     journal={Proc. Amer. Math. Soc.},
      volume={138},
      number={4},
       pages={1317\ndash 1331},
         url={https://doi.org/10.1090/S0002-9939-09-10152-1},
      review={\MR{2578525}},
}

\bib{MR4046831}{article}{
      author={Bothner, Thomas},
      author={Miller, Peter~D.},
       title={Rational solutions of the {P}ainlev\'{e}-{III} equation: large
  parameter asymptotics},
        date={2020},
        ISSN={0176-4276,1432-0940},
     journal={Constr. Approx.},
      volume={51},
      number={1},
       pages={123\ndash 224},
         url={https://doi.org/10.1007/s00365-019-09463-4},
      review={\MR{4046831}},
}

\bib{MR3879971}{article}{
      author={Bothner, Thomas},
      author={Miller, Peter~D.},
      author={Sheng, Yue},
       title={Rational solutions of the {P}ainlev\'{e}-{III} equation},
        date={2018},
        ISSN={0022-2526,1467-9590},
     journal={Stud. Appl. Math.},
      volume={141},
      number={4},
       pages={626\ndash 679},
         url={https://doi.org/10.1111/sapm.12220},
      review={\MR{3879971}},
}

\bib{BIPZ}{article}{
      author={Br\'{e}zin, E.},
      author={Itzykson, C.},
      author={Parisi, G.},
      author={Zuber, J.~B.},
       title={Planar diagrams},
        date={1978},
        ISSN={0010-3616,1432-0916},
     journal={Comm. Math. Phys.},
      volume={59},
      number={1},
       pages={35\ndash 51},
         url={http://projecteuclid.org/euclid.cmp/1103901558},
      review={\MR{471676}},
}

\bib{MR4153120}{article}{
      author={Buckingham, Robert},
       title={Large-degree asymptotics of rational {P}ainlev\'{e}-{IV}
  functions associated to generalized {H}ermite polynomials},
        date={2020},
        ISSN={1073-7928,1687-0247},
     journal={Int. Math. Res. Not. IMRN},
      number={18},
       pages={5534\ndash 5577},
         url={https://doi.org/10.1093/imrn/rny172},
      review={\MR{4153120}},
}

\bib{MR3265723}{article}{
      author={Buckingham, Robert~J.},
      author={Miller, Peter~D.},
       title={Large-degree asymptotics of rational {P}ainlev\'{e}-{II}
  functions: noncritical behaviour},
        date={2014},
        ISSN={0951-7715,1361-6544},
     journal={Nonlinearity},
      volume={27},
      number={10},
       pages={2489\ndash 2578},
         url={https://doi.org/10.1088/0951-7715/27/10/2489},
      review={\MR{3265723}},
}

\bib{MR3350600}{article}{
      author={Buckingham, Robert~J.},
      author={Miller, Peter~D.},
       title={Large-degree asymptotics of rational {P}ainlev\'{e}-{II}
  functions: critical behaviour},
        date={2015},
        ISSN={0951-7715,1361-6544},
     journal={Nonlinearity},
      volume={28},
      number={6},
       pages={1539\ndash 1596},
         url={https://doi.org/10.1088/0951-7715/28/6/1539},
      review={\MR{3350600}},
}

\bib{BM}{article}{
      author={Buckingham, Robert~J.},
      author={Miller, Peter~D.},
       title={Large-degree asymptotics of rational {P}ainlev\'{e}-{IV}
  solutions by the isomonodromy method},
        date={2022},
        ISSN={0176-4276,1432-0940},
     journal={Constr. Approx.},
      volume={56},
      number={2},
       pages={233\ndash 443},
         url={https://doi.org/10.1007/s00365-022-09586-1},
      review={\MR{4507087}},
}

\bib{MR3370612}{article}{
      author={Claeys, T.},
      author={Grava, T.},
      author={McLaughlin, K. D. T.-R.},
       title={Asymptotics for the partition function in two-cut random matrix
  models},
        date={2015},
        ISSN={0010-3616,1432-0916},
     journal={Comm. Math. Phys.},
      volume={339},
      number={2},
       pages={513\ndash 587},
         url={https://doi.org/10.1007/s00220-015-2412-y},
      review={\MR{3370612}},
}

\bib{MR2538285}{article}{
      author={Clarkson, P.~A.},
       title={Vortices and polynomials},
        date={2009},
        ISSN={0022-2526,1467-9590},
     journal={Stud. Appl. Math.},
      volume={123},
      number={1},
       pages={37\ndash 62},
         url={https://doi.org/10.1111/j.1467-9590.2009.00446.x},
      review={\MR{2538285}},
}

\bib{MR2291140}{article}{
      author={Clarkson, Peter~A.},
       title={Special polynomials associated with rational solutions of the
  {P}ainlev\'{e} equations and applications to soliton equations},
        date={2006},
        ISSN={1617-9447},
     journal={Comput. Methods Funct. Theory},
      volume={6},
      number={2},
       pages={329\ndash 401},
         url={https://doi.org/10.1007/BF03321618},
      review={\MR{2291140}},
}

\bib{MR3529955}{article}{
      author={Clarkson, Peter~A.},
       title={On {A}iry solutions of the second {P}ainlev\'{e} equation},
        date={2016},
        ISSN={0022-2526,1467-9590},
     journal={Stud. Appl. Math.},
      volume={137},
      number={1},
       pages={93\ndash 109},
         url={https://doi.org/10.1111/sapm.12123},
      review={\MR{3529955}},
}

\bib{CJ}{article}{
      author={Clarkson, Peter~A.},
      author={Jordaan, Kerstin},
       title={The relationship between semiclassical {L}aguerre polynomials and
  the fourth {P}ainlev\'{e} equation},
        date={2014},
        ISSN={0176-4276,1432-0940},
     journal={Constr. Approx.},
      volume={39},
      number={1},
       pages={223\ndash 254},
         url={https://doi.org/10.1007/s00365-013-9220-4},
      review={\MR{3144387}},
}

\bib{MR3733254}{article}{
      author={Clarkson, Peter~A.},
      author={Jordaan, Kerstin},
       title={Properties of generalized {F}reud polynomials},
        date={2018},
        ISSN={0021-9045,1096-0430},
     journal={J. Approx. Theory},
      volume={225},
       pages={148\ndash 175},
         url={https://doi.org/10.1016/j.jat.2017.10.001},
      review={\MR{3733254}},
}

\bib{MR3494152}{article}{
      author={Clarkson, Peter~A.},
      author={Jordaan, Kerstin},
      author={Kelil, Abey},
       title={A generalized {F}reud weight},
        date={2016},
        ISSN={0022-2526,1467-9590},
     journal={Stud. Appl. Math.},
      volume={136},
      number={3},
       pages={288\ndash 320},
         url={https://doi.org/10.1111/sapm.12105},
      review={\MR{3494152}},
}

\bib{MR1975781}{article}{
      author={Clarkson, Peter~A.},
      author={Mansfield, Elizabeth~L.},
       title={The second {P}ainlev\'{e} equation, its hierarchy and associated
  special polynomials},
        date={2003},
        ISSN={0951-7715,1361-6544},
     journal={Nonlinearity},
      volume={16},
      number={3},
       pages={R1\ndash R26},
         url={https://doi.org/10.1088/0951-7715/16/3/201},
      review={\MR{1975781}},
}

\bib{MR1083917}{article}{
      author={David, Fran\c{c}ois},
       title={Phases of the large-{$N$} matrix model and nonperturbative
  effects in {$2$}{D} gravity},
        date={1991},
        ISSN={0550-3213},
     journal={Nuclear Phys. B},
      volume={348},
      number={3},
       pages={507\ndash 524},
  url={https://doi-org.proxy.lib.umich.edu/10.1016/0550-3213(91)90202-9},
      review={\MR{1083917}},
}

\bib{MR1711036}{article}{
      author={Deift, P.},
      author={Kriecherbauer, T.},
      author={McLaughlin, K. T-R},
      author={Venakides, S.},
      author={Zhou, X.},
       title={Strong asymptotics of orthogonal polynomials with respect to
  exponential weights},
        date={1999},
        ISSN={0010-3640},
     journal={Comm. Pure Appl. Math.},
      volume={52},
      number={12},
       pages={1491\ndash 1552},
  url={https://doi.org/10.1002/(SICI)1097-0312(199912)52:12<1491::AID-CPA2>3.3.CO;2-R},
      review={\MR{1711036}},
}

\bib{MR1677884}{book}{
      author={Deift, Percy},
       title={Orthogonal polynomials and random matrices: a {R}iemann-{H}ilbert
  approach},
      series={Courant Lecture Notes in Mathematics},
   publisher={New York University, Courant Institute of Mathematical Sciences,
  New York; American Mathematical Society, Providence, RI},
        date={1999},
      volume={3},
        ISBN={0-9658703-2-4},
      review={\MR{1677884}},
}

\bib{MR1702716}{article}{
      author={Deift, Percy},
      author={Kriecherbauer, Thomas},
      author={McLaughlin, Kenneth T.-R.},
      author={Venakides, Stephanos},
      author={Zhou, Xin},
       title={Uniform asymptotics for polynomials orthogonal with respect to
  varying exponential weights and applications to universality questions in
  random matrix theory},
        date={1999},
        ISSN={0010-3640},
     journal={Comm. Pure Appl. Math.},
      volume={52},
      number={11},
       pages={1335\ndash 1425},
  url={https://doi.org/10.1002/(SICI)1097-0312(199911)52:11<1335::AID-CPA1>3.0.CO;2-1},
      review={\MR{1702716}},
}

\bib{MR1858269}{inproceedings}{
      author={Deift, Percy},
      author={Kriecherbauer, Thomas},
      author={McLaughlin, Kenneth T.-R.},
      author={Venakides, Stephanos},
      author={Zhou, Xin},
       title={A {R}iemann-{H}ilbert approach to asymptotic questions for
  orthogonal polynomials},
        date={2001},
   booktitle={Proceedings of the {F}ifth {I}nternational {S}ymposium on
  {O}rthogonal {P}olynomials, {S}pecial {F}unctions and their {A}pplications
  ({P}atras, 1999)},
      volume={133},
       pages={47\ndash 63},
         url={https://doi.org/10.1016/S0377-0427(00)00634-8},
      review={\MR{1858269}},
}

\bib{MR1207209}{article}{
      author={Deift, Percy},
      author={Zhou, Xin},
       title={A steepest descent method for oscillatory {R}iemann-{H}ilbert
  problems. {A}symptotics for the {MK}d{V} equation},
        date={1993},
        ISSN={0003-486X},
     journal={Ann. of Math. (2)},
      volume={137},
      number={2},
       pages={295\ndash 368},
         url={https://doi.org/10.2307/2946540},
      review={\MR{1207209}},
}

\bib{DLMF}{misc}{
       title={{\it NIST Digital Library of Mathematical Functions}},
         how={\url{https://dlmf.nist.gov/}, Release 1.1.11 of 2023-09-15},
         url={https://dlmf.nist.gov/},
        note={F.~W.~J. Olver, A.~B. {Olde Daalhuis}, D.~W. Lozier, B.~I.
  Schneider, R.~F. Boisvert, C.~W. Clark, B.~R. Miller, B.~V. Saunders, H.~S.
  Cohl, and M.~A. McClain, eds.},
}

\bib{DK}{article}{
      author={Duits, M.},
      author={Kuijlaars, A. B.~J.},
       title={Painlev\'{e} {I} asymptotics for orthogonal polynomials with
  respect to a varying quartic weight},
        date={2006},
        ISSN={0951-7715,1361-6544},
     journal={Nonlinearity},
      volume={19},
      number={10},
       pages={2211\ndash 2245},
         url={https://doi.org/10.1088/0951-7715/19/10/001},
      review={\MR{2260263}},
}

\bib{MR1953782}{article}{
      author={Ercolani, Nicolas~M.},
      author={McLaughlin, Kenneth D. T.-R.},
       title={Asymptotics of the partition function for random matrices via
  {R}iemann-{H}ilbert techniques and applications to graphical enumeration},
        date={2003},
        ISSN={1073-7928},
     journal={Int. Math. Res. Not.},
      number={14},
       pages={755\ndash 820},
         url={https://doi-org.proxy.lib.umich.edu/10.1155/S1073792803211089},
      review={\MR{1953782}},
}

\bib{MR583745}{book}{
      author={Farkas, Hershel~M.},
      author={Kra, Irwin},
       title={Riemann surfaces},
      series={Graduate Texts in Mathematics},
   publisher={Springer-Verlag, New York-Berlin},
        date={1980},
      volume={71},
        ISBN={0-387-90465-4},
      review={\MR{583745}},
}

\bib{FVZ}{article}{
      author={Filipuk, Galina},
      author={Van~Assche, Walter},
      author={Zhang, Lun},
       title={The recurrence coefficients of semi-classical {L}aguerre
  polynomials and the fourth {P}ainlev\'{e} equation},
        date={2012},
        ISSN={1751-8113,1751-8121},
     journal={J. Phys. A},
      volume={45},
      number={20},
       pages={205201, 13},
         url={https://doi.org/10.1088/1751-8113/45/20/205201},
      review={\MR{2924846}},
}

\bib{FIK}{article}{
      author={Fokas, A.~S.},
      author={Its, A.~R.},
      author={Kitaev, A.~V.},
       title={Discrete {P}ainlev\'{e} equations and their appearance in quantum
  gravity},
        date={1991},
        ISSN={0010-3616,1432-0916},
     journal={Comm. Math. Phys.},
      volume={142},
      number={2},
       pages={313\ndash 344},
         url={http://projecteuclid.org/euclid.cmp/1104248588},
      review={\MR{1137067}},
}

\bib{FIK2}{article}{
      author={Fokas, Athanassios~S.},
      author={Its, Alexander~R.},
      author={Kitaev, Alexander~V.},
       title={The isomonodromy approach to matrix models in {2}{D} quantum
  gravity},
        date={1992},
        ISSN={0010-3616},
     journal={Comm. Math. Phys.},
      volume={147},
      number={2},
       pages={395\ndash 430},
  url={http://projecteuclid.org.proxy.lib.umich.edu/euclid.cmp/1104250643},
      review={\MR{1174420}},
}

\bib{MR2804960}{article}{
      author={Fornberg, Bengt},
      author={Weideman, J. A.~C.},
       title={A numerical methodology for the {P}ainlev\'{e} equations},
        date={2011},
        ISSN={0021-9991,1090-2716},
     journal={J. Comput. Phys.},
      volume={230},
      number={15},
       pages={5957\ndash 5973},
         url={https://doi.org/10.1016/j.jcp.2011.04.007},
      review={\MR{2804960}},
}

\bib{FW}{article}{
      author={Forrester, P.~J.},
      author={Witte, N.~S.},
       title={Application of the {$\tau$}-function theory of {P}ainlev\'{e}
  equations to random matrices: {PIV}, {PII} and the {GUE}},
        date={2001},
        ISSN={0010-3616,1432-0916},
     journal={Comm. Math. Phys.},
      volume={219},
      number={2},
       pages={357\ndash 398},
         url={https://doi.org/10.1007/s002200100422},
      review={\MR{1833807}},
}

\bib{MR0419895}{article}{
      author={Freud, G\'{e}za},
       title={On the coefficients in the recursion formulae of orthogonal
  polynomials},
        date={1976},
        ISSN={0035-8975},
     journal={Proc. Roy. Irish Acad. Sect. A},
      volume={76},
      number={1},
       pages={1\ndash 6},
      review={\MR{419895}},
}

\bib{MR922628}{article}{
      author={Gonchar, Andrei~A.},
      author={Rakhmanov, Evguenii~A.},
       title={Equilibrium distributions and the rate of rational approximation
  of analytic functions},
        date={1987},
        ISSN={0368-8666},
     journal={Mat. Sb. (N.S.)},
      volume={134(176)},
      number={3},
       pages={306\ndash 352, 447},
         url={https://doi.org/10.1070/SM1989v062n02ABEH003242},
      review={\MR{922628}},
}

\bib{MR0896355}{article}{
      author={Gromak, V.~I.},
       title={On the theory of the fourth {P}ainlev\'{e} equation},
        date={1987},
        ISSN={0374-0641},
     journal={Differentsial\cprime nye Uravneniya},
      volume={23},
      number={5},
       pages={760\ndash 768, 914},
      review={\MR{896355}},
}

\bib{GLS}{book}{
      author={Gromak, Valerii~I.},
      author={Laine, Ilpo},
      author={Shimomura, Shun},
       title={Painlev\'{e} differential equations in the complex plane},
      series={De Gruyter Studies in Mathematics},
   publisher={Walter de Gruyter \& Co., Berlin},
        date={2002},
      volume={28},
        ISBN={3-11-017379-4},
         url={https://doi.org/10.1515/9783110198096},
      review={\MR{1960811}},
}

\bib{MR0096806}{book}{
      author={Jenkins, James~A.},
       title={Univalent functions and conformal mapping},
      series={Reihe: Moderne Funktionentheorie},
   publisher={Springer-Verlag, Berlin-G\"{o}ttingen-Heidelberg},
        date={1958},
      review={\MR{0096806}},
}

\bib{MR3306308}{article}{
      author={Kuijlaars, Arno B.~J.},
      author={Silva, Guilherme L.~F.},
       title={S-curves in polynomial external fields},
        date={2015},
        ISSN={0021-9045},
     journal={J. Approx. Theory},
      volume={191},
       pages={1\ndash 37},
         url={https://doi.org/10.1016/j.jat.2014.04.002},
      review={\MR{3306308}},
}

\bib{MR0217352}{article}{
      author={Luka\v{s}evi\v{c}, N.~A.},
       title={The theory of {P}ainlev\'{e}'s fourth equation},
        date={1967},
        ISSN={0374-0641},
     journal={Differencial\cprime nye Uravnenija},
      volume={3},
       pages={771\ndash 780},
      review={\MR{217352}},
}

\bib{Magnus}{inproceedings}{
      author={Magnus, Alphonse~P.},
       title={Painlev\'{e}-type differential equations for the recurrence
  coefficients of semi-classical orthogonal polynomials},
        date={1995},
   booktitle={Proceedings of the {F}ourth {I}nternational {S}ymposium on
  {O}rthogonal {P}olynomials and their {A}pplications ({E}vian-{L}es-{B}ains,
  1992)},
      volume={57},
       pages={215\ndash 237},
         url={https://doi.org/10.1016/0377-0427(93)E0247-J},
      review={\MR{1340938}},
}

\bib{Magnus2}{incollection}{
      author={Magnus, Alphonse~P.},
       title={Freud's equations for orthogonal polynomials as discrete
  {P}ainlev\'{e} equations},
        date={1999},
   booktitle={Symmetries and integrability of difference equations
  ({C}anterbury, 1996)},
      series={London Math. Soc. Lecture Note Ser.},
      volume={255},
   publisher={Cambridge Univ. Press, Cambridge},
       pages={228\ndash 243},
         url={https://doi.org/10.1017/CBO9780511569432.019},
      review={\MR{1705232}},
}

\bib{MR0862231}{article}{
      author={Nevai, Paul},
       title={G\'{e}za {F}reud, orthogonal polynomials and {C}hristoffel
  functions. {A} case study},
        date={1986},
        ISSN={0021-9045,1096-0430},
     journal={J. Approx. Theory},
      volume={48},
      number={1},
       pages={3\ndash 167},
         url={https://doi.org/10.1016/0021-9045(86)90016-X},
      review={\MR{862231}},
}

\bib{N}{book}{
      author={Noumi, Masatoshi},
       title={Painlev\'{e} equations through symmetry},
      series={Translations of Mathematical Monographs},
   publisher={American Mathematical Society, Providence, RI},
        date={2004},
      volume={223},
        ISBN={0-8218-3221-2},
         url={https://doi.org/10.1090/mmono/223},
        note={Translated from the 2000 Japanese original by the author},
      review={\MR{2044201}},
}

\bib{NY}{article}{
      author={Noumi, Masatoshi},
      author={Yamada, Yasuhiko},
       title={Symmetries in the fourth {P}ainlev\'{e} equation and {O}kamoto
  polynomials},
        date={1999},
        ISSN={0027-7630,2152-6842},
     journal={Nagoya Math. J.},
      volume={153},
       pages={53\ndash 86},
         url={https://doi.org/10.1017/S0027763000006899},
      review={\MR{1684551}},
}

\bib{MR0581468}{article}{
      author={Okamoto, Kazuo},
       title={Polynomial {H}amiltonians associated with {P}ainlev\'{e}
  equations. {I}},
        date={1980},
        ISSN={0386-2194},
     journal={Proc. Japan Acad. Ser. A Math. Sci.},
      volume={56},
      number={6},
       pages={264\ndash 268},
         url={http://projecteuclid.org/euclid.pja/1195516808},
      review={\MR{581468}},
}

\bib{O}{article}{
      author={Okamoto, Kazuo},
       title={Studies on the {P}ainlev\'{e} equations. {III}. {S}econd and
  fourth {P}ainlev\'{e} equations, {$P_{{\rm II}}$} and {$P_{{\rm IV}}$}},
        date={1986},
        ISSN={0025-5831,1432-1807},
     journal={Math. Ann.},
      volume={275},
      number={2},
       pages={221\ndash 255},
         url={https://doi.org/10.1007/BF01458459},
      review={\MR{854008}},
}

\bib{MR1485778}{book}{
      author={Saff, Edward~B.},
      author={Totik, Vilmos},
       title={Logarithmic potentials with external fields},
      series={Grundlehren der Mathematischen Wissenschaften [Fundamental
  Principles of Mathematical Sciences]},
   publisher={Springer-Verlag, Berlin},
        date={1997},
      volume={316},
        ISBN={3-540-57078-0},
         url={https://doi.org/10.1007/978-3-662-03329-6},
        note={Appendix B by Thomas Bloom},
      review={\MR{1485778}},
}

\bib{Shohat}{article}{
      author={Shohat, J.},
       title={A differential equation for orthogonal polynomials},
        date={1939},
        ISSN={0012-7094,1547-7398},
     journal={Duke Math. J.},
      volume={5},
      number={2},
       pages={401\ndash 417},
         url={https://doi.org/10.1215/S0012-7094-39-00534-X},
      review={\MR{1546133}},
}

\bib{stahl-domains}{article}{
      author={Stahl, Herbert},
       title={Extremal domains associated with an analytic function. {I},
  {II}},
        date={1985},
        ISSN={0278-1077},
     journal={Complex Variables Theory Appl.},
      volume={4},
      number={4},
       pages={311\ndash 324, 325\ndash 338},
         url={https://doi.org/10.1080/17476938508814117},
      review={\MR{858916}},
}

\bib{stahl-structure}{article}{
      author={Stahl, Herbert},
       title={The structure of extremal domains associated with an analytic
  function},
        date={1985},
        ISSN={0278-1077},
     journal={Complex Variables Theory Appl.},
      volume={4},
      number={4},
       pages={339\ndash 354},
         url={https://doi.org/10.1080/17476938508814119},
      review={\MR{858917}},
}

\bib{stahl-complexortho}{article}{
      author={Stahl, Herbert},
       title={Orthogonal polynomials with complex-valued weight function. {I},
  {II}},
        date={1986},
        ISSN={0176-4276},
     journal={Constr. Approx.},
      volume={2},
      number={3},
       pages={225\ndash 240, 241\ndash 251},
         url={https://doi.org/10.1007/BF01893429},
      review={\MR{891973}},
}

\bib{MR743423}{book}{
      author={Strebel, Kurt},
       title={Quadratic differentials},
      series={Ergebnisse der Mathematik und ihrer Grenzgebiete (3) [Results in
  Mathematics and Related Areas (3)]},
   publisher={Springer-Verlag, Berlin},
        date={1984},
      volume={5},
        ISBN={3-540-13035-7},
         url={https://doi.org/10.1007/978-3-662-02414-0},
      review={\MR{743423}},
}

\bib{MR3729446}{book}{
      author={Van~Assche, Walter},
       title={Orthogonal polynomials and {P}ainlev\'{e} equations},
      series={Australian Mathematical Society Lecture Series},
   publisher={Cambridge University Press, Cambridge},
        date={2018},
      volume={27},
        ISBN={978-1-108-44194-0},
      review={\MR{3729446}},
}

\bib{MR1251164}{article}{
      author={Veselov, A.~P.},
      author={Shabat, A.~B.},
       title={A dressing chain and the spectral theory of the {S}chr\"{o}dinger
  operator},
        date={1993},
        ISSN={0374-1990,2305-2899},
     journal={Funktsional. Anal. i Prilozhen.},
      volume={27},
      number={2},
       pages={1\ndash 21, 96},
         url={https://doi.org/10.1007/BF01085979},
      review={\MR{1251164}},
}

\bib{MR3795283}{article}{
      author={Wen, Zhi-Tao},
      author={Wong, Roderick},
      author={Xu, Shuai-Xia},
       title={Global asymptotics of orthogonal polynomials associated with a
  generalized {F}reud weight},
        date={2018},
        ISSN={0252-9599,1860-6261},
     journal={Chinese Ann. Math. Ser. B},
      volume={39},
      number={3},
       pages={553\ndash 596},
         url={https://doi.org/10.1007/s11401-018-0082-8},
      review={\MR{3795283}},
}

\bib{MR2559966}{article}{
      author={Wong, R.},
      author={Zhang, L.},
       title={Global asymptotics for polynomials orthogonal with exponential
  quartic weight},
        date={2009},
        ISSN={0921-7134,1875-8576},
     journal={Asymptot. Anal.},
      volume={64},
      number={3-4},
       pages={125\ndash 154},
      review={\MR{2559966}},
}

\end{biblist}
\end{bibdiv}

\end{document}